\documentclass[11pt]{amsart}
\usepackage{amsmath}
\usepackage{amssymb}
\usepackage{amsthm}
\usepackage{latexsym}
\usepackage{graphicx}
\usepackage{hyperref}
\usepackage{enumerate}
\usepackage[all]{xy}

\newtheorem{thm}{Theorem}[section]
\newtheorem{cor}[thm]{Corollary}
\newtheorem{lem}[thm]{Lemma}
\newtheorem{prop}[thm]{Proposition}
\newtheorem{thmintro}{Theorem}

\theoremstyle{definition}
\newtheorem{defn}[thm]{Definition}

\newtheorem{construct}[thm]{Construction}

\newcommand{\N}{\mathbb N}
\newcommand{\Z}{\mathbb Z}
\newcommand{\Q}{\mathbb Q}
\newcommand{\R}{\mathbb R}
\newcommand{\C}{\mathbb C}
\newcommand{\mf}{\mathfrak}
\newcommand{\mc}{\mathcal}
\newcommand{\mb}{\mathbf}
\newcommand{\mh}{\mathbb}
\def\Irr{{\rm Irr}}
\newcommand{\mr}{\mathrm}
\newcommand{\enuma}[1]{\begin{enumerate}[\textup{(}a\textup{)}] {#1} \end{enumerate}}

\newcommand{\cusp}{\mathrm{cusp}}
\newcommand{\nr}{\mathrm{nr}}
\newcommand{\unr}{\mathrm{unr}}

\newcommand{\Rep}{\mathrm{Rep}}
\newcommand{\End}{\mathrm{End}}

\newcommand{\Mod}{\mathrm{Mod}}
\newcommand{\Hom}{\mathrm{Hom}}

\newcommand{\isom}{\xrightarrow{\;\sim\;}}

\begin{document}

\title{Hochschild homology of reductive $p$-adic groups}
\date{\today}
\maketitle
\begin{center}
{\Large Maarten Solleveld} \\[1mm]
IMAPP, Radboud Universiteit Nijmegen\\
Heyendaalseweg 135, 6525AJ Nijmegen, the Netherlands \\
email: m.solleveld@science.ru.nl 
\end{center}
\vspace{2mm}

\begin{abstract}
Consider a reductive $p$-adic group $G$, its (complex-valued) Hecke algebra $\mc H (G)$ and the 
Harish-Chandra--Schwartz algebra $\mc S (G)$. We compute the Hochschild homology groups of 
$\mc H (G)$ and of $\mc S (G)$, and we describe the outcomes in several ways.

Our main tools are algebraic families of smooth $G$-representations. With those we construct
maps from $HH_n (\mc H (G))$ and $HH_n (\mc S (G))$ to modules of differential $n$-forms on
affine varieties. For $n = 0$ this provides a description of the cocentres of these algebras 
in terms of nice linear functions on the Grothendieck group of finite length (tempered) 
$G$-representations.

It is known from \cite{SolEnd} that every Bernstein ideal $\mc H (G)^{\mf s}$ of $\mc H (G)$
is closely related to a crossed product algebra of the form $\mc O (T) \rtimes W$. Here
$\mc O (T)$ denotes the regular functions on the variety $T$ of unramified characters of a
Levi subgroup $L$ of $G$, and $W$ is a finite group acting on $T$. We make this relation even
stronger by establishing an isomorphism between $HH_* (\mc H (G)^{\mf s})$ and 
$HH_* (\mc O (T) \rtimes W)$, although we have to say that in some cases it is necessary to 
twist $\C [W]$ by a 2-cocycle.

Similarly we prove that the Hochschild homology of the two-sided ideal $\mc S (G)^{\mf s}$ of
$\mc S (G)$ is isomorphic to $HH_* (C^\infty (T_u) \rtimes W)$, where $T_u$ denotes the Lie
group of unitary unramified characters of $L$. In these pictures of $HH_* (\mc H (G))$ and 
$HH_* (\mc S (G))$ we also show how the Bernstein centre of $\mc H (G)$ acts.

Finally, we derive similar expressions for the (periodic) cyclic homology groups of $\mc H (G)$ 
and of $\mc S (G)$ and we relate that to topological K-theory.
\end{abstract}
\vspace{2mm}

\tableofcontents

\section*{Introduction}

The Hochschild homology of an algebra $A$ (by default over $\C$) is a fairly subtle invariant.
For finitely generated commutative algebras it gives more or less the differential forms on 
the underlying affine variety -- exactly that when the algebra is smooth, and otherwise 
$HH_* (A)$ detects some singularities of the variety. For general algebras Hochschild homology
is related to noncommutative versions of differential forms \cite[Chapter 1]{Lod}.

The vector space $HH_0 (A)$ is particularly interesting, because it equals the cocentre $A / [A,A]$ 
and via the trace pairing contains information about the set of irreducible representations of $A$. 
The higher Hochschild homology groups $HH_n (A)$ also have their uses: they say something about 
higher extensions of $A$-modules (via Hochschild cohomology) and they interact with further 
invariants of algebras like (periodic) cyclic homology. When $A$ is the group algebra of a 
discrete group $\Gamma$, $HH_* (A)$ computes the group cohomology of the groups 
$Z_\Gamma (\gamma)$ with $\gamma \in \Gamma$ \cite{Bur}.\\

\textbf{Categories of representations of reductive $p$-adic groups} \\
Let $G$ be a reductive group over a non-archimedean local field, connected as algebraic group.
We aim to determine the Hochschild homology of $G$, by which we mean the Hochschild homology
of a suitable group algebra of $G$. The most natural choice is the Hecke algebra $\mc H (G)$,
because the category Mod$(\mc H (G))$ of nondegenerate left $\mc H (G)$-modules is naturally 
equivalent to the category $\Rep (G)$ of complex smooth $G$-representations. By definition
\[
HH_n (\mc H (G)) = \mr{Tor}_n^{\mc H (G) \otimes \mc H (G)^{op}} (\mc H (G), \mc H (G)),
\]
so $HH_n (\mc H (G))$ depends only on the category of $\mc H (G)$-bimodules, which is
equivalent to the category of smooth $G \times G^{op}$-representations.

Alternatively we have the Harish-Chandra--Schwartz algebra $\mc S (G)$, whose category of
nondegenerate left modules equals the category $\Rep^t (G)$ of tempered $G$-representations,
according to the conventions from the appendix of \cite{SSZ}.
We consider $\mc S (G)$ as a bornological algebra and use the complete bornological tensor 
product $\hat \otimes$ \cite{Mey2}. In that setting
\[
HH_n (\mc S (G)) = \mr{Tor}_n^{\mc S (G) \hat{\otimes} \mc S (G)^{op}} (\mc S (G), \mc S (G)) ,
\]
which depends only on the category of bornological $\mc S (G)$-bimodules. 

On the other hand, the full group $C^*$-algebra $C^* (G)$ or its reduced version $C_r^* (G)$ 
would not be suitable here, because
\[
HH_n (C^* (G)) = HH_n (C_r^* (G)) = 0 \qquad \text{for } n > 0.
\]
We approach our main goal with representation theory. We start with the Bernstein decomposition
\[
\Rep (G) = \prod\nolimits_{\mf s \in \mf B (G)} \Rep (G)^{\mf s},
\]
which induces decompositions in two-sided ideals 
\begin{equation}\label{eq:3}
\mc H (G) =  \bigoplus\nolimits_{\mf s \in \mf B (G)} \mc H (G)^{\mf s} \quad \text{and} \quad
\mc S (G) =  \bigoplus\nolimits_{\mf s \in \mf B (G)} \mc S (G)^{\mf s} .
\end{equation}
Hochschild homology decomposes accordingly, so we may focus on the algebras $\mc H (G)^{\mf s}$
and $\mc S (G)^{\mf s}$. We will make ample use of the Morita equivalence between $\mc H (G)^{\mf s}$ 
and the opposite algebra of $\End_G (\Pi_{\mf s})$, where $\Pi_{\mf s}$ is a progenerator of 
$\Rep (G)^{\mf s}$ \cite[Theorem 1.8.2.1]{Roc}. In \cite{SolEnd} we made a detailed analysis of 
$\End_G (\Pi_{\mf s})^{op}$, which links it to algebras whose Hochschild homology groups have 
already been determined. 

Let $\sigma$ be a supercuspidal representation of a Levi subgroup $L$ of $G$, representing 
$\mf s = [L,\sigma]$. Let $L^1 \subset L$ the group generated by all compact subgroups of $L$.
Then the compactly induced representation $\mr{ind}_{L^1}^L (\sigma)$ is a progenerator of
$\Rep (L)^{[L,\sigma]}$. Let $P_L$ be a parabolic subgroup of $G$ with Levi factor $L$. 
As shown first in \cite[\S III.4.1]{BeRu}, it follows from Bernstein's second adjointness 
theorem that the parabolically induced representation 
\begin{equation}\label{eq:4}
\Pi_{\mf s} = I_{P_L}^G \mr{ind}_{L^1}^L (\sigma)
\end{equation}
is a progenerator of $\Rep (G)^{\mf s}$. This progenerator was especially convenient for the 
computations in \cite{SolEnd}, and therefore we use it throughout this paper.

To $\mf s$ one can associate a finite group $W(L,\mf s)$ of transformations
of the complex torus of unramified characters $X_\nr (L)$, satisfying 
\[
Z (\Rep (G)^{\mf s}) \cong Z (\End_G (\Pi_{\mf s})^{op}) \cong \mc O (X_\nr (L))^{W(L,\mf s)} .
\]
There exists a 2-cocycle $\natural_{\mf s}$ of $W(L,\mf s)$ such that the twisted group algebra\\
$\C [W(L,\mf s), \natural_{\mf s}]$ acts ``almost" on the objects of $\Rep (G)^{\mf s}$ by 
intertwining operators. Here ``almost" means that these intertwining operators depend
rationally on $\chi \in X_\nr (L)$, and they can have poles. In this setting 
\cite[Theorem A]{SolEnd} provides an isomorphism of $\mc O (X_\nr (L))^{W(L,\mf s)}$-algebras
\begin{equation}\label{eq:1}
\C (X_\nr (L))^{W(L,\mf s)} \underset{\mc O (X_\nr (L))^{W(L,\mf s)}}{\otimes} 
\End_G (\Pi_{\mf s})^{op} \; \cong \; \C (X_\nr (L)) \rtimes \C [W(L,\mf s),\natural_{\mf s}] .
\end{equation}
This isomorphism is canonical on $\C (X_\nr (L))$ and on a Weyl group contained in
$W(L,\mf s)$. However, for the remaining elements of $W(L,\mf s)$ the images on the left hand side 
of \eqref{eq:1} are in general only canonical up to scalars.

Although $\mc H (G)^{\mf s}$ and $\mc O (X_\nr (L)) \rtimes \C [W(L,\mf s),\natural_{\mf s}]$
are usually not Morita equivalent, it has turned out that these algebras nevertheless share
many properties. By \eqref{eq:3} and the definition of temperedness for $G$-representations, 
the category $\Rep^t (G)^{\mf s}$ of tempered representations in $\Rep (G)^{\mf s} \cong 
\Mod (\mc H (G)^{\mf s})$ is $\Mod (\mc S (G)^{\mf s})$. This
subcategory is stable under tensoring with elements of $X_\unr (G)$, the group of unitary
unramified characters of $G$. Like above, from \cite{SolComp} one can expect strong similarities
between $\mc S (G)^{\mf s}$ and $C^\infty (X_\unr (L)) \rtimes \C [W(L,\mf s),\natural_{\mf s}]$.

Let $R(A)$ denote the Grothendieck group of the category of finite length $A$-representations.
We abbreviate $R(G)^{\mf s} = R (\mc H (G)^{\mf s})$ and $R^t (G)^{\mf s} = R(\mc S (G)^{\mf s})$.

\begin{thmintro}\label{thm:A}
\textup{(see Theorem \ref{thm:4.1})} \\
There exists a group isomorphism
\[
\zeta^\vee : R (G)^{\mf s} \to 
R \big( \mc O (X_\nr (L)) \rtimes \C [W(L,\mf s),\natural_{\mf s}] \big)
\]
which restricts to a bijection
\[
\zeta^\vee_t : R^t (G)^{\mf s} \to 
R \big( C^\infty (X_\unr (L)) \rtimes \C [W(L,\mf s),\natural_{\mf s}] \big) .
\]
These bijections are compatible with parabolic induction and with twists by unramified characters.
When an isomorphism \eqref{eq:1} has been fixed, $\zeta^\vee$ and $\zeta_t^\vee$ are canonical.
\end{thmintro}

\textbf{Hochschild homology and twisted extended quotients} \\
In a sense that we will make precise later, Theorem \ref{thm:A} induces isomorphisms on 
Hochschild homology.

\begin{thmintro}\label{thm:B}
\textup{(see Theorems \ref{thm:4.2} and \ref{thm:6.11})}\\
There exist $\C$-linear bijections (canonical when \eqref{eq:1} has been fixed)
\[
\begin{array}{cccc}
HH_n (\zeta^\vee ) : & HH_n \big( \mc O (X_\nr (L)) \rtimes \C [W(L,\mf s),\natural_{\mf s}] \big) &
\to & HH_n (\mc H (G)^{\mf s}) , \\
HH_n (\zeta_t^\vee ) : & HH_n \big( C^\infty (X_\unr (L)) \rtimes 
\C [W(L,\mf s),\natural_{\mf s}] \big) & \to & HH_n (\mc S (G)^{\mf s}) .
\end{array}
\]
\end{thmintro}

The Hochschild homology of twisted crossed product algebras like
\[
\mc O (X_\nr (L)) \rtimes \C [W(L,\mf s),\natural_{\mf s}] \quad \text{and} \quad
C^\infty (X_\unr (L)) \rtimes \C [W(L,\mf s),\natural_{\mf s}] 
\]
was determined in \cite[\S 1]{SolTwist}. It can be interpreted in terms of the twisted extended
quotient
\[
(X_\nr (L) /\!/ W(L,\mf s) )_{\natural_{\mf s}} := \big\{ (\chi,\pi_\chi) : \chi \in X_\nr (L), 
\pi_\chi \in \Irr (\C [W(L,\mf s),\natural_{\mf s}]) \big\} / W(L,\mf s) .
\]
Loosely speaking, $HH_n \big( \mc O (X_\nr (L)) \rtimes \C [W(L,\mf s),\natural_{\mf s}] \big)$ 
is the $\mc O (X_\nr (L))^{W(L,\mf s)}$-module of differential $n$-forms on 
$(X_\nr (L) /\!/ W(L,\mf s) )_{\natural_{\mf s}}$, and similarly for
\[
HH_n \big( C^\infty (X_\unr (L)) \rtimes \C [W(L,\mf s),\natural_{\mf s}] \big) 
\quad \text{and} \quad (X_\unr (L) /\!/ W(L,\mf s) )_{\natural_{\mf s}}.
\]
Let $\Irr_\cusp (L)$ be the set of supercuspidal irreducible $L$-representations (up to 
isomorphism), so that $\Irr (L)^{\mf s}$ is one $X_\nr (L)$-orbit in $\Irr_\cusp (L)$. 
The group $W(G,L) = N_G (L) / L$ acts naturally on $\Irr (L)$. Let $W(G,L)_{\mf s}$ be the 
stabilizer of $\Irr (L)^{\mf s}$ in $W(G,L)$. The covering map
\[
X_\nr (L) \to \Irr (L)^{\mf s} : \chi \mapsto \sigma \otimes \chi 
\]
induces a bijection 
\[
\big( X_\nr (L) /\!/ W(L,\mf s) \big)_{\natural_{\mf s}} \longrightarrow
\big( \Irr (L)^{\mf s} /\!/ W(G,L)_{\mf s} \big)_{\natural_{\mf s}} .
\]
Combining such maps we obtain a bijection 
\[
\bigsqcup\nolimits_{\mf s \in \Irr_\cusp (L) / X_\nr (L) \rtimes W(G,L)}
\big( X_\nr (L) /\!/ W(L,\mf s) \big)_{\natural_{\mf s}} \longrightarrow
\big( \Irr_\cusp (L) /\!/ W(G,L) \big)_{\natural_L} ,
\]
where $\natural_L$ is a shorthand for the data from the various 2-cocycles $\natural_{\mf s}$.
This twisted extended quotient (in the more general sense from  \cite[\S 2.1]{ABPS2}) is related
to the ideal 
\[
\bigoplus\nolimits_{\mf s \in \Irr_\cusp (L) / X_\nr (L) \rtimes W(G,L)} \mc H (G)^{\mf s} 
\quad \text{ of } \quad \mc H (G) .
\]
Let $\mf{Lev}(G)$ be a set of representatives for the conjugacy classes of Levi subgroups of $G$.
Theorem \ref{thm:B} and the above entail that $HH_n (\mc H (G))$ can be regarded as
the $Z(\Rep (G))$-module of algebraic differential $n$-forms on
\[
\bigsqcup\nolimits_{L \in \mf{Lev}(G)} \big( \Irr_\cusp (L) /\!/ W(G,L) \big)_{\natural_L} .
\] 
Similarly we may interpret $HH_n (\mc S (G))$ as the $Z(\Rep^t (G))$-module of smooth differential
$n$-forms on 
\[
\bigsqcup\nolimits_{L \in \mf{Lev}(G)} \big( \Irr^t_\cusp (L) /\!/ W(G,L) \big)_{\natural_L} .
\] 
Notice that these descriptions mainly involve data that are much easier than $\Rep (G)^{\mf s}$,
only the 2-cocycles $\natural_{\mf s}$ contain information about non-supercuspidal representations.
Fortunately $\natural_{\mf s}$ is known to be trivial in many cases, and we expect that it is
trivial whenever $G$ is quasi-split. We find it remarkable that such a simple description of a
strong invariant of very complicated algebras is possible.\\

\textbf{Hochschild homology via families of representations} \\
For more precise statements we employ algebraic families of $G$-representations. The families
relevant for us come from a parabolic subgroup $P = MU$ of $G$ and a tempered representation
$\eta$ of a Levi factor $M$ of $P$. All the representations $I_P^G (\eta \otimes \chi)$ with
$\chi \in X_\nr (M)$ can be realized on the same vector space $V_{P,\eta}$, and their matrix 
coefficients depend algebraically on $\chi$. The family of representations
\[
\mf F (M,\eta) = \{ I_P^G (\eta \otimes \chi) : \chi \in X_\nr (M) \} 
\]
induces a $\mc O (X_\nr (L))^{W(L,\mf s)}$-algebra homomorphism
\[
\begin{array}{cccc}
\mc F_{M,\eta} : & \mc H (G)^{\mf s} & \to & 
\mc O (X_\nr (M)) \otimes \End_{\C,\mr{fr}} (V_{P,\eta}) \\
 & f & \mapsto & [\chi \mapsto I_P^G (\sigma \otimes \chi)(f)]
\end{array},
\]
where the subscript fr stands for finite rank. Via Morita equivalences and the 
Hochschild--Kostant--Rosenberg theorem, that yields a map
\[
HH_n (\mc F_{M,\eta}) : HH_n (\mc H (G)^{\mf s}) \to \Omega^n (X_\nr (M)) .
\]
For $\chi \in X_\unr (M)$ the members of $\mf F (M,\eta)$ are tempered. Then 
Harish-Chandra's Plancherel isomorphism (Theorem \ref{thm:3.1}) shows that for $f \in
\mc S (G)^{\mf s}$ the matrix coefficients of $I_P^G (\eta \otimes \chi) (f)$ are smooth
functions on $X_\unr (M)$. We obtain a map
\[
HH_n (\mc F_{M,\eta}^t) : HH_n (\mc S (G)^{\mf s}) \to \Omega^n_{sm} (X_\unr (M)) ,
\]
where the subscript sm means smooth differential forms on a real manifold. With 
\cite[\S 1.2]{SolTwist} this setup can be generalized to algebraic families of virtual
representations, then we may speak of algebraic families in $\C \otimes_\Z R(G)^{\mf s}$ or
in $\C \otimes_\Z R^t (G)^{\mf s}$. 

For each $w \in W(L,\mf s)$ and each connected component $X_\nr (L)^w_c$ of $X_\nr (L)^w$,
we will construct a particular algebraic family 
\[
\mf F (w,c) = \big\{ \nu^1_{w,\chi} : \chi \in X_\nr (L)^w_c \big\} \quad \text{in} \quad 
\C \otimes_\Z R(G)^{\mf s}.
\]
From $\natural_{\mf s} : W(L,\mf s) \times W(L,\mf s) \to \C^\times$ we get a character
$\natural_{\mf s}^w$ of $Z_{W(L,\mf s)}(w)$.

\begin{thmintro}\label{thm:C}
\textup{(see Theorems \ref{thm:5.9}.b and \ref{thm:6.4})}
\begin{enumerate}[(i)]
\item The algebraic families $\mf F (w,c)$ induce a $\C$-linear bijection
\[
HH_n (\mc H (G)^{\mf s}) \to \Big( \bigoplus\nolimits_{w \in W(L,\mf s)} 
\Omega^n (X_\nr (L)^w) \otimes \natural_{\mf s}^w \Big)^{W(L,\mf s)} .
\]
\item Their tempered versions $\mf F^t (w,c) = \big\{ \nu^1_{w,\chi} : \chi \in X_\unr (L)^w_c 
\big\}$ induce an isomorphism of Fr\'echet spaces
\[
HH_n (\mc S (G)^{\mf s}) \to \Big( \bigoplus\nolimits_{w \in W(L,\mf s)} 
\Omega^n_{sm} (X_\unr (L)^w) \otimes \natural_{\mf s}^w \Big)^{W(L,\mf s)} .
\]
\end{enumerate}
\end{thmintro}

The canonicity of Theorem \ref{thm:B} can be formulated in similar terms. Namely, for each
algebraic family $\mf F (M,\eta)$ in $\Rep (G)^{\mf s}$ there are equalities
\[
\begin{array}{ccc}
HH_n (\mc F_{M,\eta}) \circ HH_n (\zeta^\vee) & = & HH_n (\mc F_{M,\zeta^\vee (\eta)}) ,\\
HH_n (\mc F^t_{M,\eta}) \circ HH_n (\zeta_t^\vee) & = & HH_n (\mc F^t_{M,\zeta^\vee (\eta)}) .
\end{array}
\]

\textbf{Hochschild homology groups in degree 0}\\
Theorem \ref{thm:C} admits a nice alternative description in degree $n = 0$. Let us say that
a linear function on $\C \otimes_\Z R (G)^{\mf s}$ is regular if it transforms every algebraic
family $\mf F (M,\eta)$ into a regular function on $X_\nr (M)$. Similarly we call a linear
function on $\C \otimes_\Z R^t (G)^{\mf s}$ smooth if it transforms $\mf F^t (M,\eta)$ into
a smooth function on $X_\unr (M)$.

\begin{thmintro}\label{thm:D}
\textup{(see Propositions \ref{prop:5.4} and \ref{prop:6.10})}
\begin{enumerate}[(i)]
\item The trace pairing $\mc H (G)^{\mf s} \times R(G)^{\mf s} \to \C$ 
induce a natural isomorphism of\\ $Z (\Rep (G)^{\mf s})$-modules
\[
HH_0 (\mc H (G)^{\mf s}) \to (\C \otimes_\Z R(G)^{\mf s})^*_{\mr{reg}} .
\]
\item The trace pairing $\mc S (G)^{\mf s} \times R^t (G)^{\mf s} \to \C$
induces a natural isomorphism of $Z(\Rep^t (G)^{\mf s})$-modules
\[
HH_0 (\mc S (G)^{\mf s}) \to (\C \otimes_\Z R_t (G)^{\mf s})^*_\infty .
\]
\end{enumerate}
\end{thmintro}

We note that Theorem \ref{thm:D}.(i) was already shown in \cite{BDK}, with much
more elementary methods. Theorem \ref{thm:D}.(ii) implies that the traces of
irreducible tempered representations in $\Rep (G)^{\mf s}$ span a dense subspace of the
space of trace functions on $\mc S (G)^{\mf s}$.\\

\textbf{The action of the Bernstein centre}\\
Theorems \ref{thm:B} and \ref{thm:C} do not yet reveal how the Bernstein centre
\[
Z( \Rep (G)^{\mf s}) \cong \mc O (X_\nr (L))^{W(L,\mf s)}
\] 
acts on $HH_n (\mc H (G)^{\mf s})$.
That action is more tricky than it could seem, because the bijections in Theorem \ref{thm:A}
do not always match the canonical actions of $\mc O (X_\nr (L))^{W(L,\mf s)}$ on the two sides.
We are aided by the finer decomposition of $R^t (G)$ and $\mc S (G)$ in ``Harish-Chandra 
blocks". Namely, to each square-integrable (modulo centre) representation $\delta$ of a Levi 
subgroup $M$ of $G$ one canonically associates a direct factor $R^t (G)^{\mf d}$ of $R^t (G)$, 
and a two-sided ideal $\mc S (G)^{\mf d}$ of $\mc S (G)$. If the supercuspidal support of 
$\delta$ is $(L,\sigma)$, then $R^t (G)^{\mf d} \subset R^t (G)^{\mf s}$ where 
$\mf d = [M,\delta]$ and $\mf s = [L,\sigma]$. The Plancherel isomorphism 
(see \cite{Wal} or Theorem \ref{thm:3.1}) entails:
\begin{equation}\label{eq:2}
\mc S (G)^{\mf s} = \bigoplus\nolimits_{\mf d \in \Delta_G^{\mf s}} \mc S (G)^{\mf d}
\end{equation}
for a suitable finite set $\Delta_G^{\mf s}$ of square-integrable (modulo centre) 
representations of Levi subgroups of $G$. This gives rise to a decomposition
\[
HH_n (\mc S (G)^{\mf s}) = 
\bigoplus\nolimits_{\mf d \in \Delta_G^{\mf s}} HH_n (\mc S (G)^{\mf d}) .
\]
For $\mc H (G)^{\mf s}$ no decomposition like \eqref{eq:2} exists. Nevertheless something
similar can be achieved with Hochschild homology groups, see below.

Again by Harish-Chandra's Plancherel isomorphism (Theorem \ref{thm:3.1}) 
\[
Z( \Rep^t (G)^{\mf d}) \cong C^\infty (X_\unr (M))^{W(M,\mf d)} ,
\]
for a certain finite group $W(M,\mf d)$ of transformations of $X_\nr (M)$.
Represent the $\mc O (X_\nr (L))^{W(L,\mf s)}$-character of $\delta$ by $\chi_\delta 
t_\delta^+$, where $\chi_\delta \in X_\unr (L)$ and $t_\delta^+ \in \Hom (L,\R_{>0})$. 
The natural map $Z(\Rep (G)^{\mf s}) \to Z(\Rep^t (G)^{\mf d})$ makes
$C^\infty (X_\unr (M))^{W(M,\mf d)}$ into a set of functions on $\chi_\delta t_\delta^+ 
X_\unr (M) \subset X_\nr (L)$. 

\begin{thmintro}\label{thm:E}
\textup{(see Theorem \ref{thm:6.1}.b, Lemma \ref{lem:6.8} and Lemma \ref{lem:5.5})}
\begin{enumerate}[(i)]
\item There exists a canonical decomposition 
\[
HH_n (\mc H (G)^{\mf s}) = \bigoplus\nolimits_{\mf d \in \Delta_G^{\mf s}}
HH_n (\mc H (G)^{\mf s})^{\mf d} ,
\]
where $HH_n (\mc H (G)^{\mf s})^{\mf d}$ is the inverse image of $HH_n (\mc S (G)^{\mf d})$
under the natural map $HH_n (\mc H (G)^{\mf s}) \to HH_n (\mc S (G)^{\mf s})$.
\item Suppose that $Z(\Rep^t (G)^{\mf d})$ does not annihilate the contribution (via Theorem 
\ref{thm:C}) of $\Omega^n_{sm}(X_\unr (L)^w_c)$ to $HH_n (\mc S (G)^{\mf s})$ . Then we can
arrange that $X_\unr (L)^w_c$ is contained in $\chi_\delta X_\unr (M)$. For $\chi \in 
X_\unr (L)^w_c$, $Z (\Rep^t (G)^{\mf d})$ acts on the fibre of $HH_n (\mc S (G)^{\mf s})$ over 
$W(L,\mf s)(w,\chi)$ via the character $W(M,\mf d) \chi_\delta^{-1} \chi$.
\item In the setting of part (ii), for $\chi \in X_\nr (L)^w_c$, $Z(\Rep (G)^{\mf s})$ 
acts on the fibre of $HH_n (\mc H (G)^{\mf s})$ over $W(L,\mf s)(w,\chi)$ via the 
character $W(L,\mf s)t_\delta^+ \chi$.
\end{enumerate}
\end{thmintro}

\textbf{Other homology theories}\\
There are standard techniques to derive the cyclic homology $HC_* (A)$ and the periodic cyclic
homology $HP_* (A)$ from the Hochschild homology of a $\C$-algebra $A$ \cite{Lod}. In our
cases $A = \mc H (G)^{\mf s}$ and $A = \mc S (G)^{\mf s}$, we can get them as the homology of
$HH_* (A)$ with respect to the usual exterior differential on forms.

\begin{thmintro}\label{thm:F}
\textup{(see \eqref{eq:1.8}, \eqref{eq:1.9} and Corollary \ref{cor:1.3})} \\
Theorem \ref{thm:C} induces isomorphisms
\[
HP_n (\mc H (G)^{\mf s}) \cong HP_n (\mc S (G)^{\mf s}) \cong \bigoplus\nolimits_{m \in \Z} 
\Big( \bigoplus\nolimits_{w \in W(L,\mf s)} H_{dR}^{n + 2m} (X_\nr (L)^w) \otimes
 \natural_{\mf s}^w \Big)^{W(L,\mf s)} .
\]
\end{thmintro}

The periodic cyclic homology of a Fr\'echet algebra relates to its topological K-theory via a
Chern character. We can pass from $\mc S (G)^{\mf s}$ to its $C^*$-completion via
suitable Morita equivalent Fr\'echet subalgebras. In this way we compute the topological 
K-theory of any Bernstein block in the reduced $C^*$-algebra of $G$:

\begin{thmintro}\label{thm:G}
\textup{(see Theorem \ref{thm:1.5})}\\
There is an isomorphism of vector spaces
\[
K_* \big( C_r^* (G)^{\mf s} \big) \otimes_\Z \C \cong 
K^*_{W(L,\mf s), \natural_{\mf s}} (X_\unr (L)) \otimes_\Z \C .
\]
\end{thmintro}
Here $K^*_{W,\natural}$ denotes $W$-equivariant K-theory, twisted by a 2-cocycle $\natural$.
Theorem \ref{thm:G} confirms \cite[Conjecture 5]{ABPS2}, modulo torsion elements in the K-groups.\\

\textbf{Relation with previous work and outlook}\\
The Hochschild homology of $\mc H (G)$ has been determined earlier in \cite{Nis1}. The methods
of Nistor are completely different from ours, he obtains a description of $HH_n (\mc H (G))$
in terms of several algebraic subgroups of $G$ and of the continuous group cohomology of
certain modules. This arises from a generalization of the standard techniques for discrete
groups, a filtration of $\mc H (G)$ as bimodule, and spectral sequences. In \cite[\S 6]{Nis1}
a ``parabolic induction map" $HH_n (\mc H (G)) \to HH_n (\mc H (M))$ is constructed, for a
Levi subgroup $M$ of $G$. It would be interesting to relate this to our methods and results,
maybe that could provide some information about supercuspidal representations.

A technique prominent in Nistor's work is localization of $HH_* (\mc H (G))$ at conjugacy
classes in $G$. That can be regarded as a higher order version of taking the trace of a
representation at a conjugacy class. Of particular interest is the localization of 
$HH_* (\mc H (G))$ at the set of compact elements of $G$, for that yields the periodic
cyclic homology $HP_* (\mc H (G))$ \cite{HiNi}. While localization at one conjugacy class 
in $G$ appears to be intractable in our setup, localization at all compact elements is
within reach. Since every compact element lies in the kernel of every unramified character,
such localization removes all differential forms that are not locally constant on (subvarieties 
of) $X_\nr (L)$. Moreover, in the description from Theorem \ref{thm:C} the locally constant
differential forms constitute a set of representatives for $HP_* (\mc H (G))$, that follows
from Lemma \ref{lem:1.4} (and with Theorem \ref{thm:F} it also works for $\mc S (G)$). 
Hence the localization of $HH_* (\mc H (G))$ at the compact elements of $G$ is given precisely 
by the subspace of locally constant differential forms.

At the same time $HP_* (\mc H (G))$ is naturally isomorphic to the equivariant homology of
the Bruhat--Tits building of $G$, which yields yet another, more geometric, picture of
$HH_* (\mc H (G))$ and $HH_* (\mc S (G))$. It would be nice if the Bernstein decomposition of 
$\mc H (G)$ and of $\mc S (G)$ could be expressed in such geometric terms, as suggested in 
\cite{BHP}.\\

\textbf{Structure of the paper}\\
This paper is part of a larger project that includes \cite{SolTwist} and \cite{KaSo}.
Initially those two and the current text were conceived as one paper. When that grew too 
big, two parts were split off and transformed into independent papers. Although neither 
\cite{SolTwist} nor \cite{KaSo} deals with $p$-adic groups, both prepare for this paper. 
Many results in Section \ref{sec:HeckeG} rely on the study of the Hochschild homology of slightly 
simpler algebras in \cite{SolTwist}. In Section \ref{sec:HHSG} we need several nontrivial
results about topological algebras and modules involving smooth functions. These are 
formulated and proven in larger generality in \cite{KaSo}.

Section \ref{sec:familiesG} is preparatory, its main purpose is to describe precisely what
kind of families of representations we will use. Already there we see that it is
convenient to replace $\mc H (G)^{\mf s}$ by its subalgebra of functions that are biinvariant
under a well-chosen compact open subgroup $K$.

We start our investigations of the Hecke algebra in earnest by transforming it into simpler
algebras via Morita equivalences, in Paragraph \ref{par:structure}. This relies largely on
\cite{SolEnd}, but we go a little further and establish Theorem \ref{thm:A}. In Paragraph
\ref{par:HHHG} we set up a good array of algebraic families of $G$-representations, and
we approach $HH_n (\mc H (G)^{\mf s})$ via formal completions at central characters. That
yields a rough description in terms of differential forms on varieties like $X_\nr (M)$,
not yet indexed by $W(L,\mf s)$ as desired, but already sufficient for Theorem \ref{thm:D}.i.
The local results thus obtained are glued together in Paragraph \ref{par:component}. When
that is done, Theorems \ref{thm:B}, \ref{thm:C} and \ref{thm:E} for $\mc H (G)^{\mf s}$
follow quickly.

For the Schwartz algebra $\mc S (G)$ no such simplifying Morita equivalences are available,
but Harish-Chandra's Plancherel isomorphism from Theorem \ref{thm:3.1} works better than
for $\mc H (G)$. Our main technique to determine $HH_n (\mc S (G)^{\mf s})$ is to derive
it from $HH_n (\mc H (G)^{\mf s})$ via a comparison of formal completions with respect to
central characters. To carry out that strategy completely, we need to check that the 
relevant modules are Fr\'echet spaces, which is done in Paragraph \ref{par:top}. In
Paragraph \ref{par:computation} we first show Theorem \ref{thm:E}.ii, so that we can work
with $C^\infty (X_\unr (M))^{W(M,\mf d)}$-modules. That plays a role in the proof of
Theorem \ref{thm:C}.ii, from which Theorem \ref{thm:D}.ii follows readily. Then
we establish Theorem \ref{thm:B}.ii and we compare $HH_n (\mc S (G)^{\mf d})$ with
$HH_n (\mc H (G)^{\mf s})^{\mf d}$.

Section \ref{sec:cyclic} contains the derivation of the (periodic) cyclic homology of
$\mc H (G)^{\mf s}$ and of $\mc S (G)^{\mf s}$. We also draw conclusions for the topological
K-theory of $G$. 
In the final section we work out the examples $G = SL_2 (F)$ and $G = GL_n (F)$.\\

\textbf{Acknowledgements.}\\
We thank David Kazhdan, Roman Bezrukavnikov and Alexander Braverman for interesting email
discussions about group algebras of reductive $p$-adic groups. We are grateful to Roger
Plymen for his feedback on an earlier version, which inspired us to add the last two sections.
We also thank the referee for his or her encouraging report.

\numberwithin{equation}{section}

\section{Algebraic families of $G$-representations}
\label{sec:familiesG}

Let $\mc G$ be a connected reductive group defined over a non-archimedean local field $F$,
and consider the group of rational points $G = \mc G (F)$. Let Rep$(G)$ be the category of smooth 
$G$-representations and let $\mr{Rep}_f (G)$ be the subcategory of finite length representations. 
Let $R(G)$ be the Grothendieck group of $\mr{Rep}_f (G)$. By imposing temperedness we obtain 
the category $\mr{Rep}_f^t (G)$ and the Grothendieck group $R^t (G)$.

We fix a Haar measure on $G$ we let $\mc H (G)$ be the algebra of locally constant compactly
supported complex-valued functions on $G$, endowed with the convolution product. Recall that the 
Schwartz algebra $\mc S (G)$ \cite[\S III.6]{Wal} satisfies $\Irr (\mc S (G)) = \Irr^t (G)$,
where the latter denotes the space of irreducible tempered $G$-representations.
We fix a compact open subgroup $K$ of $G$ and we consider the algebras $\mc H (G,K)$ and 
$\mc S (G,K)$ of $K$-biinvariant functions in, respectively, $\mc H (G)$ and $\mc S (G)$. 
By definition 
\[
\mc H (G) = \varinjlim\nolimits_K \mc H (G,K) \quad \text{and} \quad  
\mc S (G) = \varinjlim\nolimits_K \mc S (G,K) ,
\]
where the inductive limit runs over the set of all compact subgroups $K$ of $G$, partially ordered 
by reverse inclusion.

Let $X_\nr (G)$ be the group of unramified characters of $G$ and let $X_\unr (G)$ be the subgroup
of unitary unramified characters. The first is a complex algebraic torus and the second is a compact
real torus of the same dimension.

Let $P$ be a parabolic subgroup of $G$ with a Levi factor $M$, and let $I_P^G : \mr{Rep} (M) \to
\mr{Rep} (G)$ be the normalized parabolic induction functor.
Let $\sigma \in \mr{Rep}^t_f (M)$ and suppose that the space $I_P^G (V_\sigma )^K$, which 
has finite dimension by the admissibility of $I_P^G (\sigma)$, is nonzero. Since $I_P^G (V_\sigma )$ 
can be realized as a space of functions on a good maximal compact subgroup of $G$, we may identify 
the vector spaces
\[
I_P^G (V_\sigma )^K \quad \text{and} \quad I_P^G (V_\sigma \otimes \chi )^K 
\quad \text{for } \chi \in X_\nr (M).
\]
Every $f \in \mc S (G,K)$ gives a family of operators $I_P^G (\sigma \otimes \chi)(f)$ 
on $I_P^G (V_\sigma )^K$, parametrized by $\chi \in X_\unr (M)$. It turns out 
\cite[Proposition VII.1.3]{Wal} that $I_P^G (\sigma \otimes \chi)(f)$ depends smoothly on $\chi$. 
When $f \in \mc H (G,K)$, this even works for all $\chi \in X_\nr (M)$, and the outcome depends
algebraically on $\chi$. More precisely, this enables us to define algebra homomorphisms
\begin{equation}\label{eq:3.9}
\begin{array}{cccc}
\mc F_{M,\sigma} : & \mc H (G,K) & \to & 
\mc O (X_\nr (M)) \otimes \mr{End}_\C (I_P^G (V_\sigma )^K) \\
\mc F^t_{M,\sigma} : & \mc S (G,K) & \to & 
C^\infty (X_\unr (M)) \otimes \mr{End}_\C (I_P^G (V_\sigma)^K) \\
 & f & \mapsto & [\chi \mapsto I_P^G (\sigma \otimes \chi)(f)]
\end{array}.
\end{equation}
Recall the natural pairing
\begin{equation}\label{eq:3.6}
\begin{array}{ccc}
HH_0 (\mc H (G,K)) \times \mr{Rep}_f (G) & \to & \C \\
(h,\pi) & \mapsto & \mr{tr} (\pi (h),V_\pi) = \mr{tr} (\pi (h), V_\pi^K) 
\end{array} .
\end{equation}
This and its analogue for $\mc S (G,K)$ induce bilinear maps
\begin{equation}\label{eq:3.11}
\begin{array}{ccc}
HH_0 (\mc H (G,K)) \times \C \otimes_\Z R(G) & \to & \C , \\
HH_0 (\mc S (G,K)) \times \C \otimes_\Z R^t (G) & \to & \C .
\end{array}
\end{equation}
We say that a linear function $f$ on $\C \otimes_\Z R(G)$ is regular if
\[
X_\nr (M) \to \C : \chi \mapsto f (I_P^G (\sigma \otimes \chi)) \quad
\text{is a regular function,}
\]
for all $(M,\sigma)$ as above. Similarly we call $f \in (\C \otimes_\Z R^t (G) )^*$ smooth if
\[
X_\unr (M) \to \C : \chi \mapsto f (I_P^G (\sigma \otimes \chi)) \quad
\text{is a smooth function,}
\]
for all $(M,\sigma)$ as above. We write
\[
\begin{array}{lll}
(\C \otimes_\Z R (G) )^*_{\mr{reg}} & = & \{ f \in (\C \otimes_\Z R (G) )^* : f \text{ is regular} \} ,\\
(\C \otimes_\Z R^t (G) )^*_\infty & = & \{ f \in (\C \otimes_\Z R^t (G) )^* : f \text{ is smooth} \} . 
\end{array}
\]
With these notations, \eqref{eq:3.9} and \eqref{eq:3.11} induces maps
\begin{equation}\label{eq:3.12}
\begin{array}{ccc}
HH_0 (\mc H (G,K)) & \to & (\C \otimes_\Z R (G) )^*_{\mr{reg}} ,\\
HH_0 (\mc S (G,K)) & \to & (\C \otimes_\Z R^t (G) )^*_\infty .
\end{array}
\end{equation}
It is easy to see that the former is a homomorphism of $Z(\mc H (G,K))$-modules and that the
latter is is a homomorphism of $Z(\mc S (G,K))$-modules.

The normalized parabolic induction functor $I_P^G$ induces a $\Z$-linear map
\[
I_M^G : R(M) \to R(G) .
\]
It may be denoted this way, because given a Levi subgroup $M$ of $G$ it does not depend on the choice
of the parabolic subgroup $P$ with Levi factor $M$. We define
\[
R_I (G) = R(G) \; \cap \; \Q \otimes_\Z \sum\nolimits_{M \subsetneq G} I_M^G (R (M)) ,
\]
where the sum runs over all proper Levi subgroups $M$ of $G$. We say that a finite dimensional 
$G$-representation is elliptic if it admits a central character and does not belong to $R_I (G)$.
By \cite[Proposition 3.1]{BDK} every Bernstein component of $\Irr (G)$ contains only a finite number
of $X_\nr (G)$-orbits of irreducible elliptic representations. It follows from the Langlands 
classification that every such $X_\nr (G)$-orbit contains a tempered $G$-representation. 

\begin{defn}\label{def:3}
Let $\eta \in \Irr (M)$ be elliptic and tempered. Then
\[
\mf F (M,\eta) = \{ I_P^G (\eta \otimes \chi) : \chi \in X_\nr (M) \}
\]
is an algebraic family of $G$-representations. Its dimension is $\dim_\C (X_\nr (M))$, that is,
the dimension of the maximal split torus in $Z(M)$. The subset 
\[
\mf F^t (M,\eta) = \{ I_P^G (\eta \otimes \chi) : \chi \in X_\unr (M) \}
\]
is a tempered algebraic family of $G$-representations, also of dimension\\ 
$\dim_\R (X_\unr (M)) = \dim_\C (X_\nr (M))$.
\end{defn}

We fix a minimal parabolic subgroup $P_0$ of $G$ and a maximal split torus $S_0$ of $P_0$. 
A parabolic (resp. Levi) subgroup of $G$ is standard if it contains $P_0$ (resp. $S_0$). 
In the above definition it suffices to consider standard parabolic and standard Levi subgroups of
$G$, because every pair $(P,M)$ is $G$-conjugate to a standard such pair. 

Consider a Bernstein block Rep$(G)^{\mf s}$ of Rep$(G)$, determined by a tempered supercuspidal
representation of a standard Levi subgroup $L$ of $G$. Let $R(G)^{\mf s}$ be the 
Grothendieck group of $\mr{Rep}_f (G)^{\mf s}$. Similarly we define $R^t (G)^{\mf s}$ as the
Grothendieck group of the category $\mr{Rep}_f^t (G)^{\mf s}$ of tempered modules in
$\mr{Rep}_f (G)^{\mf s}$. If we restrict to standard 
parabolic/Levi subgroups of $G$ (as we will often do tacitly), Rep$(G)^{\mf s}$ contains only 
finitely algebraic families of $G$-representations as in Definition \ref{def:3}.
Moreover, by \cite[Corollary 3.1]{BDK} these families span $\Q \otimes_\Z R(G)^{\mf s}$. 

We want to minimize the redundancy, by choosing a smaller collection of algebraic families of 
$G$-representations. One step in that direction is to determine which members of an algebraic family 
are equivalent in $R(G)$. To that end we briefly recall Harish-Chandra's Plancherel isomorphism 
for $G$ \cite{Wal}.

Consider a Levi subgroup $M$ of $G$ and an irreducible square-integrable modulo centre representation
$(\delta, V_\delta)$ of $M$. Harish-Chandra's disjointness theorem \cite[Proposition III.4.1]{Wal}
asserts that every irreducible tempered $G$-representation is a direct summand of $I_P^G (\delta)$
for such a pair $(M,\delta)$, which moreover is unique up to $G$-conjugation. Irreducible 
square-integrable modulo centre representations of $M$ become discrete series upon restriction to
the derived subgroup of $M$, so the only way to deform them continuously is twisting with unitary
unramified characters of $M$. Therefore the connected components of the space $\Irr^t (G)$ 
are parametrized by pairs $(M,\delta)$ modulo the equivalence relation
\begin{equation}\label{eq:3.4}
(M,\delta) \sim (g M g^{-1}, g \cdot (\delta \otimes \chi)) \qquad g \in G, \chi \in X_\unr (M).
\end{equation}
We denote such an equivalence class by $\mf d = [M,\delta]$. (When $\delta$ is supercuspidal, we
also have the equivalence class $\mf s = [M,\delta]$, which includes the tensoring by non-unitary
unramified characters and determines an entire Bernstein component of $\Irr (G)$.)

Let $P$ be a parabolic subgroup of $G$ with Levi factor $M$ and let $\chi \in X_\unr (M)$. To
$(M,\delta,\chi)$ we associate the tempered $G$-representation $I_P^G (\delta \otimes \chi)$, 
whose isomorphism class does not depend on the choice of $P$. Then the connected 
component of $\Irr^t (G)$ associated to $(M,\delta)$ consists of the irreducible summands 
(or equivalently subquotients) of the representations $I_P^G (\delta \otimes \chi)$ with 
$\chi \in X_\unr (M)$. 

The Plancherel isomorphism describes the image of $\mc F^t_{M,\delta}$, as the invariants for 
an action of a certain finite group. The group 
\[
X_\nr (M,\delta) = \{ \chi \in X_\nr (M) : \delta \otimes \chi \cong \delta \}
\]
is finite and contained in $X_\unr (M)$, because it consists of characters that are trivial on
$Z(M)$. Consider the subset 
\[
\Irr (M)^{\mf d} = \{ \delta \otimes \chi : \chi \in X_\nr (M) \} \text{ of } \Irr (M). 
\]
The map $\chi \mapsto \delta \otimes \chi$ provides a diffeomorphism 
\begin{equation}\label{eq:3.1}
X_\nr (M) / X_\unr (M,\delta) \to \Irr (M)^{\mf d} .
\end{equation}
It is not canonical, because it depends on the choice of $\delta$ in $\Irr (M)^{\mf d}_t$. 
For each $\chi' \in X_\nr (M,\delta)$ we fix a unitary $M$-isomorphism $\delta \cong \delta 
\otimes \chi'$, and we induce it to a family of $G$-isomorphisms
\begin{equation}\label{eq:3.2}
I(\chi',P,\delta,\chi) : I_P^G (\delta \otimes \chi) \to I_P^G (\delta \otimes \chi' \chi) .
\end{equation}
We write
\[
W_{\mf d} = \{ w \in N_G (M)/M : w \text{ stabilizes } \Irr (M)^{\mf d} \} . 
\]
By \cite[Lemma 3.3]{SolComp} the action of an element $w \in W_{\mf d}$ on $\Irr (M)^{\mf d}$ can
be lifted (non-canonically) along \eqref{eq:3.1}, to an automorphism of the complex algebraic
variety $X_\nr (M)$ such that 
\[
w \cdot (\delta \otimes \chi) \cong \delta \otimes w (\chi)  \qquad
\text{for all } \chi \in X_\unr (M) .
\]
By \cite[Lemme V.3.1]{Wal} there exists a unitary $G$-isomorphism
\begin{equation}\label{eq:3.3}
I(w,P,\delta,\chi) : I_P^G (\delta \otimes \chi) \to I_P^G (\delta \otimes w (\chi)) ,
\end{equation}
depending smoothly and rationally on $\chi \in X_\unr (M)$. Let $W(M,\mf d)$ be the group of
transformations of $X_\nr (M)$ generated by $X_\nr (M,\delta)$ and the actions of elements of 
$W_{\mf d}$. Now we apply \cite[Lemma 3.3]{SolComp} to the covering of tori
\[
X_\nr (M) \to X_\nr (M) / X_\nr (M,\delta), 
\]
and we obtain a short exact sequence
\begin{equation}\label{eq:3.10}
1 \to X_\nr (M,\delta) \to W(M,\mf d) \to W_{\mf d} \to 1 .
\end{equation}
The intertwining operators \eqref{eq:3.2} and \eqref{eq:3.3} give rise to analogous families of
$G$-iso\-mor\-phisms for any element of $W(M,\mf d)$. These are far from unique, but
for any fixed $\chi \in X_\unr (M)$ they are unique up to scalars. The group 
$W(M,\mf d)$ acts on $C^\infty (X_\unr (M)) \otimes \mr{End}_\C (I_P^G (V_\delta)^K)$ by
\[
(w \cdot (f \otimes A))(\chi) = 
f ( w^{-1} \chi) \otimes I(w,P,\delta,w^{-1}\chi) A I(w,P,\delta,w^{-1}\chi)^{-1} .
\]
Let $\Delta_{G,K}$ be a set of representatives for the $(M,\delta)$ with $I_P^G (V_\delta)^K
\neq 0$, modulo the equivalence relation \eqref{eq:3.4}. We assume that every $M$ occurring here
is a Levi factor of a standard parabolic subgroup $P$.

\begin{thm}\label{thm:3.1} \textup{\cite[\S VIII.1]{Wal}} \\
There is an isomorphism of Fr\'echet algebras
\[
\begin{array}{ccl}
\mc S (G,K) & \to & \bigoplus\nolimits_{(M,\delta) \in \Delta_{G,K}} \Big( C^\infty (X_\unr (M)) 
\otimes \mr{End}_\C \big( I_P^G (V_\delta)^K \big) \Big)^{W(M,\mf d)} \\
f & \mapsto & \bigoplus\nolimits_{(M,\delta) \in \Delta_{G,K}} \mc F^t_{M,\delta}(f)
\end{array}.
\]
\end{thm}

An important ingredient of the proof of this theorem is Harish-Chandra's commuting algebra
theorem \cite[Theorem 5.5.3.2]{Sil}, a description of the involved spaces
of $G$-homomorphisms. Namely, for $\chi_1,\chi_2 \in X_\unr (M)$:
\begin{equation}\label{eq:3.5}
\mr{Hom}_G \big( I_P^G (\delta \otimes \chi), I_P^G (\delta \otimes \chi') \big) =
\mr{span} \{ I (w,P,\delta,\chi) : w \in W(M,\mf d), w(\chi) = \chi' \} .
\end{equation}
Consider an algebraic family $\mf F (M',\eta')$ contained in Rep$(G)^{\mf s}$. We may assume
that $M' \supset M$ and $\eta' \subset I^{M'}_{M' \cap P_0 M}(\delta \otimes \chi_0)$ for
some $(M,\delta) \in \Delta_{G,K}$. Let $W(M',M,\eta')$ be the subgroup of $W(M,\mf s)$
that stabilizes 
\[
\mf F (\mc M',\eta') = \{ I_{P_0 M'}^G (\eta' \otimes \chi') : \chi' \in X_\nr (M') \}, 
\]
with respect to the action via the intertwining operators $I(w,P,\delta,?)$.

\begin{lem}\label{lem:3.2}
Two members of $\mf F (M',\eta')$ in the same $W(M',M,\eta')$-orbit have the same trace. 
Two generic members of $\mf F (M',\eta')$ have the same trace if and only if they belong
to the same $W(M',M,\eta')$-orbit. Here a generic point if $\mf F (M',\eta')$ means: if an 
element $w \in W(M,\mf s)$ fixes the point or the intertwining operator associated to $w$ has 
a singularity at the cuspidal support of the point, then then $w$ has that property for all 
members of $\mf F (M',\eta')$. 
\end{lem}
\begin{proof}
The action of $w \in W(M,\mf d)$ on the collection of direct summands of the 
$I_P^G (\delta \otimes \chi)$ comes from an algebraic action on $X_\nr (M)$ and conjugation
by some operator. Hence, for a generic $\pi = I_{P_0 M'}^G (\eta' \otimes \chi') \in 
\mf F (M',\eta')$, the representation $w \pi$ lies in $\mf F (M',\eta')$ if and 
only if $w \in W(M',M,\eta')$. In combination with \eqref{eq:3.5}, that implies the second
claim for generic tempered members of $\mf F (M',\eta')$. 

In fact Harish-Chandra's commuting algebra theorem \eqref{eq:3.5} also holds for generic 
$\chi_1,\chi_2 \in X_\nr (M)$, one only needs to avoid the poles of the intertwining operators 
$I (w,P,\delta,\chi)$. This follows for instance from \cite[Theorem 1.6]{ABPS1}. Then the 
above argument can be applied to all generic members, and yields the first claim.

For any $f \in \mc H (G)$ and $w \in W(M',M,\eta')$,
\[
\mr{tr} \big( f, I_{P_0 M'}^G (\eta' \otimes \chi') \big) \quad \text{ and } \quad
\mr{tr} \big( f, I_{P_0 M'}^G (\eta' \otimes w \chi') \big)
\]
are algebraic functions of $\chi' \in X_\nr (M')$. These two functions agree for generic 
tempered $\chi' \in X_\unr (M')$, so they agree on the whole of $X_\nr (M')$. 
\end{proof}

Now we can finally describe how to choose a minimal set of algebraic families of $G$-representations
in Rep$(G)^{\mf s}$.

We start with the family $\mf F (L,\sigma)$ and proceed recursively. Suppose that for every
dimension $D > d$ we have chosen a set of $D$-dimensional algebraic families $\mf F (M_i,\omega_i)$, 
where $i$ runs through some index set $I_D$, with the following property:
for generic $\chi_i \in X_\nr (M_i)$ the set 
\[
\big\{ I_{P_j}^G (\omega_j \otimes \chi_j) : j \in I_D, D > d, \mb{Sc}(I_{P_j}^G (\omega_j \otimes 
\chi_j)) = \mb{Sc}(I_{P_i}^G (\omega_i \otimes \chi_i)) \big\}
\]
is linearly independent in $\Q \otimes_\Z R(G)^{\mf s}$. Here we regard all $\chi_j$ in one 
$W(M_j,M,\omega_j)$-orbit as the same, because by Lemma \ref{lem:3.2} they yield the same
element $I_{P_j}^G (\omega_j \otimes \chi_j)$ in $R(G)^{\mf s}$.

Next we consider the set of $d$-dimensional algebraic families $\mf F (M'_i,\omega'_i)$ with
$(P'_i,M'_i)$ standard. Suppose that for generic $\chi'_i \in X_\nr (M'_i)$, the representation
$I_{P'_i}^G (\omega'_i \otimes \chi'_i)$ is $\Q$-linearly independent from
\[
\big\{ I_{P_j}^G (\omega_j \otimes \chi_j) : j \in I_D, D > d, \mb{Sc}(I_{P_j}^G (\omega_j \otimes \chi_j))
= \mb{Sc}(I_{P'_i}^G (\omega'_i \otimes \chi'_i)) \big\} ,
\]
were we still regard $\chi_j$ as an element of $X_\nr (M_j) / W(M_j ,M,\omega_j)$.
Then we add $\mf F (M'_i,\omega'_i)$ to our collection of algebraic families.

Consider the remaining $d$-dimensional algebraic families. For $\mf F (M'_j, \omega'_j)$ we look at 
the same condition as for $\mf F (M'_i,\omega'_i)$, but now with respect to the index set 
$\cup_{D > d} I_D \cup \{i'\}$ instead of $\cup_{D > d} I_D$. If that condition is fulfilled, we add 
$\mf F (M'_j, \omega'_j)$ to our set of algebraic families. We continue this process until none of the
remaining $d$-dimensional algebraic families is (over generic points of that family) $\Q$-linearly
independent from the algebraic families that we chose already. At that point our set of 
$d$-dimensional algebraic families is complete, and we move on to families of dimension $d-1$.

In the end, this algorithm yields a collection 
\[
\{ \mf F (M_i,\omega_i) : i \in I_d, 0 \leq d \leq \dim X_\nr (L) \}
\]
such that:
\begin{itemize}
\item the representations
\begin{equation}\label{eq:3.13}
\{ I_{P_i}^G (\omega_i \otimes \chi_i ), i \in \cup_d I_d, \chi_i \in X_\nr (M_i) \}
\end{equation}
span $\Q \otimes_\Z R(G)^{\mf s}$,
\item if we remove any index from $\cup_d I_d$, the previous bullet does not hold any more,
\item for generic $\chi \in X_\nr (M_i)$, $I_{P_i}^G (\omega_i \otimes \chi_i)$ does not belong
to the span in $\Q \otimes_\Z R(G)^{\mf s}$ of the other families $\mf F (M_j, \omega_j)$.
\end{itemize}
We note that these conditions do not imply that \eqref{eq:3.13} is a basis of
$\Q \otimes_\Z R(G)^{\mf s}$. Some linear dependence is still possible for representations
with a specific cuspidal support $(L,\sigma \otimes \chi)$, namely when the algebraic R-group
of $\sigma \otimes \chi$ acts on $I_{P_0 L}^G (\sigma \otimes \chi)$ via a projective,
non-linear representation. That does not happen often though.

The formula \eqref{eq:3.9} for the partial Fourier transform $\mc F_{M,\delta}$ also applies
with any elliptic $M$-representation instead of $\delta$ (which is square-integrable modulo 
centre). For each $i \in \cup_{d=0}^{\dim X_\nr (L)} I_d$, this provides algebra homomorphisms
\begin{equation}\label{eq:3.7}
\begin{array}{cccc}
\mc F_{M_i,\eta_i} : & \mc H (G,K) & \to & 
\mc O (X_\nr (M_i)) \otimes \mr{End}_\C (I_{P_i}^G (V_{\eta_i})^K) \\
\mc F^t_{M_i,\eta_i} : & \mc S (G,K) & \to & 
C^\infty (X_\unr (M_i)) \otimes \mr{End}_\C (I_{P_i}^G (V_{\eta_i})^K) \\
 & f & \mapsto & [\chi \mapsto I_{P_i}^G (\eta_i \otimes \chi)(f)]
\end{array}.
\end{equation}
These induce maps on Hochschild homology
\begin{equation}\label{eq:3.8}
\begin{array}{cccc}
HH_n ( \mc F_{M_i,\eta_i}) : & HH_n (\mc H (G,K)) & \to & \Omega^n (X_\nr (M_i)) , \\
HH_n ( \mc F^t_{M_i,\eta_i}) : & HH_n (\mc S (G,K)) & \to & \Omega^n_{sm} (X_\unr (M_i)) .
\end{array}
\end{equation}
We added a subscript sm to emphasize that we consider smooth differential forms
on a real manifold. We will describe $HH_n (\mc H (G,K))$ and $HH_n (\mc S (G,K))$ in terms
of the maps \eqref{eq:3.8}.

\section{The Hecke algebra of $G$}
\label{sec:HeckeG}

Let $\mf B (G)$ be the set of inertial equivalence classes $\mf s = [L,\sigma ]_G$.
Let $\mc H (G)^{\mf s}$ be the two-sided ideal of $\mc H (G)$ corresponding to $\Rep (G)^{\mf s}$,
so that
\begin{equation}\label{eq:5.14}
\mc H (G) = \bigoplus\nolimits_{\mf s \in \mf B (G)} \mc H (G)^{\mf s}.
\end{equation}
At this point we need the following continuity property of the functors $HH_n$
from \cite[E.1.13]{Lod}. Namely, let $A = \varinjlim_i A_i$ be an inductive limit of algebras. Then
\begin{equation}\label{eq:5.9}
HH_n (A) \cong \varinjlim\nolimits_i HH_n (A_i) \qquad n \in \Z_{\geq 0}.
\end{equation}
In particular
\begin{equation}\label{eq:5.13}
HH_n (\mc H (G)) = \bigoplus\nolimits_{\mf s \in \mf B (G)} HH_n (\mc H (G)^{\mf s}) .
\end{equation}
We fix a Bernstein block $\Rep (G)^{\mf s}$ in $\Rep (G)$, where $\mf s = [L,\sigma]_G$.
According to \cite{BeDe} there exist arbitrarily small compact open subgroups $K$ of $G$ such that
\[
\mc H (G,K)^{\mf s} = \mc H (G)^{\mf s} \cap \mc H (G,K)
\]
is Morita equivalent with $\mc H (G)^{\mf s}$. Notice that $\mc H (G,K)^{\mf s}$ is unital
but $\mc H (G)^{\mf s}$ is not. In fact $\mc H (G)^{\mf s}$ is the direct limit of the algebras
$\mc H (G,K_i)^{\mf s}$, where each $K_i$ has the same property as $K$ and $\cap_{i=1}^\infty K_i
= \{1\}$. All the inclusions $\mc H (G,K_i)^{\mf s} \to \mc H (G,K_j)^{\mf s}$ induce isomorphisms
on Hochschild homology, by Morita invariance. Applying \eqref{eq:5.9} another time, we find
\[
HH_n (\mc H (G)^{\mf s}) \cong \lim_{i \to \infty} HH_n (\mc H (G,K_i )^{\mf s}) \cong
HH_n (\mc H (G,K)^{\mf s}).
\]
The centre of $\mc H (G,K)^{\mf s}$ is isomorphic to the centre of the category $\Rep (G)^{\mf s}$. 
The latter can be made more explicit with the notations from Section \ref{sec:familiesG}. 
Namely, \eqref{eq:3.1} induces an algebra isomorphism
\[
\mc O (\Irr (L)^{\mf s}) \cong \mc O (X_\nr (L) / X_\nr (L,\sigma)) = 
\mc O (X_\nr (L) )^{X_\nr (L,\sigma)} .
\]
An instance of the short exact sequence \eqref{eq:3.10} gives
\[
1 \to X_\nr (L,\sigma) \to W(L,\mf s) \to W_{\mf s} \to 1 ,
\]
By \cite[Proposition 3.14]{BeDe} and \eqref{eq:3.10} there are isomorphisms
\begin{equation}\label{eq:4.2}
Z( \mc H (G,K)^{\mf s}) \cong Z (\Rep (G)^{\mf s}) \cong 
\mc O (\Irr (L)^{\mf s})^{W_{\mf s}} \cong \mc O (X_\nr (L) )^{W (L,\mf s)} .
\end{equation}
It is also known from \cite[\S 3.13]{BeDe} that 
\begin{equation}\label{eq:6.9}
\mc H (G,K)^{\mf s} \text{ has finite rank as } Z(\mc H (G,K)^{\mf s}) \text{-module.}
\end{equation}

\subsection{Structure of the module category} \ 
\label{par:structure}

We aim to describe $\mc H (G,K)^{\mf s}$ and its modules locally on $X_\nr (L)$. We write 
\[
X_\nr^+ (L) = \mr{Hom}(L,\R_{>0})
\]
and we fix $u \in X_\unr (L)$. Let $W(L,\mf s)_u$ be the stabilizer of $u$ in $W(L,\mf s)$. 
We let $U_u \subset X_\nr (L)$ be a connected neighborhood of 
$u$ in $X_\nr (L)$ (for the analytic topology) satisfying \cite[Condition 6.3]{SolEnd}:
\begin{itemize}
\item $U_u$ is stable under $W(L,\mf s)_u$ and under $X_\nr^+ (L)$,
\item $W(L,\mf s)_u \cap U_u = \{u\}$,
\item a technical condition to ensure that $u$ is the ``most singular" point of $U_u$.
\end{itemize} 
The tangent space at 1 of the complex torus $X_\nr (L)$ is 
\[
\mf t := \C \otimes_\Z X^* (L),
\]
and the exponential map $\exp : \mf t \to X_\nr (L)$ is equivariant for $W(L,\mf s)_1$. We modify
it to a $W(L,\mf s)_u$-equivariant map
\[
\begin{array}{cccc}
\exp_u : & \mf t & \to & X_\nr (L) \\
& \lambda & \mapsto & u \exp (\lambda) 
\end{array}.
\]
This means that we regard $\mf t$ also as the tangent space of $X_\nr (L)$ at $u$. 
Let $\log_u$ be the branch of $\exp_u^{-1}$ with $\log_u (u) = 0$. By \cite[Condition 6.3]{SolEnd}
$\exp_u$ restricts to a diffeomorphism, $\log_u (U_u) \to U_u$. 
From \cite[\S 7]{SolEnd} we get a root system $\Phi_u$ in $\mf t$, whose Weyl group is a subgroup
of $W(L,\mf s)_u$, a basis $\Delta_u$ of $\Phi_u$, a parameter function $k^u : \Delta_u \to
\R_{\geq 0}$ and a 2-cocyle $\natural_u$ of $W(L,\mf s) / W (\Phi_u)$. In \cite{SolEnd} some of
these objects have a subscript $\sigma \otimes u$ instead of $u$, but since $W(L,\mf s)_u$ is
naturally isomorphic with $(W_{\mf s})_{\sigma \otimes u}$, we may omit $\sigma \otimes$.
To these data one can associate a twisted graded Hecke algebra
$\mh H (\mf t, W(L,\mf s)_u, k^u, \natural_u)$.
For any Levi subgroup $M$ of $G$ containing $L$, there is a parabolic subalgebra
$\mh H (\mf t, W(M,L,\mf s)_u, k^u, \natural_u)$, constructed in the same way.

\begin{thm}\textup{\cite[Corollary 8.1 and its proof, Proposition 9.5.a]{SolEnd}}
\label{thm:5.1}\\
There is an equivalence between the following categories:
\begin{itemize}
\item finite length $G$-representations, all whose irreducible subquotients have cuspidal
support in $(L,\sigma \otimes W(L,\mf s) U_u) = \{ (L,\sigma \otimes \chi) : \chi \in W(L,\mf s) U_u \}$,
\item finite length right $\mh H (\mf t, W(L,\mf s)_u, k^u, \natural_u)$-modules, all whose\\
$\mc O (\mf t)^{W(L,\mf s)_u}$-weights belong to $\log_u (U_u)$.
\end{itemize}
This equivalence of categories commutes with parabolic induction and preserves temperedness.
\end{thm}

The opposite algebra of $\mh H (\mf t, W(L,\mf s)_u, k^u, \natural_u)$ is naturally isomorphic to
\[
\mh H^G_u := \mh H (\mf t, W(L,\mf s)_u, k^u, \natural_u^{-1}).
\] 
via the simple map 
\begin{equation}\label{eq:5.1}
T_w f \mapsto T_{w^{-1}} f \qquad w \in W(L,\mf s)_u, , f \in \mc O (\mf t).
\end{equation}
Hence we may replace right $\mh H (\mf t, W(L,\mf s)_u, k^u, \natural_u)$-modules
by left $\mh H^G_u$-modules in Theorem \ref{thm:5.1}. We note that 
$Z(\mc H (G,K)^{\mf s}) \cong \mc O (X_\nr (L))^{W(L,\mf s)}$ corresponds to\\ 
$Z(\mh H^G_u) \cong \mc O (\mf t)^{W(L,\mf s)_u}$ via the maps $\exp_u$ and $\log_u$.
For later use we sketch the steps taken in \cite{SolEnd} to obtain the algebra $\mh H^G_u$ 
in Theorem \ref{thm:5.1}.

\begin{construct}\label{cons:1}
\begin{enumerate}[(i)]
\item Let $\Pi_{\mf s}$ be the progenerator of Rep$(G)^{\mf s}$ from \eqref{eq:4}. Then there are 
equivalences of categories
\begin{equation}\label{eq:4.1}
\Mod (\mc H (G,K)^{\mf s}) \cong \Rep (G)^{\mf s} \cong \Mod (\End_G (\Pi_{\mf s})^{op} ) .
\end{equation}
In particular the centres of the algebras $\mc H (G,K)^{\mf s}$ and $\End_G (\Pi_{\mf s})^{op}$ 
are canonically isomorphic.
\item Localize $\mr{End}_G (\Pi_{\mf s})^{op}$ at $W(L,\mf s) U_u$ by extending its centre
$\mc O (X_\nr (L))^{W(L,\mf s)}$ to the algebra of $W(L,\mf s)$-invariant analytic functions
$C^{an}(W(L,\mf s) U_u)^{W(L,\mf s)}$ on $W(L,\mf s) U_u$. We call the resulting algebra
$\mr{End}_G (\Pi_{\mf s})^{op}_{W(L,\mf s) U_u}$.
\item The maximal commutative subalgebra of $\mr{End}_G (\Pi_{\mf s})^{op}_{W(L,\mf s) U_u}$ is 
\[
C^{an}(W(L,\mf s) U_u) = \bigoplus\nolimits_{w \in W(L,\mf s) / W(L,\mf s)_u} C^{an}(w U_u) .
\]
The algebra $1_{U_u} \mr{End}_G (\Pi_{\mf s})^{op}_{W(L,\mf s) U_u} 1_{U_u}$ is a Morita 
equivalent subalgebra of $\mr{End}_G (\Pi_{\mf s})^{op}_{W(L,\mf s) U_u}$.
\item Localize $\mh H^G_u$ at $\log_u (U_u)$ by extending its centre $\mc O (\mf t)^{W(L,\mf s)_u}$
to the algebra $C^{an}(\log_u (U_u))^{W(L,\mf s)_u}$, and call the result $\mh H^G_{U_u}$.
\item Check that the above localizations do not change the categories of finite dimensional
modules with $\mc O (X_\nr L))$-weights (respectively $\mc O (\mf t)$-weights) in the set on 
which one localizes.
\item Show that the isomorphism 
\[
C^{an}(\log_u (U_u))^{W(L,\mf s)_u} \cong C^{an}(W(L,\mf s)_u U_u)^{W(L,\mf s)_u} \cong
C^{an}(W(L,\mf s) U_u)^{W(L,\mf s)}
\]
induced by $\exp_u$ induces an isomorphism
\begin{equation}\label{eq:5.16}
1_{U_u} \mr{End}_G (\Pi_{\mf s})^{op}_{W(L,\mf s) U_u} 1_{U_u} \cong \mh H^G_{U_u} .
\end{equation} 
\end{enumerate}
\end{construct}
To make full use of Theorem \ref{thm:5.1}, we also need a variation on step (iii) above. 
We will construct an algebra $\mh H^G_{W(L,\mf s) u}$ which is Morita equivalent with $\mh H^G_u$
and closer to $\End_G (\Pi_{\mf s})^{op}$ than $\mh H^G_u$. We start with
$\bigoplus_{w u \in W(L,\mf s)u} \mh H^G_{wu}$. In this algebra the unit element of
$\mh H^G_{wu}$ is denoted $e_{wu}$. For every element $wu \in W(L,\mf s)u$ we fix a $w$ which 
has minimal length in $w W(L,\mf s)_u$ (see \cite[end of \S 3]{SolEnd} for the definition of 
the length function). From \cite[Lemma 8.3]{SolEnd} we get an isomorphism
\[
\mr{Ad} (\mc T_w) : \mh H (\mf t, W(L,\mf s)_u, k^u, \natural_u) \to
\mh H (\mf t, W(L,\mf s)_{wu}, k^{wu}, \natural_{wu}) ,
\]
and hence also an isomorphism between their opposite algebras:
\[
\mr{Ad} (\mc T_w) : \mh H^G_u \to \mh H^G_{wu} .
\]
The advantage of this particular isomorphism comes from \cite[Lemma 8.3.b]{SolEnd}:
\[
\mr{Ad}(\mc T_w )^* : \Rep (\mh H^G_{wu}) \to \Rep (\mh H^G_u)
\]
intertwines Theorem \ref{thm:5.1} for $wu$ with Theorem \ref{thm:5.1} for $u$.
Ad$(\mc T_w)$ really is conjugation by an element $\mc T_w$ in a larger algebra, and satisfies:
\begin{align*}
& \mr{Ad}(\mc T_w) f = f \circ w^{-1} \quad \text{for } f \in \mc O (\mf t) = 
\mc O \big( T_u (X_\nr (L)) \big) ,\\
& \mr{Ad}(\mc T_w) \big|_{\C [W(L,\mf s)_u,\natural_u^{-1}]} = \text{ conjugation by }
T_w \text{ in } \C [W(L,\mf s), \natural_{\mf s}] .
\end{align*}
Here we use the 2-cocycle $\natural_{\mf s} = \natural^{-1}$ of $W(L,\mf s)$ from 
\cite[Lemma 5.7]{SolEnd}, which by \cite[Lemma 7.1]{SolEnd} extends $\natural_u^{-1}$.
As vector spaces we define
\[
\mh H^G_{W(L,\mf s) u} = \bigoplus\nolimits_{wu \in W(L,\mf s) u} \big( \mh H^G_{wu} \otimes
\bigoplus\nolimits_{\tilde w u \in W(L,\mf s) u} \C e_{wu} \mc T_w \mc T_{\tilde w}^{-1} 
e_{\tilde w u} \big) ,
\]
The multiplication of $\mh H^G_{W(L,\mf s) u}$ is given by
\begin{multline*}
(h_1 \otimes e_{w_1 u} \mc T_{w_1} \mc T_{w_2}^{-1} e_{w_2 u})
(h_2 \otimes e_{w_2 u} \mc T_{w_2} \mc T_{w_3}^{-1} e_{w_3 u}) = \\
h_1 \mr{Ad}(\mc T_{w_1}) \mr{Ad}(\mc T_{w_2})^{-1}(h_2) \otimes 
e_{w_1 u} \mc T_{w_1} \mc T_{w_3}^{-1} e_{w_3 u} ,
\end{multline*}
where $h_i \in \mh H^G_{w_i u}$ and all the $w_i$ are as chosen above. The elements 
$e_{wu} \mc T_w \mc T_{\tilde w}^{-1} e_{\tilde w u}$ of $\mh H^G_{W(L,\mf s)u}$ multiply like 
matrices with just one nonzero entry. It follows readily that 
\begin{equation}\label{eq:5.17}
\mh H^G_{W(L,\mf s)u} \cong M_{|W(L,\mf s)u|} \big( \mh H^G_u \big) .
\end{equation} 
We note that this algebra is of the form $\mh H (V,G,k,\natural)$, as in 
\cite[\S 2.3]{SolTwist}. The centre of this algebra is
\begin{multline}\label{eq:5.18}
Z \big( \mh H^G_{W(L,\mf s)u} \big) \cong \big( \bigoplus\nolimits_{wu \in W(L,\mf s)u} 
\mc O \big( T_{wu}(X_\nr (L)) \big) \big)^{W(L,\mf s)} \\
\cong \mc O \big( T_u (X_\nr (L)) \big)^{W(L,\mf s)_u} = Z \big( \mh H^G_u \big) .
\end{multline}
Let $\mh H^G_{W(L,\mf s) U_u}$ be the algebra obtained from $\mh H^G_{W(L,\mf s)u}$ by extending
its centre to
\[
\big( \bigoplus\nolimits_{wu \in W(L,\mf s)u} C^{an} ( \log_{wu} (w U_u) ) \big)^{W(L,\mf s)} .
\]
This algebra contains $\mh H^G_{U_u}$ as a Morita equivalent subalgebra, analogous to \eqref{eq:5.17}.

\begin{prop}\label{prop:5.7}
\enuma{
\item The diffeomorphism
\[
\sqcup_{wu \in W(L,\mf s)u} \exp_{wu} : \sqcup_{wu \in W(L,\mf s)u} \log_{wu}(w U_u) \to
W(L,\mf s) U_u
\]
induces an algebra isomorphism $\End_G (\Pi_{\mf s})^{op}_{W(L,\mf s) U_u} \cong
\mh H^G_{W(L,\mf s) U_u}$. That fits in a commutative diagram
\[
\begin{array}{ccccc}
\End_G (\Pi_{\mf s})^{op}_{W(L,\mf s) U_u} & \cong & \mh H^G_{W(L,\mf s) U_u} &
\longleftarrow & \mh H^G_{W(L,\mf s) u} \\
\uparrow & & \uparrow & & \uparrow \\
1_{U_u} \End_G (\Pi_{\mf s})^{op}_{W(L,\mf s) U_u} 1_{U_u} & \cong & \mh H^G_{U_u} &
\longleftarrow & \mh H^G_u 
\end{array}.
\]
\item In this diagram the vertical arrows are inclusions of Morita equivalent subalgebras and
each of the two horizontal arrows induces an equivalence between the categories of finite length
modules all whose weights for \eqref{eq:5.18} belong to $\sqcup_{wu \in W(L,\mf s)u} \log_{wu}(w U_u)$. 
}
\end{prop}
\begin{proof}
(a) The elements $\mc T_w$ involved in $\mh H^G_{W(L,\mf s)u}$ stem from \cite[\S 5]{SolEnd}.
It was shown in the proof of \cite[Lemma 8.3]{SolEnd} that 
\[
\mc T_w 1_{U_u} \in \End_G (\Pi_{\mf s})^{op}_{W(L,\mf s) U_u}
\]
and that
\[
\mr{Ad}(\mc T_w) : 1_{U_u} \End_G (\Pi_{\mf s})^{op}_{W(L,\mf s) U_u} 1_{U_u} \to
1_{w U_u} \End_G (\Pi_{\mf s})^{op}_{W(L,\mf s) U_u} 1_{w U_u}
\]  
is an algebra isomorphism. It follows that the isomorphism \eqref{eq:5.16} extends canonically
to the required isomorphism.\\
(b) The vertical arrows were already discussed before. The claim about the lower horizontal arrow
was shown in \cite[Lemma 7.2.a]{SolEnd}, based on \cite[Proposition 4.3]{Opd}. The same argument
applies to the upper horizontal arrow.
\end{proof}

We define the parabolic subalgebras of $\mh H^G_u$ to be the analogous algebras 
\[
\mh H^M_u = \mh H (\mf t, W(M,L,\mf s)_u, k^u, \natural_u^{-1}) ,
\]
constructed from Levi subgroups $M$ of $G$ containing $L$, as in \cite[Lemma 7.2.b]{SolEnd}. 
The translation from right to left modules via \eqref{eq:5.1} commutes with parabolic induction.
Namely, for a right $\mh H (\mf t, W(M,L,\mf s)_u, k^u, \natural_u)$-module $V$ there is a
natural isomorphism of $\mh H^G_u$-modules
\begin{equation}\label{eq:5.2}
\begin{array}{ccc}
V \underset{\mh H (\mf t, W(M,L,\mf s)_u, k^u, \natural_u)}{\otimes} \hspace{-11mm}
\mh H (\mf t, W(L,\mf s)_u, k^u, \natural_u) & \to & \mh H^G_u \underset{\mh H^M_u}{\otimes} V \\
v \otimes T_w & \mapsto & T_{w^{-1}} \otimes v
\end{array} \quad v \in V, w \in W(L,\mf s)_u .
\end{equation}
The subalgebras $\mh H^M_u$, with
\[
\mf t^M = \C \otimes_\Z X^* (M) \quad \text{and} \quad
\mf t_M = \C \otimes_\Z X^* (L \cap M_{\mr{der}}) ,
\] 
fulfill the conditions from \cite[pages 13--14]{SolTwist}. Indeed, that follows 
from \eqref{eq:5.2}, Theorem \ref{thm:5.1} and the properties of elliptic $G$-representations 
discussed at the start of Section \ref{sec:familiesG}.

Let $\mf F (M,\eta)$ be an algebraic family in $\Rep (G)^{\mf s}$, with $\eta$ irreducible
and elliptic. We may and will assume that $M$ is standard and we let $P$ be the unique
standard parabolic subgroup of $G$ with Levi factor $M$.
All the representations $I_P^G (\eta \otimes \chi)$ with $\chi \in X_\nr (M)$
admit a central character, so $Z(\mc H (G,K)^{\mf s})$ acts by a character on 
$I_P^G (\eta \otimes \chi)^K$. With \cite[Lemma 2.3]{SolTwist} we see that 
\[
\mc F_{M,\delta} : \mc H (G,K)^{\mf s} \to 
\mc O (X_\nr (M)) \otimes \mr{End}_\C (I_P^G (V_\delta)^K) \\
\]
is a homomorphism of $Z(\mc H (G,K)^{\mf s})$-algebras and that $HH_n (\mc F_{M,\delta})$ is
a homomorphism of $Z(\mc H (G,K)^{\mf s})$-modules.

\hspace{-4mm} Assume that some members of $\mf F (M,\eta)$ have cuspidal support in 
$(L,\sigma \otimes W(L,\mf s) U_u)$. Then the image of $\mf F (M,\eta)$ under Theorem \ref{thm:5.1} 
is an algebraic family $\mf F (M,\tilde \eta)$ of $\mh H^G_u$-modules, where $\tilde \eta
\in \Irr (\mh H^M_u)$ is elliptic and tempered. More precisely Theorem \ref{thm:5.1} 
only applies to an open part of $\mf F (M,\eta)$, and the image of that is the part of
$\mf F (M,\tilde \eta)$ with $\mc O (\mf t)^{W(L,\mf s)_u}$-weights in $\log_u (U_u)$.
By the Langlands classification (for graded Hecke algebras in \cite{Eve}, generalized to our 
setting with the method from \cite[\S 2.2]{SolAHA}) every such family of
$\mh H^G_u$-modules arises from an elliptic representation of 
\begin{equation}\label{eq:5.12}
\mh H_{u,M} = \mh H (\mf t_M, W(M,L,\mf s)_u, k^u, \natural_u^{-1}) .
\end{equation}
Hence there exists a $\lambda \in i \R \otimes_\Z X^* (M)$ such that $\mc O (\mf t^M) \subset
\mh H^M_u$ acts on $\C_\lambda \otimes \tilde \eta$ by evaluation at 0. We may replace 
$\tilde \eta$ by $\C_\lambda \otimes \tilde \eta$ without changing $\mf F (M,\tilde \eta)$.
Then $\tilde \eta$ has $\mc O (\mf t_M)$-weights in 
\begin{equation}\label{eq:5.3}
\R \Delta_{M,u} \subset \R \otimes_\Z X^* (L \cap M_{\mr{der}}) .
\end{equation}
In general the full structure of the algebra $\End_G (\Pi_{\mf s})$ (or its opposite) seems to
be rather complicated. Fortunately, it can be approximated with simpler algebras. The normalized
parabolic induction functor $I_{P_0 L}^G$ gives an embedding 
\[
\mc O (X_\nr (L)) \to \End_G (\Pi_{\mf s}) .
\]
From \eqref{eq:4.1} and \eqref{eq:4.2} we know that 
\[
Z (\End_G (\Pi_{\mf s})) \cong \mc O (X_\nr (L))^{W(L,\mf s)} .
\]
Let $\C (X_\nr (L))$ be the quotient field of $\mc O (X_\nr (L))$, i.e. the field of rational
functions on the complex affine variety $X_\nr (L)$. It is easy to see that the multiplication map
\begin{equation}\label{eq:4.3}
\C (X_\nr (L))^{W(L,\mf s)} \underset{\mc O (X_\nr (L))^{W(L,\mf s)}}{\otimes}
\mc O (X_\nr (L)) \longrightarrow \C (X_\nr (L))
\end{equation}
is a field isomorphism. According to \cite[Corollary 5.8]{SolEnd}, \eqref{eq:4.3} extends to 
an algebra isomorphism
\begin{equation}\label{eq:4.4}
\C (X_\nr (L))^{W(L,\mf s)} \underset{\mc O (X_\nr (L))^{W(L,\mf s)}}{\otimes}
\End_G (\Pi_{\mf s}) \isom \C (X_\nr (L)) \rtimes \C [W(L,\mf s),\natural] .
\end{equation}
With \eqref{eq:5.1} we also obtain the opposite version
\begin{equation}\label{eq:4.5}
\C (X_\nr (L))^{W(L,\mf s)} \underset{\mc O (X_\nr (L))^{W(L,\mf s)}}{\otimes}
\End_G (\Pi_{\mf s})^{op} \isom \C (X_\nr (L)) \rtimes \C [W(L,\mf s),\natural_{\mf s}] .
\end{equation}
Unfortunately the isomorphisms \eqref{eq:4.4} and \eqref{eq:4.5} are not canonical, they depend
on the choice of a suitable $\sigma \in \Irr (L)^{\mf s}$ and on the normalization of
certain intertwining operators. In the remainder of this paragraph we fix those choices.

We emphasize that (except in very special cases)
\begin{equation}\label{eq:4.6}
\End_G (\Pi_{\mf s})^{op} \text{ is not isomorphic with }
\mc O (X_\nr (L)) \rtimes \C [W(L,\mf s),\natural_{\mf s}] .
\end{equation}
Remarkably, it turns out that nevertheless there is a canonical bijection
\begin{equation}\label{eq:4.7}
\zeta^\vee : R(G)^{\mf s} \cong R(\mc H (G,K)^{\mf s}) \longrightarrow
R \big( \mc O (X_\nr (L)) \rtimes \C [W(L,\mf s),\natural_{\mf s}] \big) .
\end{equation}
We describe step-by-step how it is obtained.
\begin{construct}\label{cons:2}
\begin{enumerate}[(i)]
\item With the equivalences of categories \eqref{eq:4.1} we go from $R(G)^{\mf s}$ to
$R (\End_G (\Pi_{\mf s})^{op})$.
\item By decomposing finite length $\End_G (\Pi_{\mf s})^{op}$-modules along their
$\mc O (X_\nr (L))^{W(L,\mf s)}$-weights, it suffices to consider $G$-representations $\pi$
as in Theorem \ref{thm:5.1}. 
\item Via Theorem \ref{thm:5.1} and \eqref{eq:5.1} we obtain the $\mh H_u^G$-module
$1_{U_u} \Hom_G (\Pi_{\mf s}, \pi)$.
\item There is a canonical $\Z$-linear bijection
\[
\zeta_u^\vee : R \big( \mh H_u^G \big) \to 
R \big( \mc O (\mf t) \rtimes \C [W(L,\mf s)_u,\natural_u^{-1}] \big) .
\]
The construction is given in \cite[Theorem 2.4]{SolHomAHA}, while the bijectivity follows
from \cite[Theorem 1.9]{SolK}.
\item The map $\exp_u$ provides a diffeomorphism $\R \otimes_\Z X^* (L) \to  u X_\nr^+ (L)$. 
Translate the action of $\mc O (\mf t)$ on $\zeta_u^\vee \big( 1_{U_u} \Hom_G (\Pi_{\mf s}, \pi) 
\big)$ to an action of $\mc O (X_\nr (L))$, by first replacing $\mc O (\mf t)$ and 
$\mc O (X_\nr (L))$ by analytic functions and then pullback along $\exp_u$. 
This is similar to steps (ii)--(vi) from Construction \ref{cons:1}, and results in a 
$\mc O (X_\nr (L)) \rtimes \C [W(L,\mf s)_u,\natural_u^{-1}]$-module with all
$\mc O (X_\nr (L))$-weights in $u X_\nr^+ (L)$. 
\item From \cite[Lemma 7.1]{SolEnd} we get a canonical algebra isomorphism
\[
\C [W(L,\mf s)_u, \natural_u^{-1}] \longrightarrow \C [W(L,\mf s)_u ,\natural_{\mf s}] .
\]
With that we define
\[
\zeta^\vee (\pi) = \mr{ind}_{\mc O (X_\nr (L)) \rtimes 
\C [W(L,\mf s)_u,\natural_u^{-1}]}^{\mc O (X_\nr (L)) \rtimes \C [W(L,\mf s),\natural_{\mf s}]} 
\zeta_u^\vee \big( 1_{U_u} \Hom_G (\Pi_{\mf s}, \pi) \big) .
\]
\end{enumerate}
\end{construct}

\begin{thm}\label{thm:4.1}
The map $\zeta^\vee$ from \eqref{eq:4.7} has the following properties.
\enuma{
\item $\zeta^\vee$ is $\Z$-linear and bijective.
\item $\pi \in R(G)^{\mf s}$ is tempered if and only if $\zeta^\vee (\pi)$ is tempered
(i.e. all its $\mc O (X_\nr (L))$-weights lie in $X_\unr (L)$).
\item If all the irreducible subquotients of $\pi \in R(G)^{\mf s}$ have cuspidal support
in $\sigma \otimes W(L,\mf s) u X_\nr^+ (L)$, then all $\mc O (X_\nr (L))$-weights of 
$\zeta^\vee (\pi)$ lie in $W(L,\mf s) u X_\nr^+ (L)$.
\item In the setting of (c), suppose that $\pi$ is tempered. Then
\[
\zeta^\vee (\pi) = \mr{ind}_{\mc O (X_\nr (L)) \rtimes \C [W(L,\mf s)_u,\natural_{\mf s}]}^{
\mc O (X_\nr (L)) \rtimes \C [W(L,\mf s),\natural_{\mf s}]} (\C_u \otimes \pi_u) ,
\]
where $\pi_u$ denotes the restriction of $1_{U_u} \Hom_G (\Pi_{\mf s}, \pi)$
to $\C [W(L,\mf s)_u, \natural_{\mf s}]$.
\item $\zeta^\vee$ commutes with parabolic induction and unramified twists, in the sense that
\[
\zeta^\vee (I_P^G  (\chi \otimes \tau)) =
\mr{ind}_{\mc O (X_\nr (L)) \rtimes \C [W(M,L,\mf s),\natural_{\mf s}]}^{
\mc O (X_\nr (L)) \rtimes \C [W(L,\mf s),\natural_{\mf s}]} (\chi \otimes \zeta_M^\vee (\tau)) 
\]
for tempered $\tau \in R(M)^{\mf s}$ and $\chi \in X_\nr (M)$.
}
\end{thm}
\begin{proof}
(a) Since each step in Construction \ref{cons:2} is $\Z$-linear and bijective, so is $\zeta^\vee$.
The bijectivity of (vi) comes from the Morita equivalence between
\[
C^{an}(U_u) \rtimes \C [W(L,\mf s)_u, \natural_{\mf s}] \quad \text{and} \quad
C^{an}(W(L,\mf s) U_u) \rtimes \C [W(L,\mf s), \natural_{\mf s}] .
\]
(b) By Theorem \ref{thm:5.1}, steps (i)--(iii) respect temperedness. It is known from
\cite[Theorem 2.4]{SolHomAHA} that $\zeta_u^\vee$ in (iv) respects temperedness, and for (v)
that is obvious because the $\mc O (X_\nr (L))$-weights are not changed in that step. Modules
of $\mc O (X_\nr (L)) \rtimes \C [\Gamma,\natural_{\mf s}]$, for any subgroup $\Gamma$ of
$W(L,\mf s)$, are tempered if and only if all their $\mc O (X_\nr (L))$-weights lie in
$X_\unr (L)$. The functor $\mr{ind}_{\mc O (X_\nr (L)) \rtimes \C [W(L,\mf s)_u,
\natural_{\mf s}]}^{\mc O (X_\nr (L)) \rtimes \C [W(L,\mf s),\natural_{\mf s}]}$ respects this,
so step (vi) does that as well.\\
(c) Step (i) translates ``cuspidal support $(L,\sigma \otimes W(L,\mf s) \chi)$" into
``all $\mc O (X_\nr (L))$-weights in $W(L,\mf s) \chi$". After that the only step that changes 
the $\mc O (X_\nr (L))$-weights is (iv), and by \cite[Theorem 2.4.(3)]{SolHomAHA} it only
adjusts $\mc O (X_\nr (L))$-weights by elements of $X_\nr^+ (L)$.\\
(d) The only tricky point is to see that step (iv) of Construction \ref{cons:2} sends 
\[
1_{U_u} \Hom_G (\Pi_{\mf s}, \pi) \quad \text{to} \quad \C_0 \otimes \big( 1_{U_u} 
\Hom_G (\Pi_{\mf s}, \pi) \big) \big|_{\C [W(L,\mf s)_u,\natural_u^{-1}]} ,
\]
where $\C_0$ means that $\mc O (\mf t)$ acts via evaluation at $0 \in \mf t$.
That is the content of \cite[Theorem 2.4.(4)]{SolHomAHA}.\\
(e) First we check that $\zeta^\vee$ respects parabolic induction, at least when the input
is a tempered (virtual) representation tensored with an unramified character.
Steps (i) and (ii) commute with parabolic induction by \cite[Condition 4.1 and Lemma 
6.1]{SolComp}. For step (iii) that follows from \cite[Lemma 6.6 and Proposition 7.3]{SolEnd}.
In step (iv), property (e) is an important part of the construction of $\zeta_u^\vee$ in
\cite[Theorem 2.4]{SolHomAHA}, that is where we need the shape of the input. That step
(v) respects parabolic induction follows from the properties of the isomorphism between the
analytically localized versions of the involved algebras, as in \cite[Proposition 7.3]{SolEnd}.
For step (vi) we obtain the desired behaviour from \cite[Lemma 6.6]{SolEnd}.

The compatibility with unramified twists requires an explicit computation. By the above it
suffices to check that
\begin{align}\label{eq:4.8}
& \zeta^\vee (\chi \otimes \pi) = \chi \otimes \zeta^\vee (\pi) \qquad
\pi \in R^t (G)^{\mf s}, \chi \in X_\nr (G) ,\\
\nonumber \text{where} \quad & (\chi \otimes \zeta^\vee (\pi)) (f T_w) = 
f(\chi) \zeta^\vee (\pi) (f T_w) \qquad f \in \mc O (X_\nr (G)), w \in W(L,\mf s) .
\end{align} 
In steps (i) and (ii), \eqref{eq:4.8} holds by construction, see \cite[\S 6]{SolEnd}. 
In step (iii) the appropriate version of \eqref{eq:4.8} is known from 
\cite[Theorem 2.4.(2)]{SolHomAHA}. Consider the tempered $\End_G (\Pi_{\mf s})^{op}$-module
$\tau := \Hom_G (\Pi_{\mf s},\pi)$. By step (ii) we may assume that all $\mc O (X_\nr (L))$-weights
of $\tau$ lie in $W(L,\mf s) u X_\nr^+ (L)$. Step (iii), with $\chi \, |\chi|^{-1} u$ in the role
of $u$, sends $\chi \otimes \tau$ to 
\begin{equation}\label{eq:4.9}
1_{U_{\chi \, |\chi|^{-1} u}} (\chi \otimes \tau) = 
1_{U_{\chi \, |\chi|^{-1} u}} (|\chi| \otimes \chi \, |\chi|^{-1} \otimes \tau) =
\log |\chi| \otimes 1_{U_{\chi |\chi|^{-1}u}} (\chi \, |\chi|^{-1} \otimes \tau) 
\end{equation}
with $\log |\chi| \in \mf t_\R^{W(L,\mf s)}$ as character of $\mc O (\mf t)$. Step (iv)
transforms \eqref{eq:4.9} to
\[
\log |\chi| \otimes \big( 1_{U_{\chi \, |\chi|^{-1}u}} (\chi \, |\chi|^{-1} \otimes \tau) \big) 
\big|_{\C [W(L,\mf s)_{\chi \, |\chi|^{-1} u}, \natural_{\mf s}]} =
\log |\chi| \otimes ( 1_{U_u} \tau ) \big|_{\C [W(L,\mf s)_u, \natural_u^{-1}]} ,
\]
where the equality holds because $\chi$ is $W(L,\mf s)$-invariant. Now step (v), again with 
respect to $\chi \, |\chi|^{-1} u$, yields 
\[
|\chi| \otimes \chi \, |\chi|^{-1} u \otimes ( 1_{U_u} \tau ) \big|_{\C [W(L,\mf s)_u, 
\natural_u^{-1}]} = \C_{\chi u} \otimes ( 1_{U_u} \tau ) \big|_{\C [W(L,\mf s)_u, \natural_u^{-1}]} .
\]
Finally, using the $W(L,\mf s)$-invariance of $\chi$ again, step (vi) returns
\begin{multline*}
\mr{ind}_{\mc O (X_\nr (L)) \rtimes \C [W(L,\mf s)_u,\natural_u^{-1}]}^{\mc O (X_\nr (L)) \rtimes 
\C [W(L,\mf s),\natural_{\mf s}]} \big( \C_{\chi u} \otimes ( 1_{U_u} \tau ) 
\big|_{\C [W(L,\mf s)_u, \natural_u^{-1}]} \big) = \\
\chi \otimes \mr{ind}_{\mc O (X_\nr (L)) \rtimes \C [W(L,\mf s)_u,\natural_u^{-1}]}^{
\mc O (X_\nr (L)) \rtimes \C [W(L,\mf s),\natural_{\mf s}]} \big( \C_u \otimes ( 1_{U_u} \tau ) 
\big|_{\C [W(L,\mf s)_u, \natural_u^{-1}]} \big) .
\end{multline*}
Similar computations show that the last expression equals $\chi \otimes \zeta^\vee (\pi)$.
\end{proof}

The properties listed in Theorem \ref{thm:4.1} imply for instance that $\zeta^\vee$ maps algebraic
families in $R(G)^{\mf s}$ (or equivalently in $R(\mc H (G,K)^{\mf s})$ to algebraic families
in $R \big( \mc O (X_\nr (L)) \rtimes \C [W(L,\mf s),\natural_{\mf s}] \big)$.

\subsection{Local descriptions of Hochschild homology} \ 
\label{par:HHHG}

In this section we will determine The Hochschild homology of $\mc H (G)$ with a method 
based on the families of $G$-representations from Section \ref{sec:familiesG}. 
With the procedure from page \pageref{eq:3.13} we pick a finite number (say $n_{\mf s}$) of
algebraic families of $G$-representations $\mf F (M_i,\eta_i)$, such that the $\eta_i$
are irreducible and the members of these families span $\Q \otimes_\Z R (G)^{\mf s}$ in a
minimal way. We may assume that each $M_i$ is the standard Levi factor of a standard 
parabolic subgroup $P_i$ of $G$. Writing 
\begin{equation}\label{eq:5.5}
\mc F_{\mf s} = \bigoplus\nolimits_{i=1}^{n_{\mf s}} \mc F_{M_i,\eta_i}
\end{equation}
we obtain a homomorphism of $Z(\mc H (G,K)^{\mf s})$-modules
\[
HH_n (\mc F_{\mf s}) : HH_n (\mc H (G,K)^{\mf s}) \to 
\bigoplus\nolimits_{i=1}^{n_{\mf s}} \Omega^n (X_\nr (M_i)) ,
\]
where $Z(\mc H (G,K)^{\mf s})$ acts on the right hand side via the central characters of
the involved representations $\pi (M_i,\eta_i,\chi_i)$.

We aim to establish an analogue of \cite[Theorems 1.13 and 2.8]{SolTwist} for 
$HH_n (\mc F_{\mf s})$. The families that have no cuspidal supports in $\sigma \otimes
W(L,\mf s)U_u$ can be ignored for the current purposes (we may call them $U_u$-irrelevant).
For the remaining families, as explained above we may assume without loss of generality that 
\eqref{eq:5.3} holds. Select $\chi_{\eta_i} \in u X_\unr (M_i)$ such that  
\begin{equation}\label{eq:5.4}
\mb{Sc}(\eta_i) \in \sigma \otimes W(M,L,\mf s) \chi_{\eta_i} X_\nr^+ (L) .
\end{equation}
By Theorem \ref{thm:5.1} the algebraic families of $\mh H^G_u$-representations 
$\mf F (M_i,\tilde \eta_i)$ span the part of $\Q \otimes_\Z R( \mh H^G_u)$ with 
$\mc O (\mf t)^{W(L,\mf s)_u}$-weights in $\log_u (U_u)$. As $\log_u (U_u)$ contains
$\mf t_\R = \R \otimes_\Z X^* (L)$ and is open in $\mf t$, the geometric structure of 
$\Irr (\mh H^G_u)$ \cite[\S 11]{SolGHA} entails that the $\mf F (M_i, \tilde \eta_i)$ span the 
whole of $\Q \otimes_\Z R( \mh H^G_u )$. Thus we are in the setting of 
\cite[Lemma 4.2--Theorem 4.8]{SolTwist}.

For $g \in W(L,\mf s)_u$ and $v \in \mf t^g$, in \cite[(1.19)]{SolTwist} an element 
\[
\nu_{g,v} \in \C \otimes_\Z R \big( \mc O (X_\nr (L)) \rtimes \C [W(L,\mf s),\natural_{\mf s}] \big)
\]
was defined, as evaluation at $(g,v)$ in one picture of $HH_0 \big( \mc O (X_\nr (L)) \rtimes 
\C [W(L,\mf s),\natural_{\mf s}] \big)$. As in \cite[(2.8)]{SolTwist}, applying 
$(\zeta_u^\vee)^{-1}$ produces a virtual $\mh H^G_u$-representation 
\begin{equation}\label{eq:4.10}
\nu^1_{g,v} = (\zeta_u^\vee)^{-1} \nu_{g,v} = 
\sum\nolimits_{i = 1, U_u\text{-rel}}^{n_{\mf s}} \lambda_{g,i} \mr{tr}\, 
\pi (M_i, \tilde \eta_i, \phi_{g,i}(v)) \qquad g \in W(L,\mf s)_u, v \in \mf t^g ,
\end{equation}
which also occurs in \cite[(2.10)]{SolTwist}. Here the $U_u$-irrelevant indices $i$ are left out 
of the sum, but we may still include by setting $\lambda_{g,i} = 0$ for those $i$. From 
\cite[Lemma 1.10]{SolTwist} we know that each $\phi_{g,i} : \mf t^g \to \mf t^{M_i}$ is given 
by an element of $W(L,\mf s)_u$. Hence $\phi_{g,i}$ induces regular maps
\begin{equation}\label{eq:5.6}
\begin{aligned}
& \phi_{g,i} : u X_\nr (L)^{g,\circ} = \exp_u (\mf t^g) \to \exp_u (\mf t^{M_i}) = u X_\nr (M_i) ,\\
& \chi_{\eta_i}^{-1} \phi_{g,i} : u X_\nr (L)^{g,\circ} \to X_\nr (M_i) .
\end{aligned}
\end{equation}
The $\phi_{g,i}$ are not defined when $i$ is $U_u$-relevant. Using \eqref{eq:5.6} we put
\[
\nu^1_{g,u'} = \sum\nolimits_{i = 1,U_u\text{-rel}}^{n_{\mf s}} \lambda_{g,i} \mr{tr} \, 
\pi (M_i, \eta_i, \chi_{\eta_i}^{-1} \phi_{g,i}(u')) \qquad g \in W(L,\mf s)_u, u' \in U_u .
\] 
Since the right hand side is well-defined for any $u' \in u X_\nr (L)^{g,\circ}$,
we may extend the definition of $\nu^1_{g,u'}$ to such $u'$. The map \eqref{eq:5.6}
induces a homomorphism of $\mc O (X_\nr (L))^{W (L,\mf s)}$-algebras
\begin{equation}\label{eq:5.15}
\chi_{\eta_i}^{-1} \phi_{g,i}^* : \mc O (X_\nr (M_i)) \otimes \End_\C (I_{P_i}^G (V_{\eta_i})^K) 
\to \mc O (u X_\nr (L)^{g,\circ}) \otimes \End_\C (I_{P_i}^G (V_{\eta_i})^K) .
\end{equation}
Here $\mc O (X_\nr (L))^{W (L,\mf s)}$ acts on the domain via the central characters of the 
members of $\mf F (M_i,\eta_i)$, whereas the $\mc O (X_\nr (L))^{W (L,\mf s)}$-module structure 
on the range is given at $\chi \in u X_\nr (L)^{g,\circ})$ by $W(L,\mf s) \chi t_{\eta_i}^+$,
where the central character of $\eta_i$ is represented by $\chi_{\eta_i} t_{\eta_i}^+$ 
with $t_{\eta_i}^+ \in X_\nr^+ (L)$. In other words, the natural module structure on the right 
hand side of \eqref{eq:5.15} is adjusted by the positive part of the central character of $\eta_i$.
When we consider the map on Hochschild homology induced by \eqref{eq:5.15}, the range does not
depend on $i$, but the $\mc O (X_\nr (L))^{W (L,\mf s)}$-module structure still does.
Like in \cite[(1.33) and (2.14)]{SolTwist}, we can combine the maps on Hochschild homology 
induced by the homomorphisms \eqref{eq:5.15} a $\C$-linear map
\begin{multline}\label{eq:5.25}
HH_n (\phi_u^* ) = \bigoplus\nolimits_{g \in \langle W(L,\mf s)_u \rangle} \sum\nolimits_{i = 1, 
U_u\text{-rel}}^{n_{\mf s}} \lambda_{g,i} HH_n (\chi_{\eta_i}^{-1} \phi^*_{g,i}) : \\
\sum\nolimits_{i = 1, U_u\text{-rel}}^{n_{\mf s}} \Omega^n (X_\nr (M_i))
\to \bigoplus\nolimits_{g \in \langle W(L,\mf s)_u \rangle} \Omega^n (u X_\nr (L)^{g,\circ}) .
\end{multline}
The maps $HH_n (\phi_u^*)$, 
for various $u \in X_\unr (L)$, are our main tools to describe $HH_n (\mc H (G,K)^{\mf s})$.

Recall that the formal completion of a commutative algebra $A$ with respect to a finite set
of characters $X$ is denoted $\widehat{A}_X$. With that notation, for 
$u' \in U_u$ there are algebra isomorphisms
\begin{equation}\label{eq:5.8}
\widehat{\mc O (X_\nr (L))}_{W(L,\mf s) u'}^{W(L,\mf s)} \cong 
\widehat{\mc O (X_\nr (L))}_{u'}^{W(L,\mf s)_{u'}} \!
\cong \widehat{\mc O (\mf t)}_{\log_u (u')}^{W(L,\mf s)_{u'}} \cong
\widehat{\mc O (\mf t)}_{W(L,\mf s)_u \log_u (u')}^{W(L,\mf s)_u} . \!
\end{equation}

\begin{prop}\label{prop:5.2}
For $u' \in U_u$ the following modules over the formal completion \eqref{eq:5.8} are isomorphic:
\enuma{
\item $\widehat{\mc O (X_\nr (L))}_{u'}^{W(L,\mf s)_{u'}} \underset{Z(\mc H (G,K)^{\mf s})}{\otimes}
HH_n (\mc H (G,K)^{\mf s})$, 
\item $\widehat{\mc O (\mf t)}_{\log_u (u')}^{W(L,\mf s)_{u'}} 
\underset{Z(\mh H^G_{W(L,\mf s)u)}}{\otimes} HH_n (\mh H^G_{W(L,\mf s)u})$,
\item $\widehat{\mc O (\mf t)}_{\log_u (u')}^{W(L,\mf s)_{u'}} \underset{Z(\mh H^G_u)}{\otimes}
HH_n (\mh H^G_u)$,
\item $\widehat{\mc O (\mf t)}_{\log_u (u')}^{W(L,\mf s)_{u'}} \underset{\mc O (\mf t)^{W(L,\mf s)_u}
}{\otimes} HH_n (\phi^*)^{-1} \Big( \bigoplus\limits_{g \in \langle W(L,\mf s)_u \rangle} 
\big( \Omega^n (\mf t^g) \otimes \natural_{\mf s}^g \big)^{Z_{W(L,\mf s)_u} (g)} \Big)$,
\item $\widehat{\mc O (X_\nr (L))}_{u'}^{W(L,\mf s)_{u'}} \hspace{-8mm} 
\underset{\mc O (X_\nr (L))^{W(L,\mf s)}}{\otimes} \hspace{-6mm}
HH_n (\phi_u^*)^{-1} \Big( \bigoplus\limits_{g \in \langle W(L,\mf s)_u \rangle} \hspace{-6mm} 
\big( \Omega^n (u X_\nr (L)^{g,\circ}) \otimes \natural_{\mf s}^g \big)^{Z_{W(L,\mf s)_u} (g)} \Big)$.
}
The isomorphism between (a) and (d) is induced by $HH_n (\mc F_{\mf s})$.
\end{prop}  
The character $\natural_{\mf s}^g : Z_{W(L,\mf s)_u} (g) \to \C^\times$ figuring in parts (d) and (e) 
is defined as
\[
\natural_{\mf s}^g (h) = T_g T_h T_g^{-1} T_h^{-1} \qquad h \in Z_{W(L,\mf s)_u} (g) .
\]
It extends naturally to a map $W(L,\mf s) \to \C [W(L,\mf s)), \natural_{\mf s}]$ with 
good properties, see \cite[Lemma 1.3]{SolTwist}.
\begin{proof}
Recall the explanation of Theorem \ref{thm:5.1} between \eqref{eq:5.1} and \eqref{eq:5.15}.
By the Morita invariance of Hochschild homology \cite[\S 1.2]{Lod}:
\begin{equation}\label{eq:5.7}
HH_n (\mc H (G,K)^{\mf s}) \cong HH_n (\End_G (\Pi_{\mf s})^{op} ) \qquad
\text{as } Z(\mc H (G,K)^{\mf s})\text{-modules.}
\end{equation}
Let $I_{u'} \subset \mc O (X_\nr (L))$ be the maximal ideal of functions vanishing at $u'$. As
\[
\mc O (X_\nr (L)) / I_{u'}^m \cong C^{an}(W(L,\mf s)U_u) / I_{u'}^m C^{an}(W(L,\mf s) U_u)
\] 
for any $m \in \N$, the algebras $\mc O (X_\nr (L))^{W(L,\mf s)}$ and 
$C^{an}(W(L,\mf s) U_u)^{W(L,\mf s)}$ have the same formal completion at $u'$. It follows that
in the process described between  \eqref{eq:5.1} and \eqref{eq:5.16} the analytic localization 
steps do not change the formal completions of the involved algebras (at $u'$ and $\log_u (u')$
respectively). Then Proposition \ref{prop:5.7} yields the isomorphism between (a),(b) and (c).

The isomorphism between (c) and (d) is a consequence of \cite[Theorem 2.8]{SolTwist}.
As $\mf F (M_i,\tilde \eta_i)$ is constructed from $\mf F (M_i,\eta_i)$ via Theorem 
\ref{thm:5.1}, \eqref{eq:5.8} induces isomorphisms of 
$\widehat{\mc O (\mf t)}_{\log_u (u')}^{W(L,\mf s)_{u'}}$-modules 
\begin{multline*}
\widehat{\mc O (\mf t)}_{\log_u (u')}^{W(L,\mf s)_{u'}} \underset{\mc O (\mf t)^{W(L,\mf s)_u}}
{\otimes} \bigoplus_{i = 1, U_u\text{-rel}}^{n_{\mf s}} \Omega^n (\mf t^{M_i}) \cong \\
\widehat{\mc O (X_\nr (L))}_{u'}^{W(L,\mf s)_{u'}} \underset{ \mc O (X_\nr (L))^{W(L,\mf s)}}
{\otimes} \bigoplus_{i = 1,U_u\text{-rel}}^{n_{\mf s}} \Omega^n (u X_\nr (M_i)) . 
\end{multline*}
By Theorem \ref{thm:5.1} this restricts to an isomorphism between (d) and (e).

The isomorphism between (c) and (d) is obtained by evaluating elements of $\mh H^G_u$ at the
families $\mf F (M_i, \tilde \eta_i)$. Hence the isomorphism between 
$HH_n (\End_G (\Pi_{\mf s})^{op} )$ and (d) comes from evaluating elements of
$\End_G (\Pi_{\mf s})^{op}$ at the same algebraic families. When we pass from (d) to (e),
the families $\mf F (M_i, \tilde \eta_i)$ are translated to the families 
$\mf F (M_i, \eta_i)$. The isomorphism between (a) and (e) can be constructed from that 
between $HH_n (\End_G (\Pi_{\mf s})^{op} )$ and (e) by composing with \eqref{eq:5.7}, which is 
induced by a Morita equivalence. Thus the isomorphism between (a) and (e) is given by evaluating
$\mc H (G,K)^{\mf s}$ at the families $\mf F (M_i,\eta_i)$. In other words, it is given
by $HH_n (\mc F_{\mf s})$, while ignoring the $U_u$-irrelevant families.
\end{proof}

We will lift Proposition \ref{prop:5.2} to a statement about $HH_n (\mc H (G,K)^{\mf s})$
on the whole of $X_\nr (L)$.

\begin{lem}\label{lem:5.6}
The map $HH_n (\mc F_{\mf s})$ is an injection from $HH_n (\mc H (G,K)^{\mf s})$ to the set
of $\omega \in\bigoplus_{i=1}^{n_{\mf s}} \Omega^n (X_\nr (M_i))$ such that
\[
HH_n (\phi_u^*) \omega \in \bigoplus\nolimits_{g \in \langle W(L,\mf s)_u \rangle} 
\big( \Omega^n (u X_\nr (L)^{g,\circ}) \otimes \natural_{\mf s}^g \big)^{Z_{W(L,\mf s)_u} (g)} 
\qquad \forall u \in X_\unr (L) .
\]
The injection is $\mc O (X_\nr (L))^{W(L,\mf s)}$-linear if we endow each
$\Omega^n (u X_\nr (M_i))$ with the $\mc O (X_\nr (L))^{W(L,\mf s)}$-module structure coming 
from the central characters of $\mc H (G,K)^{\mf s}$-representations in $\mf F (M_i,\eta_i )$. 
\end{lem}
\begin{proof}
First we consider an arbitrary complex affine variety $V$ and a finitely generated 
$\mc O (V)$-module $M$. It is known from \cite[Lemma 2.9]{SolTwist} that
\begin{equation}\label{eq:4.18}
\text{if the formal completion } \hat M_v \text{ is 0 for all } v \in V, \text{ then } M = 0.
\end{equation}
Since $\mc H (G,K)^{\mf s}$ has finite rank as a module over the Noetherian algebra
$Z(\mc H (G,K)^{\mf s})$, so does $HH_n (\mc H (G,K)^{\mf s})$. Consider a nonzero $x \in HH_n 
(\mc H (G,K)^{\mf s})$. In view of \eqref{eq:4.18}, the $Z(\mc H (G,K)^{\mf s})$-submodule
generated by $x$ has at least one nonzero formal completion, say at $W(L,\mf s) u'$.
Then $x$ is nonzero in that completion, and by Proposition \ref{prop:5.2} the image of
$HH_n (\mc F_{\mf s}) x$ (in a formal completion) is nonzero. Hence $HH_n (\mc F_{\mf s})$
is injective.

Proposition \ref{prop:5.2} shows that the specialization of $HH_n (\mc F_{\mf s}) x$ 
at any central character $W(L,\mf s) u' \subset W(L,\mf s)U_u$ 
has the property involving $HH_n (\phi_u^*)$. Hence $HH_n (\mc F_{\mf s})x$ 
satisfies the stated condition, at least on $U_u$. For each $g$, the required property extends 
from $U_u \cap \Omega^n (u X_\nr (L)^{g,\circ})$ to $\Omega^n (u X_\nr (L)^{g,\circ})$ because 
$U_u$ is Zariski-dense and the $g$-component of $HH_n (\mc F_{\mf s})x$ is an algebraic
differential form. Thus the image of $HH_n (\mc F_{\mf s})$ is contained in the set 
specified in the statement.
\end{proof}

To attain surjectivity in Lemma \ref{lem:5.6}, we have to take the relations between specialization
at $u$ and at $wu$ into account. This is where the algebras from Proposition \ref{prop:5.7}
show their usefulness. Let $HH_n (\tilde \phi_u^*)$ be the map $HH_n (\tilde \phi^*)$ from
\cite[(2.17)]{SolTwist}, for $\mh H^G_{W(L,\mf s)u}$. According to
\cite[Proposition 2.16]{SolTwist} there is a $\C$-linear bijection
\begin{multline}\label{eq:4.19}
HH_n (\tilde \phi_u^*) \circ HH_n (\mc F_1) :\; HH_n ( \mh H^G_{W(L,\mf s)u} ) \;
\longrightarrow \\ 
\Big( \bigoplus_{g \in [W(L,\mf s) / W(L,\mf s)_u]} \bigoplus_{w \in W(L,\mf s)_u} 
\Omega^n (g (u X_\nr (L)^{w,\circ})) \otimes \natural_{\mf s}^{gwg^{-1}} \Big)^{W(L,\mf s)} .
\end{multline}
Here $HH_n (\mc F_1)$ is a version of $HH_n (\mc F_{\mf s})$ for $\mh H^G_{W(L,\mf s)u}$, see 
\cite[around (2.23)]{SolTwist}.

\begin{thm}\label{thm:5.3}
\enuma{
\item The $\mc O (X_\nr (L))^{W(L,\mf s)}$-linear map $HH_n (\mc F_{\mf s})$ is a bijection from
$HH_n (\mc H (G,K)^{\mf s})$ to the set of
$\omega \in \bigoplus_{i=1}^{n_{\mf s}} \Omega^n (X_\nr (M_i))$ such that
\[
HH_n (\tilde \phi_u^*) \omega \in \Big( \bigoplus_{g \in [W(L,\mf s) / W(L,\mf s)_u]} 
\bigoplus_{w \in W(L,\mf s)_u} \Omega^n (g (u X_\nr (L)^{w,\circ})) \otimes 
\natural_{\mf s}^{gwg^{-1}} \Big)^{W(L,\mf s)} 
\]
for all $u \in X_\unr (L)$.
\item The restriction of $HH_n (\mc F_{\mf s}) HH_n (\mc H (G,K)^{\mf s})$ to
$W(L,\mf s) U_u$ is isomorphic to
\[
\Big( \bigoplus\nolimits_{g \in [W(L,\mf s) / W(L,\mf s)_u]} \bigoplus\nolimits_{w \in W(L,\mf s)_u} 
\Omega^n (g (u X_\nr (L)^{w,\circ})) \otimes \natural_{\mf s}^{gwg^{-1}} \Big)^{W(L,\mf s)} 
\]
via $HH_n (\tilde \phi_u^*)$.
}
\end{thm}
\begin{proof}
(a) From Proposition \ref{prop:5.2}.b, \eqref{eq:4.19} and \eqref{eq:5.8} we obtain an
alternative description of the formal completion of $HH_n (\mc H (G,K)^{\mf s})$ at $W(L,\mf s)u'$, 
namely the formal completion at $W(L,\mf s)u'$ of
\begin{equation}\label{eq:5.19}
HH_n (\tilde \phi_u^*)^{-1} \Big( \bigoplus_{g \in [W(L,\mf s) / W(L,\mf s)_u]} 
\bigoplus_{w \in W(L,\mf s)_u} \Omega^n (g (u X_\nr (L)^{w,\circ})) \otimes 
\natural_{\mf s}^{gwg^{-1}} \Big)^{W(L,\mf s)} .
\end{equation} 
Like in Lemma \ref{lem:5.6}, it follows that the image of $HH_n (\mc F_{\mf s})$ is contained in
\eqref{eq:5.19} for all $u \in X_\unr (L)$. The advantage is that now the behaviour at the entire
$W(L,\mf s)$-orbit of $u'$ is captured by \eqref{eq:5.19}. Consider the intersection of the spaces
\eqref{eq:5.19}, over all $u \in X_\unr (L)$. Divide that by the image of $HH_n (\mc F_{\mf s})$.
Proposition \ref{prop:5.2} and \eqref{eq:4.19} tell us that the quotient is a
$\mc O (X_\nr (L))^{W(L,\mf s)}$-module all whose formal completions are zero. As each 
$\mc O (X_\nr (M_i))$ is a finitely generated $\mc O (X_\nr (L))^{W(L,\mf s)}$-module,
so are $\bigoplus_{i=1}^{n_{\mf s}} \Omega^n (X_\nr (M_i))$ and its submodules. Hence we
may apply \eqref{eq:4.18}, which says that the quotient under consideration is the
zero module. In other words, the image of $HH_n (\mc F_{\mf s})$ is precisely the intersection
of the spaces \eqref{eq:5.19}.\\
(b) Here restriction means that we only consider the 
\[
\Omega^n (X_\nr (M_i)) \quad \text{with} \quad 
X_\nr (M_i) \cap W(L,\mf s) U_u \neq \emptyset . 
\]
Suppose that $x \in HH_n (\mc F_{\mf s}) HH_n (\mc H (G,K)^{\mf s})$ is nonzero on
$W(L,\mf s) U_u$. Pick a $u' \in U_u$ at which $x$ is nonzero. Then Proposition \ref{prop:5.2} 
shows that $HH_n (\tilde \phi_u^*) x$ cannot be zero. This proves the injectivity.

The map $HH_n (\tilde \phi_u^*)$ is $\mc O (X_\nr (L))^{W(L,\mf s)}$-linear if we let that
algebra act on\\ $\bigoplus_{i=1}^{n_{\mf s}} \Omega^n (X_\nr (M_i))$ via the maps
\begin{equation}\label{eq:5.27}
X_\nr (M_i) \to X_\nr (L) : \chi \mapsto \chi_{\eta_i} \chi .
\end{equation}
Fix a character $\lambda \in W(L,\mf s) U_u$ of $\mc O (X_\nr (L))^{W(L,\mf s)}$. There are
only finitely many $\mc H (G,K)^{\mf s}$-representations $\pi (Q_i,\eta_i,\lambda_i)$ with
$\chi_{\eta_i} \lambda_i \in W(L,\mf s) \lambda$, so together these support only finitely 
many central characters. By \eqref{eq:5.4} all those central characters lie in $W(L,\mf s) U_u$. 
Then Proposition \ref{prop:5.2} and \eqref{eq:4.19} imply that 
$HH_n (\mc F_{\mf s}) HH_n (\mc H (G,K)^{\mf s})$ and \eqref{eq:5.19} have isomorphic formal 
completions at $W(L,\mf s) \lambda$, with the respect to the 
$\mc O (X_\nr (L))^{W(L,\mf s)}$-module structure coming from \eqref{eq:5.27}. 

Hence the cokernel of $HH_n (\tilde \phi_u^*)$ is a finitely generated 
$\mc O (X_\nr (L))^{W(L,\mf s)}$-module all whose formal completions at points of $W(L,\mf s )U_u$ 
are zero. Thus
\begin{equation}\label{eq:5.34}
I_\lambda^m \text{coker} HH_n (\tilde \phi_u^*) = \text{coker} HH_n (\tilde \phi_u^*)
\qquad \forall \lambda \in W(L,\mf s) U_u, m \in \Z_{>0} .
\end{equation}
Furthermore coker$ HH_n (\tilde \phi_u^*)$ is of form  $\mc O (X_\nr (L)/W(L,\mf s))^r / N$ 
for some submodule $N$ of $\mc O (X_\nr (L)/W(L,\mf s))^r$. Then \eqref{eq:5.34} entails
\[
\mc O (X_\nr (L)/W(L,\mf s))^r = I_\lambda^m \mc O (X_\nr (L)/W(L,\mf s))^r + N
\]
for all $\lambda \in W(L,\mf s) U_u$ and all $m \in \Z_{>0}$. With the Zariski-density of 
$W(L,\mf s) U_u$, it follows that $N = \mc O (X_\nr (L)/W(L,\mf s))^r$. Hence
coker$ HH_n (\tilde \phi_u^*) = 0$ and $HH_n (\tilde \phi_u^*)$ is surjective.
\end{proof}

Recall that $HH_0 (\mc H (G))$ and $HH_0 (\mc H (G,K)^{\mf s})$ were already computed 
in \cite{BDK}. We will now recover those results via families of representations. For
a variation using only tempered representations we refer to \cite{Mui}.

\begin{prop}\label{prop:5.4}
\enuma{
\item $HH_0 (\mc F_{\mf s})$ provides an isomorphism between $HH_0 (\mc H (G,K)^{\mf s})$ 
and the set of elements of $\bigoplus\nolimits_{i=1}^{n_{\mf s}} \mc O (X_\nr (M_i))$ that 
descend to linear functions on $\C \otimes_\Z R (\mc H (G,K)^{\mf s})$.
\item Part (a) yields an isomorphism of $Z(\mc H (G,K)^{\mf s})$-modules
\[
HH_0 (\mc H (G,K)^{\mf s}) \cong (\C \otimes_\Z R(G)^{\mf s})^*_{\mr{reg}} .
\]
}
\end{prop}
\begin{proof}
(a) With Theorem \ref{thm:5.1} and Proposition \ref{prop:5.7} we reduce this to an issue for
$\mh H^G_{W(L,\mf s)u}$. In that setting \cite[Proposition 2.16.a]{SolTwist} is equivalent to 
Theorem \ref{thm:5.3}.a, and the desired description is \eqref{eq:4.19}.\\
(b) The definition of $HH_0 (\mc F_{\mf s})$ involves the generalized trace map and the
Hochschild--Kostant--Rosenberg theorem, like in \cite[Paragraph 1.2]{SolTwist}. Unwinding this,
we find that the map 
\[
HH_0 (\mc H (G,K)^{\mf s}) \to \big( \C \otimes_\Z R(\mc H (G,K)^{\mf s}) \big)^*
\]
from part (a) is just \eqref{eq:3.12}. In particular every element of $HH_0 (\mc H (G,K)^{\mf s})$
determines a regular linear function on $\C \otimes_\Z R(\mc H (G,K)^{\mf s})$. The map is 
injective because $HH_n (\mc F_{\mf s})$ is injective and because for
$f \in \bigoplus_{i=1}^{n_{\mf s}} \mc O (X_\nr (M_i))$ the values $f (P_i,\delta_i,v_i)$
can be recovered from the image of $f$ in $\big( \C \otimes_\Z R(\mc H (G,K)^{\mf s}) \big)^*$.
By Morita equivalence, we may replace $R(\mc H (G,K)^{\mf s})$ with $R (G)^{\mf s}$.

Conversely, for every $\lambda \in \big( \C \otimes_\Z R(\mc H (G,K)^{\mf s}) \big)^*_{\mr{reg}}$ 
the canonical pairing with $\mf F_{M_i,\eta_i}$ produces a regular function on $X_\nr (M_i)$,
so $\lambda$ comes from an element of $\bigoplus_{i=1}^{n_{\mf s}} \mc O (X_\nr (M_i))$.
\end{proof} 

Let $\Delta_G^{\mf s}$ be a set of representatives for the inertial equivalence classes of
square-integrable modulo centre representations $\delta$ of standard Levi subgroups $M$ of $G$,
such that $I_P^G (\delta) \in \Rep (G)^{\mf s}$. From Theorem \ref{thm:3.1} we see that the
category of tempered representations in $\Rep (G)^{\mf s}$ decomposes as
\begin{equation}\label{eq:5.10}
\Rep^t (G)^{\mf s} = \bigoplus\nolimits_{\mf d = [M,\delta] \in \Delta_G^{\mf s}} \Rep^t (G)^{\mf d},
\end{equation} 
where $\Rep^t (G)^{\mf d}$ is the full subcategory generated by the subquotients of $I_P^G (\delta 
\otimes \chi)$ with $\chi \in X_\unr (M)$. With Theorem \ref{thm:5.1} and the same arguments as 
in the proof of \cite[Theorem 2.2]{SolTwist}, \eqref{eq:5.10} induces a decomposition
\begin{equation}\label{eq:5.11}
R^t (\mh H^G_u ) = \bigoplus\nolimits_{\mf d \in \Delta_G^{\mf s}} R^t \big( \mh H^G_u \big)^{\mf d} .
\end{equation}
By Proposition \ref{prop:5.7} $R_t (\mh H^G_{W(L,\mf s)u})$ decomposes in the same way.

It is known from \cite[Proposition 9.5]{SolEnd} that the equivalence of categories in Theorem 
\ref{thm:5.1} sends square-integrable modulo centre representations to tempered essentially discrete 
series representations. With that and the same process that made $\tilde \eta$ out of $\eta$, 
described around \eqref{eq:5.12}, we can associate to $\mf d = [M,\delta] \in \Delta_G^{\mf s}$ 
a discrete series representation $\tilde \delta$ of $\mh H_M$. Thus \eqref{eq:5.11} is a 
decomposition of the kind considered in \cite[Theorem 2.2 and (2.25)]{SolTwist}. 
We define 
\[
\mc F_{\mf d} = \bigoplus\nolimits^{n_{\mf s}}_{i=1,i \prec \mf d} \mc F_{M_i,\eta_i}.
\]

\begin{lem}\label{lem:5.5}
\enuma{
\item There is a canonical decomposition
\[
HH_n (\mc H (G,K)^{\mf s}) = \bigoplus\nolimits_{\mf d \in \Delta_G^{\mf s}}
HH_n (\mc H (G,K)^{\mf s})^{\mf d} ,
\]
where the part indexed by $\mf d = [M,\delta]$ is obtained by applying $HH_n (\mc F_{\mf s})^{-1}$ to
\[
HH_n (\mc F_{\mf d}) HH_n (\mc H (G,K)^{\mf s}) = \bigoplus\nolimits_{i \prec \mf d} 
\Omega^n (\mf t^{M_i}) \cap HH_n ( \mc F_{\mf s}) HH_n (\mc H (G,K)^{\mf s}) .
\]
\item Select $\chi_\delta \in X_\unr (L), t_\delta^+ \in X_\nr^+ (L)$ such that 
$\chi_\delta t_\delta^+$ represents the $Z(\mc H (G,K)^{\mf s}$-character of $\delta$. The map 
\begin{multline*}
HH_n (\tilde \phi_u^*) \circ HH_n (\mc F_{\mf d}) : HH_n (\mc H (G,K)^{\mf s})^{\mf d} 
\longrightarrow \\
\Big( \bigoplus\nolimits_{g \in [W(L,\mf s) / W(L,\mf s)_u]} \bigoplus\nolimits_{w \in W(L,\mf s)_u} 
\Omega^n (g (u X_\nr (L)^{w,\circ})) \otimes \natural_{\mf s}^{gwg^{-1}} \Big)^{W(L,\mf s)} 
\end{multline*}
is $\mc O (X_\nr (L))^{W(L,\mf s)}$-linear if we let $\mc O (X_\nr (L))^{W(L,\mf s)}$ act on 
the target such that:
\begin{itemize}
\item if $g (u X_\nr (L)^{w,\circ}) \subset \chi_\delta X_\nr (M)$, then it acts at
$g u \chi$ with $\chi \in X_\nr (L)^{w,\circ}$ via the character $W(L,\mf s) u \chi t_\delta^+$,
\item in the same situation $\mc O (X_\nr (L))^{W(L,\mf s)}$ acts at $h u \chi$, where
$h \in W(L,\mf s)$ and $\chi \in X_\nr (L)^{w,\circ}$, also via the character 
$W(L,\mf s) u \chi t_\delta^+$,
\item if $h (u X_\nr (L)^{w,\circ}) \not \subset \chi_\delta X_\nr (M)$ for any
$h \in W(L,\mf s)$, then $\mc O (X_\nr (L))^{W(L,\mf s)}$ annihilates 
$\Omega^n (g (u X_\nr (L)^{w,\circ}))$.
\end{itemize}
}
\end{lem}
\begin{proof}
(a) This follows from \cite[Lemma 2.12, (2.25), Corollary 2.13]{SolTwist} and 
Theorem \ref{thm:5.3}.\\
(b) The condition in the third bullet means that for $i \prec \mf d$ no map $\phi_{w,i} :
u X_\nr (L)^{w,\circ} \to \chi_\delta X_\nr (M)$ can exist. In that case $\lambda_{w,i} = 0$
and the image of $HH_n (\mc H (G,K)^{\mf s})^{\mf d}$ in $\Omega^n (g (u X_\nr (L)^{w,\circ})$
is 0. From that and \eqref{eq:5.15} we see that, for each $i$ separately, there exists such a 
$\mc O (X_\nr (L))^{W(L,\mf s)}$-module structure as indicated, only with 
$t^+_{\eta_i} \in X_\nr^+ (L)$ instead of $t^+_\delta$.

By \cite[Corollary 2.13]{SolTwist}, $HH_n (\tilde \phi_u^*) \circ HH_n (\mc F_{\mf d})$ is 
$\mc O (X_\nr (L))^{W(L,\mf s)}$-linear if we let it act according to the central characters of 
the virtual representations $\nu_{g,w,v}^{\mf d}$ from \cite[(2.26)]{SolTwist}. That means that 
the natural module structure is adjusted by a representative 
$cc(\delta) \in \mf t_\R$ of the central character of $\delta$ (as representation of $\mh H^G_u$).
So in that setting $\log (t^+_{\eta_i})$ and $\log (cc (\delta))$ represent the same central
character, for all $i \prec \mf d$. We have translate these to $\mc H (G,K)^{\mf s}$-representations 
with Theorem \ref{thm:5.1} and Proposition \ref{prop:5.7}. Then $cc(\delta)$ becomes $t^+_\delta$. 
Hence $W(L,\mf s) t^+_{\eta_i} = W(L,\mf s) t^+_\delta$ for all $i \prec \mf d$.

Thus the $\mc O (X_\nr (L))^{W(L,\mf s)}$-module structures for the $i \prec \mf d$ agree, and
combine to make $HH_n (\tilde \phi_u^*) \circ HH_n (\mc F_{\mf d})$ a module homomorphism with
the indicated character shift.
\end{proof}

\subsection{Hochschild homology for one entire Bernstein component} \
\label{par:component}

We would like to combine the local conditions involving $HH_n (\phi_u^*)$ to a smaller set of
conditions that describe $HH_n (\mc H (G,K)^{\mf s})$ globally on $X_\nr (L)$. This is difficult
because the algebras $\mh H^G_u$ and $\mh H^G_{W(L,\mf s)u}$ do not vary continuously with $u
\in X_\unr (L)$. To compensate for that, we relate the local conditions coming from $u, u' \in 
X_\unr (L)$ that are close. When $W(L,\mf s)_{u'} \subset W(L,\mf s)_u$, we define
\[
HH_n (\phi_u^*)_{u'} = \bigoplus_{w \in W(L,\mf s)_{u'}} \sum_{i=1, U_u\text{-rel}}^{n_{\mf s}} 
\lambda_{w,i} HH_n (\chi_{\eta_i}^{-1} \phi_{w,i}^*) .
\]

\begin{lem}\label{lem:5.8}
\enuma{
\item Let $u' \in U_u \cap X_\unr (L)$. Then 
\begin{multline*}
HH_n (\phi_{u'}^*)_{u'}^{-1} \Big( \bigoplus\nolimits_{w \in W(L,\mf s)_{u'}} \big( 
\Omega^n (u' X_\nr (L)^{w,\circ}) \otimes \natural_{\mf s}^w \big)^{Z_{W(L,\mf s)_{u'}} (w)} \Big) = \\
HH_n (\phi_{u}^*)_{u'}^{-1} \Big( \bigoplus\nolimits_{w \in W(L,\mf s)_{u'} } \big( 
\Omega^n (u X_\nr (L)^{w,\circ}) \otimes \natural_{\mf s}^w \big)^{Z_{W(L,\mf s)_{u}} (w)} \Big).
\end{multline*}
\item Part (a) also holds for $u \in u' X_\unr (L)^{W(L,\mf s)_{u'},\circ}$.
}
\end{lem}
\begin{proof}
(a) Notice that $W(L,\mf s)_{u'} \subset W(L,\mf s)_u$ by the conditions on $U_u$.
We may choose $U_{u'}$ so small that it is contained in $U_u$. From the proof of 
\cite[Proposition 2.16.a]{SolTwist} we know that
\begin{multline*}
\bigoplus\nolimits_{w \in W(L,\mf s)_{u'} } \big( \Omega^n (u' X_\nr (L)^{w,\circ}) 
\otimes \natural_{\mf s}^w \big)^{Z_{W(L,\mf s)_{u'}} (w)} \cong \\
\Big( \bigoplus\nolimits_{g \in [W(L,\mf s) / W(L,\mf s)_{u'}]} 
\bigoplus\nolimits_{w \in W(L,\mf s)_{u'}} 
\Omega^n (g (u' X_\nr (L)^{w,\circ})) \otimes \natural_{\mf s}^{gwg^{-1}} \Big)^{W(L,\mf s)} .
\end{multline*}
By construction $HH_n (\phi_{u'}^*)_{u'}^{-1}$ of the left hand side equals 
$HH_n (\tilde \phi_{u'}^*)^{-1}$ of the right hand side. By Theorem \ref{thm:5.3}.b, this 
describes precisely the restriction of\\ 
$HH_n (\mc F_{\mf s}) HH_n (\mc H (G,K)^{\mf s})$ to $W(L,\mf s) U_{u'}$. Similarly 
\[
HH_n (\phi_u^*)_u^{-1} \Big( \bigoplus\nolimits_{w \in W(L,\mf s)_u} \big( 
\Omega^n (u X_\nr (L)^{w,\circ}) \otimes \natural_{\mf s}^w \big)^{Z_{W(L,\mf s)_u} (w)} \Big)
\]
describes precisely the restriction of $HH_n (\mc F_{\mf s}) HH_n (\mc H (G,K)^{\mf s})$ to 
$W(L,\mf s) U_{u}$. Restricting that further $W(L,\mf s) U_{u'}$ means that we remove the 
summands for $w \in W(L,\mf s)_u$ that do not fix $u'$, because for those $U_{u'} \cap
X_\nr (L)^w = \emptyset$ and $w(U_{u'}) \cap U_{u'} = \emptyset$, by the properties of $U_{u'}$.
That leaves us with
\[
HH_n (\phi_{u}^*)_{u'}^{-1} \Big( \bigoplus\nolimits_{w \in W(L,\mf s)_{u'}} \big( \Omega^n 
(u X_\nr (L)^{w,\circ}) \otimes \natural_{\mf s}^w \big)^{Z_{W(L,\mf s)_{u}} (w)} \Big).
\]
(b) Pick a path $p$ from $u'$ to $u$ in $u' X_\unr (L)^{W(L,\mf s)_{u'},\circ}$. We even assume 
that $W(L,\mf s)_y = W(L,\mf s)_{u'}$ for all $y$ on $p$, because that condition holds on an
open dense subset of $X_\unr (L)$. By the compactness of $X_\unr (L)$, we can choose a finite 
subset $Y$ of $p$, such that the $U_y$ with $y \in Y$ cover $p$. Choose a finite sequence
$y_1,y_2,\ldots,y_m$ in $Y$, such that $u' \in U_{y_1}, u \in U_{y_m}$ and 
\[
U_{y_i} \cap U_{y_{i+1}} \cap p \neq \emptyset \text{ for all } 1 \leq i < m . 
\]
For $1 \leq i < m$ we pick $z_i \in U_{y_i} \cap U_{y_{i+1}} \cap p$. We follow the new sequence
\[
u', y_1, z_1, y_2, z_2, \ldots, z_{m-1}, y_m, u .
\]
At each step part (a) guarantees that the relevant preimages under $HH_n (\phi^*_? )_{u'}$ do
not change.
\end{proof}

For $c \in \pi_0 (X_\nr (L)^w)$, we denote the corresponding connected component of $w$-fixed
points by $X_\nr (L)^w_c$. Then $W(L,\mf s)$ acts naturally on the set of such components, and on
the set of pairs $(w,c)$. We denote the stabilizer of $(w,c)$ by $W(L,\mf s)_{w,c}$, this
is a subgroup of $Z_{W(L,\mf s)}(w)$. We register these connected components with the list of
pairs $(w,c)$, where $w \in W(L,\mf s)$ and $c \in \pi_0 (X_\nr (L)^w)$. We write $(w',c') \geq 
(w,c)$ if $X_\nr (L)^{w'}_{c'} \supset X_\nr (L)^w_c$, $(w',c') \sim (w,c)$ if
$X_\nr (L)^{w'}_{c'} \supset X_\nr (L)^w_c$ and $(w',c') > (w,c)$ if $X_\nr (L)^{w'}_{c'} 
\supsetneq X_\nr (L)^w_c$.

We ready to reorganize the conditions that describe $H_n (\mc F_{\mf s}) HH_n (\mc H (G,K)^{\mf s})$
in Theorem \ref{thm:5.3}. This is done with decreasing induction on the dimension of the 
connected components $X_\nr (L)^w_c$, or equivalently on the pairs $(w,c)$.
\begin{construct}\label{cons:3}
\begin{enumerate}[(i)]
\item We start with $w=1$ and $X_\nr (L)^w_c = X_\nr (L)$. Pick $u_1 \in X_\unr (L)$ with
$W(L,\mf s)_{u_1} = \{1\}$. Then $HH_n (\phi^*_{1,c}) := HH_n (\phi_{u_1}^*)$ is just a map
\[
\bigoplus\nolimits_{i=1}^{n_{\mf s}} \Omega^n (X_\nr (M_i)) \to 
\Omega^n (u_1 X_\nr (L)) = \Omega^n (X_\nr (L)) ,
\]
and it sends $HH_n (\mc F_{\mf s}) HH_n (\mc H (G,K)^{\mf s})$ to $\Omega^n (X_\nr (L))^{W(L,\mf s)}$.
By Lemma \ref{lem:5.8}, this completely describes the restriction of 
$HH_n (\mc F_{\mf s}) HH_n (\mc H (G,K)^{\mf s})$ to the subset of $X_\nr (L)$ not fixed by any
nontrivial element of $W(L,\mf s)$. Remove $(1,c)$ from the list of pairs.
\item Assume that for some connected components $X_\nr (L)^w_c$ we have already chosen a map
\begin{equation}\label{eq:5.26}
HH_n (\phi_{w,c}^*) : \bigoplus\nolimits_{i=1}^{n_{\mf s}} \Omega^n (X_\nr (M_i)) 
\to \Omega^n (u_{w,c} X_\nr (L)^{w,\circ}) = \Omega^n (X_\nr (L)^w_c) , 
\end{equation}
of the form $\sum_{i=1, U_u\text{-rel}}^{n_{\mf s}}\lambda_{w,i} HH_n (\chi_{\eta_i}^{-1} \phi_{w,i}^*)$ 
coming from $HH_n (\phi_u^*)$ for some $u = u_{w,c}$ with $u_{w,c} \in X_\nr (L)^w_c$ but not in any 
connected component of smaller dimension. Assume that the set of pairs $(w,c)$ for which this has 
been done is closed under passing to larger pairs. Assume that all those pairs 
have been removed from the list. Finally and most importantly, we assume that for all those pairs
$(w,c)$ the restriction of $HH_n (\mc F_{\mf s}) HH_n (\mc H (G,K)^{\mf s})$ to
$X_\nr (L)^w_c$ without the connected components of smaller dimension equals
\begin{equation}\label{eq:5.20}
\bigcap\nolimits_{(w',c') \geq (w,c)} HH_n (\phi_{w',c'}^*)^{-1} 
\big( \Omega^n (X_\nr (L)^{w'}_{c'}) \otimes \natural_{\mf s}^{w'} \big)^{W(L,\mf s)_{w',c'}} .
\end{equation}
\item From the list of remaining pairs, pick a $(g,c)$ with $X_\nr (L)^g_c$ of maximal dimension.
Select $u = u_{g,c}$ in $X_\unr (L)^g_c$ but not in any connected component of smaller dimension.
Define 
\begin{equation}\label{eq:5.28}
HH_n (\phi_{g,c}^*) = \sum\nolimits_{i=1, U_u\text{-rel}}^{n_{\mf s}} \lambda_{g,i} 
HH_n (\chi_{\eta_i}^{-1} \phi_{g,i}^*) ,
\end{equation}
a part of $HH_n (\phi_u^*)$. If there are other $h \in W(L,\mf s)$ with $X_\nr (L)^h_c = 
X_\nr (L)^g_c$, then we take $u_{h,c} = u_{g,c}$ and we define $HH_n (\phi_{h,c}^*)$ in the 
same way. We need to check that 
\begin{multline}\label{eq:5.21}
\bigcap\nolimits_{(w',c') > (g,c)} HH_n (\phi_{w',c'}^*)^{-1} \big( \Omega^n (X_\nr (L)^{w'}_{c'}) 
\otimes \natural_{\mf s}^{w'} \big)^{W(L,\mf s)_{w',c'}} \cap \\
\bigcap\nolimits_{(h,c) \sim (g,c)} HH_n (\phi_{h,c}^*)^{-1} 
\big( \Omega^n (X_\nr (L)^{h}_{c}) \otimes \natural_{\mf s}^{h} \big)^{W(L,\mf s)_{h,c}} 
\end{multline}
equals
\begin{equation}\label{eq:5.22}
HH_n (\phi_{u}^*)^{-1} \bigoplus\nolimits_{w \in \langle W(L,\mf s)_{u} \rangle} 
\big( \Omega^n (u X_\nr (L)^{w,\circ}) \otimes \natural_{\mf s}^w \big)^{Z_{W(L,\mf s)_{u}} (w)} .
\end{equation}
This follows from Lemma \ref{lem:5.8}, which says that all the parts $HH_n (\phi_{u}^*)_{u'}$
with $u' \in U_u$ and $u' \notin X_\nr (L)^g_c$ are accounted for by the $(w',c') > (g,c)$.
Lemma \ref{lem:5.8} also tells us that \eqref{eq:5.21} and \eqref{eq:5.22} describe exactly
the restriction of\\ $HH_n (\mc F_{\mf s}) HH_n (\mc H (G,K)^{\mf s})$ to $X_\nr (L)^g_c$ 
without the components of smaller dimension. 
\item For components $(g',c')$ in the $W(L,\mf s)$-orbit of $(g,c)$ or any of the $(h,c)$, we
define the maps $HH_n (\phi_{g',c'}^*)$ by imposing $W(L,\mf s)$-equivariance (where the group
acting involves the characters $\natural_{\mf s}^g$). This construction ensures that
\begin{multline}\label{eq:5.23}
HH_n (\phi_{g,c}^*)^{-1} \big( \Omega^n (X_\nr (L)^{g}_{c}) \otimes \natural_{\mf s}^{g} 
\big)^{W(L,\mf s)_{g,c}} = \\
\Big( \sum_{(g',c') \in W(L,\mf s)(g,c)} HH_n (\phi_{g',c'}^*) \Big)^{-1} \Big( \bigoplus_{(g',c') 
\in W(L,\mf s)(g,c)} \Omega^n (X_\nr (L)^{h}_{c}) \otimes \natural_{\mf s}^{h} \Big)^{W(L,\mf s)} . 
\end{multline}
\item Remove $(g,c)$ and the pairs $(h,c) \sim (g,c)$ from the list of pairs. Stop if there are
no pairs left, otherwise return to step (iii).
\end{enumerate}
\end{construct}
With \eqref{eq:5.28} we associate to $(w,\chi)$ the virtual $\mc H (G,K)^{\mf s}$-representation
\begin{equation}\label{eq:5.31}
\nu^1_{w,\chi} = \sum\nolimits_{i=1, U_u\text{-rel}}^{n_{\mf s}} 
\lambda_{w,i} \mr{tr} \, \pi (M_i,\eta_i, \chi_{\eta_i}^{-1} \phi_{w,i}(\chi)) .
\end{equation}
In other words, the specialization of $HH_n (\phi_{w,c}^*)$ at $\chi \in X_\nr (L)^w_c$
corresponds to the map on Hochschild homology induced by $\nu^1_{w,\chi}$. This means that
\begin{equation}\label{eq:4.13}
HH_n (\phi^*_{w,c}) \circ HH_n (\mc F_{\mf s}) : HH_n (\mc H (G,K)^{\mf s}) \to
\Omega^n (X_\nr (L)^w_c) 
\end{equation}
is induced by the algebraic family of virtual representations 
$\{ \nu^1_{w,\chi} : \chi \in X_\nr (L)^w_c \}$. From \cite[Lemma 2.5.a]{SolTwist} (translated to 
the current setting with Theorem \ref{thm:5.1}) and step (iv) above we see that
\begin{equation}\label{eq:5.29}
\nu^1_{g w g^{-1},g \chi} = \natural_{\mf s}^w (g) \nu^1_{w,\chi} .
\end{equation}
Considering $\nu^1_{w,\chi}$ as virtual $G$-representation via the equivalence of categories
\eqref{eq:4.1}, we deduce from \eqref{eq:5.28} and \eqref{eq:4.10} that
\begin{equation}\label{eq:4.11}
\zeta^\vee (\nu^1_{w,\chi}) = \nu_{w,\chi} .
\end{equation}
The above procedure gives rise to a description of $HH_n (\mc H (G,K)^{\mf s})$ that is more
concrete than Theorem \ref{thm:5.3}.

\begin{thm}\label{thm:5.9}
For each $w \in W(L,\mf s)$ and each $c \in \pi_0 (X_\nr (L)^w)$, let $HH_n (\phi_{w,c}^*)$
be as above. We define 
\[
HH_n (\tilde \phi_{\mf s}^*) = 
\bigoplus\nolimits_{w \in W(L,\mf s), c \in \pi_0 (X_\nr (L)^w)} HH_n (\phi_{w,c}^*)
\]
\enuma{
\item The map
\[
HH_n (\tilde \phi_{\mf s}^*) \;:\; \bigoplus_{i=1}^{n_{\mf s}} \Omega^n (X_\nr (M_i)) 
\to \bigoplus_{w \in W(L,\mf s), c \in \pi_0 (X_\nr (L)^w)} \Omega^n (X_\nr (L)^w_c)
\]
is injective.
\item $HH_n (\tilde \phi_{\mf s}^*)$ gives a $\C$-linear bijection
\[
HH_n (\mc F_{\mf s}) HH_n (\mc H (G,K)^{\mf s}) \to \Big( \bigoplus\nolimits_{w \in W(L,\mf s)} 
\Omega^n (X_\nr (L)^w) \otimes \natural_{\mf s}^w \Big)^{W(L,\mf s)} .
\]
\item For $\mf d \in \Delta_G^{\mf s}$, the restriction of $HH_n (\tilde \phi_{\mf s}^*) 
\circ HH_n (\mc F_{\mf s})$ to the direct summand $HH_n (\mc H (G,K)^{\mf s})^{\mf d}$ of
$HH_n (\mc H (G,K)^{\mf s})$ becomes $\mc O (X_\nr (L))^{W(L,\mf s)}$-linear if we endow the 
target with the same module structure as in Lemma \ref{lem:5.5}.
}
\end{thm}
\begin{proof}
(a) Recall that the specialization of $HH_n (\phi_{w,c}^* )$ at $\chi \in X_\nr (L)^w_c$ came 
from a virtual representation  of $\mh H^G_{\chi |\chi|^{-1}}$, translated to a virtual 
representation $\nu_{1,w,\chi}$ of $\mc H (G,K)^{\mf s}$ via Theorem \ref{thm:5.1} and Proposition 
\ref{prop:5.7}. 

Consider one $u \in X_\unr (L)$. From \eqref{eq:5.20} and Theorem \ref{thm:5.3} we see that the 
$\nu_{1,w,\chi}$  with central characters in $W(L,\mf s) U_u$ span the same set of virtual 
representations with central characters in $W(L,\mf s) U_u$ as all the virtual representations 
$\nu_{g,u'}$ defined via $\mh H^G_u$. The latter collection spans the entire part of 
$\C \otimes_\Z R (\mc H (G,K)^{\mf s})$ with central characters in $W(L,\mf s) U_u$, so the 
$\nu_{w,\chi}$ span that as well. It follows that the specialization of $HH_n (\phi_{\mf s}^*) x$ 
at $\chi \in W(L,\mf s)U_u$ is zero if and only if the specialization of $HH_n (\tilde \phi_u^*)x$ 
at $\chi$ is zero.

From \cite[Lemmas 1.12 and 2.7]{SolTwist} we know that $HH_n (\phi_u^*)$ and 
$HH_n (\tilde \phi_u^*)$ are injective, for twisted graded Hecke algebras and for 
$\mh H (V,G,k,\natural)$ as in \cite[Paragraph 2.3]{SolTwist}, and hence also for the algebras
$\mh H^G_u$ and $\mh H^G_{W(L,\mf s)}$. By the above considerations with virtual representations,
$HH_n (\tilde \phi_{\mf s}^*)$ contains at least as much information as $HH_n (\tilde \phi_u^*)$.
Hence $HH_n (\tilde \phi_{\mf s}^*) x$ is nonzero as soon as $x \in \bigoplus_{i=1}^{n_{\mf s}} 
\Omega^n (X_\nr (M_i))$ does not vanish on $U_u$. That holds for for every $u \in X_\unr (L)$,
so $HH_n (\tilde \phi_{\mf s}^*) x$ is injective.\\
(b) We already know from Lemma \ref{lem:5.6} and part (a) that 
$HH_n (\tilde \phi_{\mf s}^*) \circ HH_n (\mc F_{\mf s})$ is injective.
From Theorem \ref{thm:5.3}, \eqref{eq:5.20} and \eqref{eq:5.23} we know that the assertion
holds locally. Hence
\begin{multline}\label{eq:5.24}
HH_n (\tilde \phi_{\mf s}^*) \circ HH_n (\mc F_{\mf s}) HH_n (\mc H (G,K)^{\mf s}) \subset \\
\Big( \bigoplus\nolimits_{w \in W(L,\mf s), c \in \pi_0 (X_\nr (L)^w)} \Omega^n (X_\nr (L)^w_c) 
\otimes \natural_{\mf s}^w \Big)^{W(L,\mf s)} .
\end{multline}
Both sides are finitely generated $\mc O (X_\nr (L))^{W(L,\mf s)}$-modules
(for natural module structure, not the module structure determined by the characters of the
underlying virtual representations). We consider the quotient module $M$. Since the two sides
of \eqref{eq:5.24} are isomorphic locally, all formal completions of $M$ with respect to
characters of $\mc O (X_\nr (L))^{W(L,\mf s)}$ are zero. By \eqref{eq:4.18}
$M = 0$, so the inclusion \eqref{eq:5.24} is an equality. Finally, we note that
\[
\bigoplus\nolimits_{c \in \pi_0 (X_\nr (L)^w)} \Omega^n (X_\nr (L)^w_c) = \Omega^n (X_\nr (L)^w).
\]
(c) This follows from Lemma \ref{lem:5.5}, since every component of 
$HH_n (\tilde \phi_{\mf s}^*)$ occurs as component of a $HH_n (\tilde \phi_u^*)$. 
\end{proof}

Recall that $\mc H (G,K)^{\mf s}$ is Morita equivalent with $\End_G (\Pi_{\mf s})^{op}$ and 
that in \eqref{eq:4.5} we fixed an isomorphism of $\mc O (X_\nr (L))^{W(L,\mf s)}$-algebras
\[
\C (X_\nr (L))^{W(L,\mf s)} \underset{\mc O (X_\nr (L))^{W(L,\mf s)}}{\otimes}
\End_G (\Pi_{\mf s})^{op} \isom \C (X_\nr (L)) \rtimes \C [W(L,\mf s),\natural_{\mf s}] .
\]
Recall the bijection $\zeta^\vee$ from \eqref{eq:4.7} and Theorem \ref{thm:4.1}.

\begin{thm}\label{thm:4.2}
There exists a unique $\C$-linear bijection 
\[
HH_n (\zeta^\vee) : HH_n \big( \mc O (X_\nr (L)) \rtimes \C [W(L,\mf s),\natural_{\mf s}] \big)
\to HH_n (\mc H (G,K)^{\mf s})
\]
such that
\[
HH_n (\mc F_{M,\eta}) \circ HH_n (\zeta^\vee) = HH_n (\mc F_{M,\zeta^\vee (\eta)})
\]
for all algebraic families $\mf F (M,\eta)$ in $\Rep (G)^{\mf s}$.
\end{thm}
\begin{proof}
By \cite[Theorem 1.2]{SolTwist} there is a $\mc O (X_\nr (L))^{W(L,\mf s)}$-linear bijection
\begin{equation}\label{eq:4.12}
HH_n \big( \mc O (X_\nr (L)) \rtimes \C [W(L,\mf s),\natural_{\mf s}] \big) \to \Big( 
\bigoplus_{w \in W(L,\mf s)} \Omega^n (X_\nr (L)^w) \otimes \natural_{\mf s}^w \Big)^{W(L,\mf s)} .
\end{equation}
It is not canonical, but from \cite[(1.15) and (1.17)]{SolTwist} we know that the non-canonicity
is limited to one scalar factor for each direct summand indexed by a conjugacy class in
$W(L,\mf s)$. We can fix these scalar factors by requiring that \eqref{eq:4.12} on the 
summand indexed by $w$ is induced by the algebraic family of virtual representations 
\begin{equation}\label{eq:4.14}
\{ \nu_{w,\chi} : \chi \in X_\nr (L)^w \}. 
\end{equation}
Indeed, the bijection \eqref{eq:4.12} is recovered in that way in \cite[Theorem 1.13.a]{SolTwist}. 
The only issue is that \cite[\S 1.2]{SolTwist} applies not to tori like $X_\nr (L)$, but to complex 
vector spaces. Fortunately \cite[Theorem 1.13]{SolTwist} can easily be extended to our setting by 
localization of $\mc O (X_\nr (L))^{W(L,\mf s)}$ to sets of the form $W(L,\mf s) U_u / W(L,\mf s)$. 
Thus we make \eqref{eq:4.12} canonical.

By Lemma \ref{lem:5.6} and Theorem \ref{thm:5.9}.b,
\[
HH_n (\tilde{\phi}_{\mf s}^*) \circ HH_n (\mc F_{\mf s}) : HH_n (\mc H (G,K)^{\mf s}) \to \Big( 
\bigoplus_{w \in W(L,\mf s)} \Omega^n (X_\nr (L)^w) \otimes \natural_{\mf s}^w \Big)^{W(L,\mf s)} 
\]
is a $\C$-linear bijection. We define $HH_n (\zeta^\vee)$ as the composition of \eqref{eq:4.12}
with $\big( HH_n (\tilde{\phi}_{\mf s}^*) \circ HH_n (\mc F_{\mf s}) \big)^{-1}$. From 
\eqref{eq:4.13}, \eqref{eq:4.11} and \eqref{eq:4.14} we see that
\begin{equation}\label{eq:4.15}
HH_n (\mc F) \circ HH_n (\zeta^\vee) = HH_n (\zeta^\vee \circ \mc F)
\end{equation}
whenever $\mc F$ is one of the families $\{ \nu^1_{w,\chi} : \chi \in X_\nr (L)^w_c \}$ with
$c \in \pi_0 (X_\nr (L)^w)$. Since every such $\mc F$ is a linear combination of algebraic
families of $\mc H (G,K)^{\mf s})$-representations, \eqref{eq:4.15} is implied the condition
in the theorem. Hence $HH_n (\zeta^\vee)$ is unique.

It remains to check that 
\begin{equation}\label{eq:4.16}
HH_n (\mc F_{M,\eta}) \circ HH_n (\zeta^\vee) = HH_n (\mc F_{M,\zeta^\vee (\eta)})
\end{equation}
for an arbitrary algebraic family in $\Rep (G)^{\mf s}$. By \cite[Lemma 1.9]{SolTwist} the
virtual representations $\nu_{w,\chi}$ with $w \in [W(L,\mf s)]$ and $v \in X_\nr (L)^w / 
Z_{W(L,\mf s)}(w)$ such that $\natural_{\mf s}^w (W(L,\mf s)_v \cap Z_{W(L,\mf s)}(w)) = 1$
form a basis of 
\begin{equation}\label{eq:4.17}
\C \otimes_\Z R\big( \mc O (X_\nr (L)) \rtimes \C [W(L,\mf s),\natural_{\mf s}] \big) .
\end{equation}
Hence there exist coefficients $c (w,v,\chi) \in \C$ such that
\[
\mf F (M,\zeta^\vee (\eta),\chi) = \sum\nolimits_{w,v} c(w,v,\chi) \nu_{w,v} .
\]
Then \eqref{eq:4.11} and the bijectivity of Theorem \ref{thm:4.1} imply
\[
\mf F (M,\eta,\chi) = \sum\nolimits_{w,v} c(w,v,\chi) \nu^1_{w,v} .
\]
Recall from \eqref{eq:5.31} that $\nu^1_{w,v}$ is a linear combination of the 
$\mf F (M_i,\eta_i ,v_i)$. With Theorem \ref{thm:4.1} that can be transferred to \eqref{eq:4.17}.
Hence there exist $c' (i,\chi,v_i) \in \C$ such that
\[
\begin{array}{lll}
\mf F (M,\eta,\chi) & = & \sum\nolimits_{i,v_i} c' (i,\chi,v_i) \mf F (M_i,\eta_i,v_i) , \\
\mf F (M,\zeta^\vee (\eta),\chi) & = &
\sum\nolimits_{i,v_i} c' (i,\chi,v_i) \mf F (M_i,\zeta^\vee (\eta_i),v_i) . 
\end{array}
\]
From \eqref{eq:4.15} we deduce that 
\[
HH_n (\mf F (M,\eta,\chi)) \circ HH_n (\zeta^\vee) = HH_n (\mf F (M,\zeta^\vee (\eta),\chi)) .
\]
The same argument for all $\chi \in X_\nr (M)$ simultaneously yields \eqref{eq:4.16}.
\end{proof}

Theorem \ref{thm:4.2} is a homological counterpart to \cite[Theorem 9.9]{SolEnd}, which matches the
irreducible representations of $\mc H (G,K)^{\mf s}$ with those of 
$\mc O (X_\nr (L)) \rtimes \C [W(L,\mf s),\natural_{\mf s}]$.


\section{The Schwartz algebra of $G$} 
\label{sec:HHSG}

The Harish-Chandra--Schwartz algebra of a reductive $p$-adic group is an inductive limit of
Fr\'echet spaces, but itself not a Fr\'echet algebra. To do homological algebra with such
topological algebras, we have to agree on a suitable topological tensor product. The best choice
is to work in the category of complete bornological vector spaces, with the complete bornological
tensor product \cite[Chapter I]{Mey2}. We denote it by $\hat{\otimes}$, which is reasonable since for 
Fr\'echet algebras it agrees with the projective tensor product \cite[Theorem I.87]{Mey2}. 

The Hochschild homology of a complete bornological algebra $A$ is defined as 
\begin{equation}\label{eq:1.3}
HH_n (A) = \mr{Tor}_n^{A \hat{\otimes} A^{op}} (A,A),
\end{equation}
working in the category of complete bornological $A$-modules. When $A$ is unital, $HH_* (A)$ can be 
computed with the completed bar-complex $C_n (A) = A^{\hat{\otimes} n + 1}$ and the usual differential
\[
b_n (a_0 \otimes \cdots \otimes a_n) = \sum_{i=0}^{n-1} a_0 \otimes \cdots \otimes a_i a_{i+1} \otimes
\cdots \otimes a_n + (-1)^n a_n a_0 \otimes a_1 \otimes \cdots \otimes a_{n-1} .
\]
Under additional conditions, these functors $HH_n$ are continuous:

\begin{lem}\label{lem:1.2}
Suppose that $A = \varinjlim_i A_i$ is a strict inductive limit of nuclear Fr\'echet algebras (where 
strict means that the transition maps $A_i \to A_j$ are injective and have closed range). Then there
is a natural isomorphism 
\[
HH_n (A) \cong \varinjlim\nolimits_i HH_n (A_i) .
\]
\end{lem}
\begin{proof}
In \cite[Theorem 2]{BrPl} this was shown with respect to the inductive tensor product. Under the
assumptions of the lemma, inductive tensor products agree with completed bornological tensor products,
for the $A_i$ and for $A$ \cite[Theorem I.93]{Mey2}.
\end{proof}

Recall that $\mc S (G)$ is the inductive limit of the algebras $\mc S (G,K)$, where $K$ runs over
the compact open subgroups of $G$. As $\mc S (G,K)$ is a closed subspace of $\mc S (G,K')$ when 
$K' \subset K$, $\mc S (G)$ is even a strict inductive limit. The Plancherel isomorphism from 
Theorem \ref{thm:3.1} shows that each $\mc S (G,K)$ is nuclear Fr\'echet algebra. 
Thus Lemma \ref{lem:1.2} applies and says that
\begin{equation}\label{eq:6.1}
HH_n (\mc S (G)) \cong \varinjlim\nolimits_K HH_n (\mc S( G,K)) .
\end{equation}
The decomposition \eqref{eq:5.14} induces a decomposition of the Schwartz algebra of $G$ 
as a direct sum of two-sided ideals:
\[
\mc S (G) = \bigoplus\nolimits_{\mf s \in \mf B (G)} \mc S (G)^{\mf s},
\]
indexed by the inertial equivalence classes $\mf s = [ L,\sigma]_G$ of cuspidal pairs 
for $G$. This persists to the subalgebras of $K$-biinvariant functions:
\[
\mc S (G,K) = \bigoplus\nolimits_{\mf s \in \mf B (G)} \mc S (G,K)^{\mf s},
\]
but for each $K$ only finitely many of the $\mc S (G,K)^{\mf s}$ are nonzero. 
The analogue of \eqref{eq:6.1} for $\mc S (G)^{\mf s}$ reads
\begin{equation}\label{eq:6.2}
HH_n (\mc S (G)^{\mf s}) = HH_n \big( \varinjlim\nolimits_K \mc S (G,K)^{\mf s} \big) 
\cong \varinjlim\nolimits_K HH_n (\mc S( G,K)^{\mf s}) .
\end{equation}
Whenever $\mc H (G)^{\mf s}$ and $\mc H (G,K)^{\mf s}$ are Morita equivalent (which by
\cite{BeDe} happens for arbitrarily small $K$), also $\mc S (G)^{\mf s}$ and 
$\mc S (G,K)^{\mf s}$ are Morita equivalent. Let us denote that situation by 
$\mf s \in \mf B (G,K)$. Then \eqref{eq:6.1} and \eqref{eq:6.2} imply
\begin{equation}\label{eq:6.3}
HH_n (\mc S (G)) \cong \varinjlim_K \bigoplus_{\mf s \in \mf B (G,K)} 
HH_n (\mc S( G,K)^{\mf s}) \cong \bigoplus_{\mf s \in \mf B (G)} HH_n (\mc S (G)^{\mf s})
\end{equation}
and each $HH_n (\mc S (G)^{\mf s})$ is isomorphic with $HH_n (\mc S (G,K)^{\mf s})$
when $\mf s \in \mf B (G,K)$. 

\subsection{Topological algebraic aspects} \
\label{par:top}

We want to determine the Hochschild homology of the nuclear Fr\'echet algebra 
$\mc S (G,K)^{\mf s}$, which is also the closure of $\mc H (G,K)^{\mf s}$ in $\mc S (G)$. We have 
to take the topology of $\mc S (G,K)^{\mf s}$ into account, which creates challenges that were
absent in the purely algebraic setting of Section \ref{sec:HeckeG}. Before we start the actual
computation, we first settle most issues of topological-algebraic nature.

Recall from Theorem \ref{thm:3.1} that there is an isomorphism of Fr\'echet algebras
\begin{equation}\label{eq:6.4}
\mc S (G,K)^{\mf s} \cong \bigoplus\nolimits_{\mf d = (M,\delta) \in \Delta_G^{\mf s}} 
\Big( C^\infty (X_\unr (M)) \otimes \mr{End}_\C \big( I_P^G (V_\delta)^K \big) \Big)^{W(M,\mf d)} .
\end{equation}
By \cite[Th\'eor\`eme 0.1]{Hei}, \eqref{eq:6.4} restricts to an algebra isomorphism
\begin{multline}\label{eq:6.5}
\mc H (G,K)^{\mf s} \cong \big( \mc O (X_\nr (L)) \otimes
\End_\C (I_{P_0 L}^G (V_\sigma)^K ) \big) \; \cap \\ 
\bigoplus\nolimits_{\mf d = (M,\delta) \in \Delta_G^{\mf s}} 
\Big( C^\infty (X_\unr (M)) \otimes \mr{End}_\C \big( I_P^G (V_\delta)^K \big) \Big)^{W(M,\mf d)} .
\end{multline}
Let $e_{\mf d} \in \mc S (G,K)$ be the central idempotent corresponding
to the direct summand of \eqref{eq:6.4} indexed by $\mf d$. We define
$\mc S (G,K)^{\mf d} = e_{\mf d} \mc S (G,K)$, so that by \eqref{eq:6.4}
\begin{equation}\label{eq:6.6}
\mc S (G,K)^{\mf d} \cong \Big( C^\infty (X_\unr (M)) \otimes 
\mr{End}_\C \big( I_P^G (V_\delta)^K \big) \Big)^{W(M,\mf d)} .
\end{equation}
Then $\mc S (G,K)^{\mf s} = \bigoplus_{\mf d \in \Delta_G^{\mf s}} \mc S (G,K)^{\mf d}$ and
\begin{equation}\label{eq:6.7}
Z ( \mc S (G,K)^{\mf s}) = \bigoplus\nolimits_{\mf d \in \Delta_G^{\mf s}} Z (\mc S (G,K)^{\mf d}) 
\cong \bigoplus\nolimits_{\mf d \in \Delta_G^{\mf s}} C^\infty (X_\unr (M))^{W(M,\mf d)} .
\end{equation}
For $\mf d \in \Delta_G^{\mf s}$ we have $W_{\mf d} \subset W_{\mf s}$. The relation between
$W(M,\mf d)$ and $W(L,\mf s)$ is less clear, because the groups $X_\nr (M,\delta)$ and
$X_\nr (L,\sigma)$ from \eqref{eq:3.10} may differ.

To analyse (Fr\'echet) modules over \eqref{eq:6.7}, we will make ample use of the following result.
It is the specialization of \cite[Lemma 3.4]{KaSo} to the affine variety $X_\nr (M)$ with the
submanifold $X_\unr (M)$ and the action of $W(M,\mf d)$.

\begin{prop}\label{prop:6.12}
Let $\tilde Y$ be an affine variety with an embedding $\imath$ in $X_\nr (M)$. Suppose that:
\begin{itemize}
\item $\imath (\tilde Y)$ is closed in $X_\nr (M)$ and isomorphic to $\tilde Y$,
\item $Y := \imath^{-1} (X_\unr (M))$ is a real analytic Zariski-dense submanifold of 
$\tilde Y$ and diffeomorphic to $\imath (Y)$.
\end{itemize}
Let $p$ be an idempotent in the ring of continuous $C^\infty (M)^{W(M,\mf d)}$-linear
endomorphisms of $\Omega^n_{sm}(Y)$, such that $p$ stabilizes $\Omega^n (\tilde Y)$. Then the
natural map
\[
C^\infty (X_\unr (M))^{W(M,\mf d)} \underset{\mc O (X_\nr (M))^{W(M,\mf d)}}{\otimes} 
p \Omega^n (\tilde Y) \longrightarrow p \Omega^n_{sm} (Y) 
\]
is an isomorphism of Fr\'echet $C^\infty (X_\unr (M))^{W(M,\mf d)}$-modules.
\end{prop}

The dense subspace $e_{\mf d}\mc H (G,K)^{\mf s}$ of $\mc S (G,K)^{\mf d}$ is a subalgebra
because $e_{\mf d}$ is central, but it is not contained in $\mc H (G,K)$. Its irreducible 
representations are the constituents of the $I_P^G (\delta \otimes \chi)$ with $\chi \in X_\nr (M)$. 
Its centre is the restriction of $Z(\mc H (G,K)^{\mf s})$ to $\mc F (M,\delta)$, so
\begin{equation}\label{eq:6.39}
Z( e_{\mf d} \mc H (G,K)^{\mf s}) = e_{\mf d} Z(\mc H (G,K)^{\mf s}) \cong
\mc O (X_\nr (M))^{W(M,\mf d)} .
\end{equation}

\begin{lem}\label{lem:6.2}
\enuma{
\item The multiplication map
\[
\mu_{\mf d} : Z (\mc S (G,K)^{\mf d}) \underset{e_{\mf d} Z (\mc H (G,K)^{\mf s})}{\otimes}
e_{\mf d} \mc H (G,K)^{\mf s} \longrightarrow \mc S (G,K)^{\mf d}
\]
is an isomorphism of Fr\'echet $Z(\mc S (G,K)^{\mf d})$-modules. 
\item The $Z (\mc S (G,K)^{\mf d})$-module $\mc S (G,K)^{\mf d}$ is generated by a finite 
subset of $e_{\mf d} \mc H (G,K)^{\mf s}$.
}
\end{lem}
\begin{proof}
(a) From \eqref{eq:6.6} and \eqref{eq:6.7} we see that that $Z(\mc S (G,K)^{\mf d})$-module
$\mc S (G,K)$ is a direct summand of
\[
C^\infty (X_\unr (M)) \otimes \mr{End}_\C \big( I_P^G (V_\delta)^K \big) \cong  
C^\infty (X_\unr (M))^{\dim (I_P^G (V_\delta)^K )^2} ,
\]
namely, it is the image of the idempotent the averages over $W(M,\mf d)$.
It follows from \eqref{eq:6.5} that
\begin{multline}\label{eq:6.8}
e_{\mf d} \mc H (G,K)^{\mf s} \cong \big( \mc O (X_\nr (M)) \otimes \mr{End}_\C 
\big( I_P^G (V_\delta)^K \big) \big) \; \cap \\
\Big( C^\infty (X_\unr (M)) \otimes \mr{End}_\C \big( I_P^G (V_\delta)^K \big) \Big)^{W(M,\mf d)} .
\end{multline}
From this and \eqref{eq:6.6},\eqref{eq:6.7} and \eqref{eq:6.39} we see that we are in the right
position to apply Proposition \ref{prop:6.12}, which yields exactly the statement.\\
(b) This follows from part (a) and \eqref{eq:6.9}. 
\end{proof}

We note that there are natural homomorphisms of $Z(\mc H (G,K)^{\mf s})$-modules
\[
HH_n (\mc H (G,K)^{\mf s})^{\mf d} \to HH_n (\mc H (G,K)^{\mf s}) \to 
HH_n (e_{\mf d} \mc H (G,K)^{\mf s}) ,
\]
and the outer sides should be closely related. However, the composed map is in general not 
bijective, $HH_n (e_{\mf d} \mc H (G,K)^{\mf s})$ can be more intricate.

As discussed around \eqref{eq:1.3}, we compute the Hochschild homology of Fr\'echet algebras with 
respect to the complete projective tensor product. We establish some topological properties of the 
Hochschild homology groups of $\mc S (G,K)^{\mf s}$, making use of \cite{KaSo}.

\begin{prop}\label{prop:6.6}
$HH_n (\mc S (G,K)^{\mf s})$ is a quotient of two closed submodules of a finitely generated
Fr\'echet $Z(\mc S (G,K)^{\mf s})$-module. In particular $HH_n (\mc S (G,K)^{\mf s})$ is a Fr\'echet
$Z(\mc S (G,K)^{\mf s})$-module.
\end{prop}
\begin{proof}
To compute $HH_n (\mc S(G,K)^{\mf s})$ according to the definition \eqref{eq:1.3}, we can use any
(bornological or Fr\'echet) projective bimodule resolution of $\mc S (G,K)^{\mf s}$. One such
resolution was constructed in \cite[Theorem 4.2]{OpSo1}, for $\mc S (G,K)$ but that is enough 
because $\mc S (G,K)^{\mf s}$ is a direct summand of $\mc S (G,K)$. The set of $n$-chains of 
that resolution is a finitely generated projective $\mc S (G,K)^{\mf s} \hat{\otimes} 
\mc S (G,K)^{\mf s,op}$-module. By construction this projective resolution contains a projective 
bimodule resolution of $\mc H (G,K)^{\mf s}$, namely the set of elements that live in powers of
$\mc H (G,K)^{\mf s} \hat{\otimes} \mc H (G,K)^{\mf s,op}$.

By tensoring with $\mc S (G,K)^{\mf s}$ over $\mc S (G,K)^{\mf s} \hat{\otimes} 
\mc S (G,K)^{\mf s,op}$, we obtain a differential complex $(C_*,d_*)$ that computes 
$HH_n (\mc S (G,K)^{\mf s})$. Each term $C_n$ is a direct summand of
\[
(\mc S (G,K)^{\mf s} \hat{\otimes} \mc S (G,K)^{\mf s,op})^r 
\underset{\mc S (G,K)^{\mf s} \hat{\otimes} \mc S (G,K)^{\mf s,op}}{\otimes} \mc S (G,K)^{\mf s}
\cong (\mc S (G,K)^{\mf s})^r ,
\]
for some $r \in \N$. By Theorem \ref{thm:3.1} and \cite[Theorem 3.1.b]{KaSo}, 
$(\mc S (G,K)^{\mf s})^r$ and its direct summand $C_n$ are finitely generated Fr\'echet 
$Z(\mc S (G,K)^{\mf s})$-modules. The set of $n$-cycles $Z_n$ is closed in $C_n$ 
(by the continuity of the boundary map) and hence closed in $(\mc S (G,K)^{\mf s})^r$. 

The intersection $C'_n$ of $C_n$ with $(\mc H (G,K)^{\mf s})^r$ is a finitely generated \\
$Z(\mc H (G,K)^{\mf s})$-module, by \eqref{eq:6.9}. That yields a differential complex $(C'_n,d_n)$ 
which computes $HH_n (\mc H (G,K)^{\mf s})$. 

Choose a finite set $Y_n \subset (\mc H (G,K)^{\mf s})^r$ that generates $C'_n$ as 
$Z(\mc H (G,K)^{\mf s}$-module. With Lemma \ref{lem:6.2} we see that $Y_n$ also generates $C_n$ as 
$Z(\mc S (G,K)^{\mf s})$-module. As the boundary map $d_n$ is  $Z(\mc S (G,K)^{\mf s})$-linear, 
the set of $n$-boundaries $B_n = d_n (C_{n-1})$ is generated as $Z(\mc S (G,K)^{\mf s})$-module by 
$d_n (Y_{n-1})$. There are inclusions
\begin{equation}\label{eq:6.15}
B_n \subset Z_n \subset C_n \subset \Big( \Big( C^\infty (X_\unr (M)) \otimes 
\mr{End}_\C \big( I_P^G (V_\delta)^K \big) \Big)^{W(M,\mf d)} \Big)^r .
\end{equation}
We want to show that $B_n = Z(\mc S (G,K)^{\mf s}) d(Y_{n-1})$ is a closed subspace of the right
hand side, just like $Z_n$ and $C_n$. The right hand side of \eqref{eq:6.15} embeds as 
$C^\infty (X_\unr (M))^{W(M,\mf d)}$-module in  
\[
C^\infty (X_\unr (M))^{r'}, \text{ where } r' = r \dim \big( I_P^G (V_\delta)^K \big)^2 .
\]
Via this embedding the elements of $d(Y_{n-1})$ become analytic (in fact algebraic) functions on 
$X_\unr (M) \times \{1,\ldots,r'\}$. By \cite[Corollaire V.1.6]{Tou}, generalized to an 
$W(M,\mf d)$-invariant setting in \cite[Theorem 1.2]{KaSo}, the finite set $d(Y_{n-1})$
generates a closed $C^\infty (X_\unr (M))^{W(M,\mf d)}$-submodule $C^\infty (X_\unr (M))^{r'}$. 
Hence $B_n$ is closed in any of the modules from \eqref{eq:6.15}. Now
\eqref{eq:6.15} and $HH_n (\mc S (G,K)^{\mf s}) = Z_n / B_n$ provide the required properties.
\end{proof}

We use the same algebraic families $\mf F (M_i,\eta_i)$ in $\Rep (G)^{\mf s}$ as in
Paragraph \ref{par:HHHG}. Recall from Definition \ref{def:3} that $\mf F (M_i,\eta_i)$ naturally
contains a tempered algebraic family
\[
\mf F^t (M_i,\eta_i) = \{ I_{P_i}^G (\eta_i \otimes \chi_i) : \chi_i \in X_\unr (M_i) \} .
\]
By Theorem \ref{thm:3.1} it gives rise to a homomorphism of Fr\'echet $Z(\mc S(G,K)^{\mf s})$-algebras 
\begin{equation}\label{eq:6.10}
\mc F^t_{M_i,\eta_i} : \mc S (G,K)^{\mf s} \to
C^\infty (X_\unr (M_i)) \otimes \End_\C \big( I_{P_i}^G (V_{\eta_i})^K \big) .
\end{equation}
Here $Z(\mc S(G,K)^{\mf s})$ acts on the right hand side via evaluations at the central characters
of the underlying $G$-representations $\pi (M_i,\eta_i,\chi_i)$. In terms of \eqref{eq:6.7}, the
direct summands 
\[
Z (\mc S (G,K)^{\mf d}) = C^\infty (X_\unr (M))^{W(M,\mf d)}
\] 
of $Z(\mc S (G,K)^{\mf s})$ annihilate the range of \eqref{eq:6.10} when $i \not\prec \mf d$. 
When $i \prec \mf d$, pick $\chi_{\eta_i,\delta}$ so that $\eta_i$ is a subquotient of 
$I_P^G (\delta \otimes \chi_{\eta_i,\delta})$. Then $C^\infty (X_\unr (M))^{W(M,\mf d)}$ acts
on the range of \eqref{eq:6.10} via the map
\begin{equation}\label{eq:6.11}
X_\unr (M_i) \to X_\unr (M) : \chi \mapsto \chi_{\eta_i,\delta} \chi .
\end{equation}
We note that
\[
\mc F_{\mf d}^t := \bigoplus\nolimits_{i=1, i \prec \mf d}^{n_{\mf s}} \mc F^t_{M_i,\eta_i}
\]
annihilates all the $\mc S (G,K)^{\mf d'}$ with $\mf d' \in \Delta_G^{\mf s} \setminus \{\mf d\}$.
We write
\[
\mc F_{\mf s}^t = \bigoplus\nolimits_{\mf d \in \Delta_G^{\mf s}} \mc F_{\mf d}^t =
\bigoplus\nolimits_{i=1}^{n_{\mf s}} \mc F^t_{M_i,\eta_i} .
\]
These Fr\'echet algebra homomorphisms induce homomorphisms of Fr\'echet \\
$Z(\mc S(G,K)^{\mf s})$-modules
\begin{align*}
& HH_n (\mc F_{\mf d}^t ) \;:\; HH_n (\mc S (G,K)^{\mf d}) \to
\bigoplus\nolimits_{i=1, i \prec \mf d}^{n_{\mf s}} \Omega^n_{sm} (X_\unr (M_i)) ,\\
& HH_n (\mc F_{\mf s}^t) = \bigoplus\nolimits_{\mf d \in \Delta_G^{\mf s}} HH_n (\mc F_{\mf d}^t ) :
HH_n (\mc S (G,K)^{\mf s}) \to \bigoplus\nolimits_{i=1}^{n_{\mf s}} \Omega^n_{sm} (X_\unr (M_i)) . 
\end{align*}
The algebra inclusion $Z (\mc H (G,K)^{\mf s}) \to Z(\mc S (G,K)^{\mf s})$ induces a surjection
\[
\mr{pr} : \bigsqcup\nolimits_{\mf d = [M,\delta] \in \Delta_G^{\mf s}} X_\unr (M) / W(M,\mf d) \to
X_\nr (L) / W(L,\mf s) .
\]

\begin{prop}\label{prop:6.3}
For $\xi \in X_\unr (M) / W(M,\mf d)$ the following spaces are naturally isomorphic:
\begin{enumerate}[(i)]
\item the formal completion of $HH_n (\mc S (G,K)^{\mf d})$ at $\xi$,
\item the formal completion of $HH_n (e_{\mf d} \mc H (G,K)^{\mf s})$ at pr$(\xi)$,
\item the formal completion of $HH_n (\mc H (G,K)^{\mf s})^{\mf d}$ (as in Lemma \ref{lem:5.5}) 
at pr$(\xi)$.
\end{enumerate}
\end{prop}
We remark that (ii) and (iii) need not be isomorphic for more general central characters
(e.g. the central character of $\pi (M,\delta,\chi)$ with $\chi \in X_\nr (M)$ not unitary).
\begin{proof}
Let $FP_\xi^{W(M,\mf d)} \cong FP_\lambda^{W(M,\mf d)_\lambda}$ be the algebra of 
$W(M,\mf d)$-invariant formal power series on $(M,\delta,X_\unr (M))$ centred at $\xi$. It is
naturally isomorphic to the formal completion of 
\begin{equation}\label{eq:6.36}
\mc O (X_\nr (M))^{W(M,\mf d)} \cong Z (e_{\mf d} \mc H (G,K)^{\mf s})
\end{equation}
at $W(M,\mf d) \lambda$. By \cite[Theorem 2.5]{OpSo2} the functor
\begin{equation}\label{eq:6.19}
FP_\xi^{W(M,\mf d)} \underset{Z(\mc S (G,K)^{\mf d})}{\otimes} =  
FP_\lambda^{W(M,\mf d)_\lambda} \underset{C^\infty (M)^{W(M,\mf d)}}{\otimes}
\end{equation}
is exact on a large class of $Z(\mc S (G,K)^{\mf d})$-modules. This class contains all modules
which as topological vector spaces are quotients of $\mc S (\Z)$, and all modules that we
need here are of that form. This exactness implies that 
\begin{equation}\label{eq:6.13}
FP_\xi^{W(M,\mf d)} \underset{Z(\mc S (G,K)^{\mf d})}{\otimes} HH_n (\mc S (G,K)^{\mf d}) \cong
H_n \big( FP_\xi^{W(M,\mf d)} \underset{Z(\mc S (G,K)^{\mf d})}{\otimes} 
C_* (\mc S (G,K)^{\mf d}), b_* \big).
\end{equation}
Let $I_{W(M,\mf d) \lambda}^\infty \subset C^\infty (X_\unr (M))^{W(M,\mf d)}$ be the ideal of 
$W(M,\mf d)$-invariant smooth functions that are flat at $W(M,\mf d) \lambda$.
By a theorem of Borel, see \cite[Th\'eor\`eme IV.3.1 and Remarque IV.3.5]{Tou}, the Taylor series map
$C^\infty (M) \to FP_v$ is surjective. Taking $W(M,\mf d)$-invariants leads to an isomorphism
\[
FP_\xi^{W(M,\mf d)} \cong C^\infty (X_\unr (M))^{W(M,\mf d)} / I_{W(M,\mf d) \lambda}^\infty .
\]
In particular the ideal $ I_{W(M,\mf d) \lambda}^\infty$ annihilates \eqref{eq:6.13}. 
With Lemma \ref{lem:6.2}.a we find that \eqref{eq:6.13} is isomorphic with
\begin{multline}\label{eq:6.14}
H_n \big( C_* \big( FP_\xi^{W(M,\mf d)} \underset{Z(\mc S (G,K)^{\mf d})}{\otimes} 
\mc S (G,K)^{\mf d} \big) , b_* \big) \cong \\
H_n \big( C_* \big( FP_\xi^{W(M,\mf d)} \underset{Z(\mc S (G,K)^{\mf d})}{\otimes} Z(\mc S (G,K)^{\mf d})
\underset{e_{\mf d} Z(\mc H (G,K)^{\mf s}}{\otimes}  e_{\mf d} \mc H (G,K)^{\mf s} \big) , b_* \big) \cong \\
H_n \big( C_* \big( FP_\xi^{W(M,\mf d)} \underset{e_{\mf d} Z(\mc H (G,K)^{\mf s})}{\otimes}  e_{\mf d} 
\mc H (G,K)^{\mf s} \big) , b_* \big) \cong \\
HH_n \big( FP_\xi^{W(M,\mf d)} \underset{e_{\mf d} Z(\mc H (G,K)^{\mf s})}{\otimes}  e_{\mf d} 
\mc H (G,K)^{\mf s} \big) .
\end{multline}
By \eqref{eq:6.36} at the exactness of 
$FP_\xi^{W(M,\mf d)} \underset{\mc O (X_\nr (M))^{W(M,\mf d)}}{\otimes}$, the last expression can be
identified with the formal completion of $HH_n (e_{\mf d} \mc H (G,K)^{\mf s})$ at pr$(\xi)$.
Hence (i) and (ii) are naturally isomorphic.

Next we apply $HH_n (\mc F_{\mf d}^t)$ to the last line of \eqref{eq:6.14}, with image in
\[
FP_\lambda^{W(M,\mf d)_\lambda} \underset{C^\infty (M)^{W(M,\mf d)}}{\otimes} 
\bigoplus_{i=1, i \prec \mf d}^{n_{\mf s}} \Omega^n_{sm} (X_\unr (M_i)) .
\]
As $\mc F_{\mf d}^t (e_{\mf d}) = 1$, we may just as well set $\mc F_{\mf d}(e_{\mf d}) = 1$ and 
apply $HH_n (\mc F_{\mf d})$. Then the image becomes
\begin{multline}
HH_n (\mc F_{\mf d}) HH_n \Big( \widehat{\mc O (X_\nr (M))}_{W(M,\mf d)\lambda}^{W(M,\mf d)}
\underset{e_{\mf d} Z(\mc H (G,K)^{\mf s})}{\otimes} e_{\mf d} \mc H (G,K)^{\mf s} \Big) \cong \\
\widehat{\mc O (X_\nr (M))}_{W(M,\mf d)\lambda}^{W(M,\mf d)}
\underset{e_{\mf d} Z(\mc H (G,K)^{\mf s})}{\otimes} HH_n (\mc F_{\mf d}) HH_n (\mc H (G,K)^{\mf s}) = \\
\widehat{\mc O (X_\nr (L))}^{W(L,\mf s)}_{\mr{pr}(\xi)} \underset{Z(\mc H (G,K)^{\mf s})}{\otimes}
HH_n (\mc F_{\mf d}) HH_n (\mc H (G,K)^{\mf s}) .
\end{multline}
To the last expression we apply $\mr{id} \otimes HH_n (\mc F_{\mf s})^{-1}$ (which exists by
Lemma \ref{lem:5.6}), and we obtain the desired
description of \eqref{eq:6.14} and of \eqref{eq:6.13}.
\end{proof}

The injectivity of $HH_n (\mc F_{\mf s}^t)$ is more subtle for these topological algebras
than it was in the earlier purely algebraic settings (e.g. Lemma \ref{lem:5.6}).

\begin{lem}\label{lem:6.7}
\enuma{
\item The continuous $Z(\mc S (G,K)^{\mf s})$-linear map
\[
HH_n (\mc F_{\mf s}^t) : HH_n (\mc S (G,K)^{\mf s}) \to 
\bigoplus\nolimits_{i=1}^{n_{\mf s}} \Omega^n_{sm} (X_\unr (M_i)) 
\]
is injective.
\item The natural map $HH_n (\mc H (G,K)^{\mf s}) \to HH_n (\mc S (G,K)^{\mf s})$ is injective. 
}
\end{lem}
\begin{proof}
(a) The kernel of $HH_n (\mc F_{\mf s}^t)$ is a closed $Z(\mc S(G,K)^{\mf s})$-submodule of \\
$HH_n (\mc S (G,K)^{\mf s})$, so by Proposition \ref{prop:6.6} it is a quotient of two closed 
submodules of a finitely generated Fr\'echet $Z(\mc S (G,K)^{\mf s})$-module.
Using the central idempotents $e_{\mf d}$, we can decompose
\[
\ker HH_n (\mc F_{\mf s}^t) = 
\bigoplus\nolimits_{\mf d \in \Delta_G^{\mf s}} \ker HH_n (\mc F_{\mf d}^t) . 
\]
Here each $HH_n (\mc F_{\mf d}^t)$ is a quotient of two closed submodules of a finitely
generated Fr\'echet $C^\infty (X_\unr (M))^{W(M,\mf d)}$-module. Suppose that $\ker HH_n 
(\mc F_{\mf d}^t)$ is nonzero for one specific $\delta$. By \cite[Lemma 1.1]{KaSo} at least one 
of its formal completions is nonzero, say at $\xi = W(M,\mf d) (M,\delta,\chi)$. By Proposition 
\ref{prop:6.3} that formal completion of $\ker HH_n (\mc F_{\mf d}^t)$ can be considered as 
a submodule of the formal completion of $HH_n (\mc H (G,K)^{\mf s})^{\mf d}$ at pr$(\xi)$.

From Theorem \ref{thm:5.3} and Lemma \ref{lem:5.5} we know that $HH_n (\mc F_{\mf d})$ is
injective on\\ $\widehat{HH_n (\mc H (G,K)^{\mf s})^{\mf d}}_{\mr{pr}(\xi)}$. That holds for
$HH_n (\mc F_{\mf d}^t)$ as well, because $\mc F_{\mf d}^t = \mc F_{\mf d}$ on these formal
completions. Hence 
\[
HH_n (\mc F_{\mf d}) (\ker HH_n (\mc F_{\mf s}^t))
\]
has a nonzero formal completion at $\xi$, which is clearly a contradiction.\\
(b) The algebra homomorphism $\mc F_{\mf s}^t$ extends
\[
\mc F_{\mf s} : \mc H (G,K)^{\mf s} \to 
\bigoplus\nolimits_{i=1}^{n_{\mf s}} \Omega^n (X_\nr (M_i)) \otimes_\C \End_\C (I_{P_i}^G (\eta_i)^K) .
\]
By Lemma \ref{lem:5.6} the map $HH_n (\mc F_{\mf s})$, is injective, just as $HH_n (\mc F_{\mf s}^t)$.
Hence the natural map 
\[
HH_n (\mc H (G,K)^{\mf s}) \to HH_n (\mc S (G,K)^{\mf s})
\] 
equals the injection $HH_n (\mc F_{\mf s}^t)^{-1} \circ HH_n (\mc F_{\mf s})$.
\end{proof}

\subsection{Computation of Hochschild homology} \
\label{par:computation}

For $u \in X_\unr (L)$ and $g \in W (L,\mf s)_u$, the maps $\phi_{g,i}$ from \cite[(2.12)]{SolTwist} 
and $\chi_{\eta_i}^{-1} \phi_{g,i}^*$ from \eqref{eq:5.15} are well-defined in this setting, 
only now as
\[
\chi_{\eta_i}^{-1} \phi_{g,i}^* : 
C^\infty (X_{\unr}(M_i)) \otimes \End_\C (I_{P_i}^G (V_{\eta_i})^K ) \to
C^\infty (X_{\unr}(L)^{g,\circ}) \otimes \End_\C (I_{P_i}^G (V_{\eta_i})^K ) .
\]
This enables to define the smooth version of \eqref{eq:5.26}:
\[
HH_n (\phi_{w,c}^*) = \sum_{i=1}^{n_{\mf s}} HH_n (\chi_{\eta_i}^{-1} \phi_{w,i}^*) \;:\;
\bigoplus_{i=1, U_u\text{-rel}}^{n_{\mf s}} \Omega^n_{sm} (X_\unr (M_i)) \to 
\Omega^n_{sm} (X_\unr (L)^w_c) .
\]
Like in Theorem \ref{thm:5.9}, we define
\[
HH_n (\tilde \phi_{\mf s}^*) = \bigoplus\nolimits_{w \in W(L,\mf s), c \pi_0 (X_\nr (L)^w_c)}
HH_n (\phi_{w,c}^*) .
\]

\begin{lem}\label{lem:6.9}
The continuous map
\[
HH_n (\tilde \phi_{\mf s}^*) : \bigoplus_{i=1}^{n_{\mf s}} \Omega^n_{sm} (X_\unr (M_i)) 
\to \bigoplus_{w \in W(L,\mf s), c \in \pi_0 (X_\nr (L)^w)} \Omega^n (X_\unr (L)^w_c)
\]
is injective.
\end{lem}
\begin{proof}
This can be shown in the same way as Theorem \ref{thm:5.9}.a. Ultimately the argument relies
on \cite[Lemma 1.12]{SolTwist} which holds just as well in a smooth settting, as explained on
\cite[p. 21]{SolTwist}.
\end{proof}

We fix $\mf d = [M,\delta]$ and we represent the central character of $\delta$ by $\chi_\delta
t_\delta^+$ with $\chi_\delta \in X_\unr (L)$ and $t_\delta^+ \in X_\nr^+ (L)$. We let 
\[
Z(\mc S (G,K)^{\mf d}) \cong C^\infty (X_\unr (M))^{W(M,\mf d)} \quad \text{act on} \quad
\bigoplus\nolimits_{w \in W(L,\mf s)} \Omega^n_{sm} (X_\unr (L)^w) 
\]
in the following way:
\begin{itemize}
\item if $g (X_\nr (L)^w_c) \subset \chi_\delta X_\nr (M)$ for some $g,w \in W(L,\mf s)$, then
it acts at $(w,\chi)$ with $\chi \in X_\unr (L)^w_c$ via the character 
$W(M,\mf d) \chi_\delta^{-1} g \chi$,
\item if $W(L,\mf s) X_\nr (L)^w_c \not\subset W(L,\mf s) \chi_\delta X_\nr (M)$, then 
$Z(\mc S (G,K)^{\mf d})$ acts as zero on $\Omega^n_{sm} (X_\unr (L)^w_c)$.
\end{itemize}

\begin{lem}\label{lem:6.8}
The following map is $Z(\mc S (G,K)^{\mf d})$-linear:
\[
HH_n (\tilde \phi_{\mf s}^*) \circ HH_n (\mc F_{\mf d}^t) \;:\;
HH_n (\mc S (G,K)^{\mf d}) \to \bigoplus\nolimits_{w \in W(L,\mf s)} \Omega^n_{sm} (X_\unr (L)^w) .
\]
\end{lem}
\begin{proof}
From Lemma \ref{lem:5.5}.b and Theorem \ref{thm:5.9}.c we know how $Z(\mc H (G,K)^{\mf s})$ 
acts on
\[
HH_n (\tilde \phi_{\mf s}^*) HH_n (\mc F_\delta) HH_n (\mc H (G,K)^{\mf s}) \; \subset \;
\bigoplus\nolimits_{w \in W(L,\mf s)} \Omega^n (X_\nr (L)^w) .
\]
That action is pointwise, in the sense that upon specialization at any point of
$X_\nr (L)^w$ the $Z(\mc H (G,K)^{\mf s})$-action goes via evaluation at a character
(or is just zero). Via the natural map
\begin{equation}\label{eq:6.35}
Z( \mc H (G,K)^{\mf s}) \to Z( \mc S(G,K)^{\mf d}) : f \mapsto e_{\mf d} f ,
\end{equation}
such an action naturally gives rise to an action of $C^\infty (X_\unr (M))^{W(M,\mf d)}$ on
\[
\bigoplus\nolimits_{w \in W(L,\mf s)} \Omega^n_{sm} (X_\unr (L)^w) ,
\]
which is pointwise in the same sense. The map
\[
X_\unr (M) / W(M,\mf d) \to X_\nr (L) / W(L,\mf s)
\]
induced by \eqref{eq:6.35} sends $W(M,\mf d) \chi$ to $W(L,\mf s) \chi \chi_\delta t_\delta^+$.
When we compare this with Lemma \ref{lem:5.5}.b, we see that 
$Z(\mc S (G,K)^{\mf s})$ acts in the way described just before the lemma.
\end{proof}

Next we prove the most technical step towards our description of $HH_n (\mc S (G,K)^{\mf s})$.

\begin{lem}\label{lem:6.5}
The $Z(\mc S (G,K)^{\mf d})$-module $HH_n (\mc F_{\mf d}^t) HH_n (\mc S (G,K)^{\mf d})$ contains
\[
\bigoplus_{i=1, i \prec \mf d}^{n_{\mf s}} \!\! \Omega^n_{sm} (X_\unr (M_i)) 
\bigcap HH_n (\tilde \phi_{\mf s}^* )^{-1} \Big( \bigoplus_{w \in W(L,\mf s)} \!\!\!
\Omega^n_{sm} (X_\unr (L)^w) \otimes \natural_{\mf s}^w \Big)^{W(L,\mf s)} .
\]
\end{lem}
\begin{proof}
From Lemma \ref{lem:6.8} we know that $HH_n (\phi_{\mf s}^*)$
becomes $Z(\mc S (G,K)^{\mf d})$-linear if we restrict its domain to the summands 
$\Omega^n (X_\unr (M_i))$ with $i \prec \mf d$. We consider
\begin{equation}\label{eq:6.40}
HH_n (\tilde \phi_{\mf s}^* ) \Big( \bigoplus_{i=1, i \prec \mf d}^{n_{\mf s}} \!\! 
\Omega^n (X_\nr (M_i)) \Big) \bigcap \Big( \bigoplus_{w \in W(L,\mf s)} \!\!\!
\Omega^n (X_\nr (L)^w) \otimes \natural_{\mf s}^w \Big)^{W(L,\mf s)}
\end{equation}
If we take the direct sum over $\mf d \in \Delta_G^{\mf s}$, then by Theorem \ref{thm:5.9} 
we obtain precisely the term on the right hand side. By continuous extension we find that 
the direct sum over $\mf d \in \Delta_G^{\mf s}$ of the spaces
\begin{equation}\label{eq:6.41}
HH_n (\tilde \phi_{\mf s}^* ) \Big( \bigoplus_{i=1, i \prec \mf d}^{n_{\mf s}} \!\! 
\Omega^n_{sm} (X_\unr (M_i)) \Big) \bigcap \Big( \bigoplus_{w \in W(L,\mf s)} \!\!\!
\Omega^n_{sm} (X_\unr (L)^w) \otimes \natural_{\mf s}^w \Big)^{W(L,\mf s)}
\end{equation}
equals \eqref{eq:6.32}. Lemma \ref{lem:6.8} tells us that \eqref{eq:6.32} is a 
$Z(\mc S (G,K)^{\mf d})$-submodule of
\begin{equation}\label{eq:6.42}
\bigoplus\nolimits_{w \in W(L,\mf s)} \Omega^n_{sm} (X_\unr (L)^w),
\end{equation} 
In fact it is direct summand, namely the image of the idempotent 
\[
|W(L,\mf s)|^{-1} \sum\nolimits_{w \in W(L,\mf s)} w.
\] 
Hence there exists a continuous idempotent $C^\infty (X_\unr (M))^{W(M,\mf d)}$-linear
endomorphism $p$ of\eqref{eq:6.42} with image \eqref{eq:6.41}. Although the 
$Z(\mc S (G,K)^{\mf d})$-action on \eqref{eq:6.42} annihilates some of the direct summands, 
that is not a problem because the action of $Z(\mc S (G,K)^{\mf d})$ on the subspace 
\eqref{eq:6.41} is induced by embeddings $X_\unr (L)^w \to X_\unr (M)$ as usual. We note that
\[
p \Big( \bigoplus\nolimits_{w \in W(L,\mf s)} \Omega^n (X_\nr (L)^w) \Big) = \eqref{eq:6.40}. 
\]
Now we can apply Proposition \ref{prop:6.12}, which yields an isomorphism of Fr\'echet \\
$C^\infty (X_\unr (M))^{W(M,\mf d)}$-modules
\begin{equation}\label{eq:6.43}
C^\infty (X_\unr (M))^{W(M,\mf d)} \underset{\mc O (X_\nr (M)^{W(M,\mf d)}}{\otimes}
\eqref{eq:6.40} \; \cong \; \eqref{eq:6.41} .
\end{equation}
With the injectivity and the $C^\infty (X_\unr (M))^{W(M,\mf d)}$-linearity of 
$HH_n (\tilde \phi_{\mf s}^* )$ we find that the $C^\infty (X_\unr (M))^{W(M,\mf d)}$-module
in the statement is generated by
\begin{equation} \label{eq:6.30} 
\bigoplus_{i=1, i \prec \mf d}^{n_{\mf s}} \!\! \Omega^n (X_\nr (M_i)) 
\bigcap HH_n (\tilde \phi_{\mf s}^* )^{-1} \Big( \bigoplus_{w \in W(L,\mf s)} \!\!\!
\Omega^n (X_\nr (L)^w) \otimes \natural_{\mf s}^w \Big)^{W(L,\mf s)} .
\end{equation}
By Theorem \ref{thm:5.9}.b $HH_n (\mc F_{\mf d}) HH_n (\mc H (G,K)^{\mf s})$ contains 
\eqref{eq:6.30}, so\\ $HH_n (\mc F_{\mf d}^t) HH_n (\mc S (G,K)^{\mf s})$ contains \eqref{eq:6.30} 
as well. Hence the $C^\infty (X_\unr (M))^{W(M,\mf d)}$-module
$HH_n (\mc F_{\mf d}^t) HH_n (\mc S (G,K)^{\mf s})$ contains the module in the statement.
\end{proof}

Everything is in place to establish a smooth version of Theorem \ref{thm:5.9}.

\begin{thm}\label{thm:6.4}
$HH_n (\tilde \phi_{\mf s}^*) \circ HH_n (\mc F_{\mf s})$ gives an isomorphism of 
Fr\'echet spaces
\[
HH_n (\mc S (G,K)^{\mf s}) \to \Big( \bigoplus\nolimits_{w \in W(L,\mf s)} 
\Omega^n_{sm} (X_\unr (L)^w) \otimes \natural_{\mf s}^w \Big)^{W(L,\mf s)} .
\]
\end{thm}
\begin{proof}
Evaluation of 
\begin{equation}\label{eq:6.12}
HH_0 (\tilde \phi_{\mf s}^*) HH_0 (\mc F_{\mf s}) HH_0 (\mc H (G,K)^{\mf s})
\end{equation}
at $(w,\chi)$ corresponds to the map on $HH_0$ induced by the virtual representation
$\nu^1_{w,\chi}$ of $\mc H (G,K)^{\mf s}$ from \eqref{eq:5.31}. If we evaluate at a family of 
$\chi$'s simultaneously, that interpretation becomes valid and nontrivial in degrees $n>0$ as 
well. The $W(L,\mf s)$-invariance of \eqref{eq:6.12} (and its versions in degrees $n > 0$) in 
Theorem \ref{thm:5.9}.b is a consequence of:
\begin{itemize}
\item the relations \eqref{eq:5.29} between these virtual representations, 
\item the fact the Hochschild homology does not distinguish equivalent virtual representations 
\cite[Lemma 1.7]{SolTwist}. 
\end{itemize}
Our maps in the smooth setting are basically the same as the earlier maps in an algebraic setting, 
only restricted to tempered representations and allowing for smooth functions. Therefore
\[
HH_n (\tilde \phi_{\mf s}^*) HH_n (\mc F_{\mf s}^t) HH_n (\mc S (G,K)^{\mf s})
\]
also consists of $W(L,\mf s)$-invariant elements. More explicitly, it is contained in
\begin{equation}\label{eq:6.32}
\Big( \bigoplus\nolimits_{w \in W(L,\mf s)} 
\Omega^n_{sm} (X_\unr (L)^w) \otimes \natural_{\mf s}^w \Big)^{W(L,\mf s)} .
\end{equation}
Comparing this with Lemma \ref{lem:6.5}, we deduce that the inclusion in Lemma \ref{lem:6.5}
is in fact an equality. Now \eqref{eq:6.41} entails that
\[
HH_n (\tilde \phi_{\mf s}^*) HH_n (\mc F_{\mf s}^t) HH_n (\mc S (G,K)^{\mf s}) =
\bigoplus_{\mf d \in \Delta_G^{\mf s}} HH_n (\tilde \phi_{\mf s}^*) 
HH_n (\mc F_{\mf s}^t) HH_n (\mc S (G,K)^{\mf s}) 
\]
equals the right hand side of \eqref{eq:6.32}. Thus $HH_n (\tilde \phi_{\mf s}^*) 
HH_n (\mc F_{\mf s}^t)$ is a continuous bijection between the Fr\'echet spaces 
$HH_n (\mc S (G,K)^{\mf s})$ and \eqref{eq:6.32}. By the open mapping theorem, it is 
an isomorphism of Fr\'echet spaces.
\end{proof}

Like in Proposition \ref{prop:5.4}, there is an alternative description of $HH_0 (\mc S (G))$.
We note that $HH_0 (\mc H (G))$ can also be described with tempered representations only
\cite{Mui}, like in the following proposition.

\begin{prop}\label{prop:6.10}
\enuma{
\item For $\mf d = [M,\delta] \in \Delta_G^{\mf s}$, the map $HH_0 (\mc F_{\mf d}^t)$ is an 
isomorphism of $C^\infty (X_\unr (M))^{W(M,\mf d)}$-modules from $HH_0 (\mc S (G,K)^{\mf d})$ 
to the set of elements of $\bigoplus_{i=1, i \prec \mf d}^{n_{\mf s}} C^\infty (X_\unr (M_i))$
that descend to linear functions on $\C \otimes_\Z R (\mc S (G,K)^{\mf d})$.
\item Part (a) yields an isomorphism of $Z(\mc S (G,K)^{\mf s})$-modules
\[
HH_0 (\mc S(G,K)^{\mf d}) \cong (\C \otimes_\Z R^t (G)^{\mf d})^*_\infty .
\]
}
\end{prop}
\begin{proof}
(a) From Proposition \ref{prop:5.4}.a and Theorem \ref{thm:5.9}.b we see that the stated condition 
on $f \in \bigoplus_{i=1, i \prec \mf d}^{n_{\mf s}} C^\infty (X_\unr (M_i))$, only with 
$R(\mc H (G,K)^{\mf s})$, is equivalent to the condition
\begin{equation}\label{eq:6.38}
HH_n (\phi_{\mf s}^*) f \in \Big( \bigoplus\nolimits_{w \in W(L,\mf s)} 
\mc O (X_\nr (L)^w) \otimes \natural_{\mf s}^w \Big)^{W(L,\mf s)} .
\end{equation}
The condition in the statement is local, so can be checked locally in terms of \eqref{eq:6.38}.
If one restrict to $R (\mc S (G,K)^{\mf d})$, only the parts of the condition of the 
subvarieties $X_\unr (L)^w$ remain. Then we get exactly the description of 
$HH_0 (\mc S (G,K)^{\mf s})$ already established in part (a).\\
(b) This is analogous to Proposition \ref{prop:5.4}.b.
\end{proof}

To establish an analogue of Theorem \ref{thm:4.2} for $\mc S (G,K)^{\mf s}$, we first
study a smooth version of \cite[Theorem 1.2]{SolTwist}.

\begin{prop}\label{prop:6.13}
Let $W$ be a finite group acting by diffeomorphisms on a smooth real manifold $X$.
Let $\natural : W^2 \to \C^\times$ be a 2-cocycle and let $\{ T_w : w \in W\}$ be the
standard basis of $\C [W,\natural]$. Define $\natural^w (w') = T_w T_{w'} T_w^{-1} T_{w'}^{-1}
\in \C^\times T_{w w' w^{-1} w^{'-1}}$.

There is an isomorphism of Fr\'echet $C^\infty (X)^W$-modules
\[
HH_n \big( C^\infty (X) \rtimes \C [W, \natural] \big) \cong 
\Big( \bigoplus\nolimits_{w \in W} \Omega^n_{sm} (X^w) \otimes \natural^w \Big)^W .
\]
\end{prop}
\begin{proof}
Our argument is a modification of \cite[proof of Theorem 1.2]{SolTwist}. Let 
\[
1 \to Z \to \tilde W \to W \to 1
\]
be a finite central extension of $W$, such that the inflation of $\natural$ to $\tilde W$ 
becomes trivial in $H^2 (\tilde W,\C^\times)$. Then there exist a central idempotent 
$p_\natural \in \C [Z]$ and an algebra isomorphism $p_\natural \C [\tilde W] \to \C [W, \natural]$.
It sends $p_\natural \tilde w$ to $c_{\tilde w} w$, where $w \in W$ denotes the image of 
$\tilde w \in \tilde W$ and $c_{\tilde w} \in \C^\times$ is a suitable scalar. Then
\[
HH_n \big( C^\infty (X) \rtimes \C [W, \natural] \big) \cong 
HH_n \big( p_\natural (C^\infty (X) \rtimes \C [\tilde W]) \big) \cong
p_\natural HH_n \big( C^\infty (X) \rtimes \tilde W \big) . 
\]
By \cite{Bry1,Bry2} there is an isomorphism 
\[
HH_n \big( C^\infty (X) \rtimes \tilde W \big) \cong
\Big( \bigoplus\nolimits_{\tilde w \in \tilde W} \Omega^n_{sm} (X^{\tilde w}) \Big)^{\tilde W}. 
\]
With this at hand, the same analysis as in the analogous algebraic setting 
\cite[(1.5)--(1.14)]{SolTwist} proves the required isomorphism of $C^\infty (X)^W$-modules.

An easier version of Proposition \ref{prop:6.6} shows that 
$HH_n \big( C^\infty (X) \rtimes \C [W, \natural] \big)$ is a Fr\'echet space.
Clearly $\bigoplus\nolimits_{w \in W} \Omega^n_{sm} (X^w) \otimes \natural^w$ is a Fr\'echet space,
so its closed subspace $\Big( \bigoplus\nolimits_{w \in W} \Omega^n_{sm} (X^w) \otimes 
\natural^w \Big)^W$ is Fr\'echet as well. As described in \cite[(1.15)]{SolTwist}, the isomorphism
with $HH_n \big( C^\infty (X) \rtimes \C [W, \natural] \big)$ has two ingredients:
\begin{itemize}
\item the Connes--Hochschild--Kostant--Rosenberg theorem, which is a topological isomorphism
$HH_n (C^\infty (X^w)) \cong \Omega^n_{sm}(X^w)$,
\item some simple constructions in $\C [\tilde W]$.
\end{itemize}
This entails that our isomorphism of $C^\infty (X)^W$-modules is a continuous bijection between
Fr\'echet spaces. Then the open mapping theorem guarantees that it is a homeomorphism. 
\end{proof}

For an algebraic family of $\mc O (X_\nr (L)) \rtimes \C [W(L,\mf s), 
\natural_{\mf s}]$-representations $\mf F (M,\eta)$, parametrized by $X_\nr (M)$ and on a
vector space $V_{M,\eta}$, we define
\[
\begin{array}{cccc}
\mc F^t_{M,\eta} : & C^\infty (X_\unr (L)) \rtimes \C [W(L,\mf s), \natural_{\mf s}] & \to &
C^\infty (X_\unr (M)) \otimes \End_\C ( V_{M,\eta} ) \\
& f & \mapsto & [\chi \mapsto \mf F (M,\eta,\chi) (f) ]  
\end{array} .
\]
Recall from Theorem \ref{thm:4.1} that $\zeta^\vee$ restricts to a bijection
\[
\zeta^\vee_t :  R^t (G)^{\mf s} \cong R (\mc S (G,K)^{\mf s}) \longrightarrow
R \big( C^\infty (X_\unr (L)) \rtimes \C [W(L,\mf s),\natural_{\mf s}] \big) .
\]

\begin{thm}\label{thm:6.11}
There exists a unique isomorphism of Fr\'echet spaces
\[
HH_n (\zeta^\vee_t) : HH_n \big( C^\infty (X_\unr (L)) \rtimes \C [W(L,\mf s),
\natural_{\mf s}] \big) \to HH_n (\mc S (G,K)^{\mf s})
\]
such that
\[
HH_n (\mc F^t_{M,\eta}) \circ HH_n (\zeta^\vee_t) = HH_n (\mc F^t_{M,\zeta^\vee (\eta)})
\]
for all algebraic families $\mf F (M,\eta)$ in $\Mod (\mc H (G,K)^{\mf s})$.
\end{thm}
\begin{proof}
This is analogous to the proof of Theorem \ref{thm:4.2}. For the construction of 
$HH_n (\zeta_t^\vee)$ we use Proposition \ref{prop:6.13} (with $X_\nr (L), W(L,\mf s), 
\natural_{\mf s}$) and Theorem \ref{thm:6.4} instead of \cite[Theorem 1.2]{SolTwist}, Theorem 
\ref{thm:5.9}.b and Lemma \ref{lem:5.6}. In all the involved algebraic families of representations 
$\mf F (M,\eta)$, temperedness of $\mf F (M,\eta,\chi)$ is equivalent to $\chi \in X_\unr (L)$.
The uniqueness and further properties of the thus defined map $HH_n (\zeta_t^\vee)$ can be shown 
in exactly the same way as for $HH_n (\zeta^\vee)$. 
\end{proof}

Theorems \ref{thm:4.2} and \ref{thm:6.11} relate the Hochschild homology of 
$\mc H (G)^{\mf s}$ and $\mc S (G)^{\mf s}$ to that of the twisted crossed products 
\[
\mc O (X_\nr (L)) \rtimes \C [W(L,\mf s), \natural_{\mf s}] \quad \text{and} \quad 
C^\infty (X_\unr (L)) \rtimes \C [W(L,\mf s), \natural_{\mf s}] .
\]
These theorems can be considered as confirmations of the ABPS conjectures \cite{ABPS2}
on the level of Hochschild homology.

Finally, we take a closer look at $HH_n (\mc S (G,K)^{\mf d})$.
From the Plancherel isomorphism \eqref{eq:6.4} we get 
\begin{equation}\label{eq:6.18}
HH_n (\mc S (G,K)^{\mf s}) = 
\bigoplus\nolimits_{\mf d \in \Delta_G^{\mf s}} HH_n (\mc S (G,K)^{\mf d}).
\end{equation}
By Lemma \ref{lem:6.7}.b we can regard $HH_n (\mc H (G,K)^{\mf s})$ as a 
subspace of $HH_n (\mc S (G,K)^{\mf s})$.

\begin{thm}\label{thm:6.1}
\enuma{
\item The maps $HH_n (\mc F_{\mf d}^t)$ and $HH_n (\phi_{\mf s}^*)$ provide isomorphisms
between the Fr\'echet $Z(\mc S (G,K)^{\mf d})$-modules $HH_n (\mc S (G,K)^{\mf d})$,
\begin{align*}
& \bigoplus_{i=1, i \prec \mf d}^{n_{\mf s}} \Omega^n_{sm} (X_\unr (M_i)) 
\bigcap HH_n (\tilde \phi_{\mf s}^* )^{-1} \Big( \bigoplus_{w \in W(L,\mf s)} 
\Omega^n_{sm} (X_\unr (L)^w) \otimes \natural_{\mf s}^w \Big)^{W(L,\mf s)} \text {and} \\
& HH_n (\tilde \phi_{\mf s}^*) \Big( \bigoplus_{i=1, i \prec \mf d}^{n_{\mf s}} \Omega^n_{sm} 
(X_\unr (M_i)) \Big) \bigcap \Big( \bigoplus_{w \in W(L,\mf s)} 
\Omega^n_{sm} (X_\unr (L)^w) \otimes \natural_{\mf s}^w \Big)^{W(L,\mf s)} 
\end{align*}
\item $HH_n (\mc S (G,K)^{\mf d})$ is the closure of $HH_n (\mc H (G,K)^{\mf s})^{\mf d}$
in $HH_n (\mc S (G,K)^{\mf s})$. 
\item There is a natural isomorphism of Fr\'echet $Z(\mc S (G,K)^{\mf d})$-modules
\[
Z(\mc S (G,K)^{\mf d}) \underset{Z (\mc H (G,K)^{\mf s})}{\otimes} 
HH_n (\mc H (G,K)^{\mf s})^{\mf d} \longrightarrow HH_n (\mc S (G,K)^{\mf d}) .
\]
}
\end{thm}
\begin{proof}
(a) The $Z(\mc S (G,K)^{\mf d})$-linearity comes from Lemma \ref{lem:6.8} and the isomorphisms 
follow immediately from Theorem \ref{thm:6.4}. The range of $HH_n (\mc F_{\mf d}^t)$ is a 
Fr\'echet space because by the continuity of $HH_n (\tilde \phi_{\mf s}^*)$ it is closed in 
$\bigoplus_{i=1, i \prec \mf d}^{n_{\mf s}} \Omega^n_{sm} (X_\unr (M_i))$.
The range of $HH_n (\tilde \phi_{\mf s}^* )  HH_n (\mc F_{\mf d}^t)$ is Fr\'echet because
as checked directly after \eqref{eq:6.42} it is a direct summand of $\bigoplus_{w \in W(L,\mf s)} 
\Omega^n_{sm} (X_\unr (L)^w)$.\\
(b) Recall from Lemma \ref{lem:5.5} that 
\begin{equation}\label{eq:6.37}
HH_n (\mc H (G,K)^{\mf s})^{\mf d} = 
HH_n (\mc F_{\mf s})^{-1} HH_n (\mc F_{\mf d}) HH_n (\mc H (G,K)^{\mf s}) .
\end{equation}
From part (a) and Theorem \ref{thm:5.9}.b we see that the closure of 
\[
HH_n (\mc F_{\mf d}) HH_n (\mc H (G,K)^{\mf s}) \quad \text{in} \quad 
\bigoplus\nolimits_{i=1, i \prec \mf d}^{n_{\mf s}} \Omega^n_{sm} (X_\unr (M_i)) 
\]
equals $HH_n (\mc F_{\mf d}^t) HH_n (\mc S (G,K)^{\mf d})$. To this we apply 
$HH_n (\mc F_{\mf s}^t)^{-1}$, which exists and is continuous by part (a). 
We find that the closure of \eqref{eq:6.37} equals
\begin{multline*}
HH_n (\mc F_{\mf s}^t)^{-1} HH_n (\mc F_{\mf d}^t) HH_n (\mc S (G,K)^{\mf d}) = \\
HH_n (\mc F_{\mf s}^t)^{-1} HH_n (\mc F_{\mf s}^t) HH_n (\mc S (G,K)^{\mf d}) =
HH_n (\mc S (G,K)^{\mf d}) . 
\end{multline*}
(c) The map is induced by the algebra homomorphism
\[
\mc H (G,K)^{\mf s} \to \mc S (G,K)^{\mf d} : h \mapsto e_{\mf d} h 
\]
and the $Z(\mc S (G,K)^{\mf d})$-module structure of $HH_n (\mc S (G,K)^{\mf d})$, so it is
natural. Part (b) implies that the $Z(\mc H (G,K)^{\mf s})$-action on 
$HH_n (\mc H (G,K)^{\mf s})^{\mf d}$ factors through
\[
Z (\mc H (G,K)^{\mf s}) \to Z (\mc S (G,K)^{\mf d}) : z \mapsto e_{\mf d} z .
\]
Hence we may just as well consider it as an action of 
\[
e_{\mf d} Z(\mc H (G,K)^{\mf s}) \cong \mc O (X_\nr (M))^{W(M,\mf d)} .
\]
After \eqref{eq:6.42} we constructed a continuous $C^\infty (X_\unr (M))^{W(M,\mf d)}$-linear
idempotent endomorphism $p$ of $\bigoplus_{w \in W(L,\mf s)} \Omega^n_{sm} (X_\unr (L)^w)$,
with image \eqref{eq:5.21}. By part (a) $HH_n (\mc S (G,K)^{\mf d})$ is isomorphic as
$C^\infty (X_\unr (M))^{W(M,\mf d)}$-module to the image of $p$, via the map
$HH_n (\phi_{\mf s}^*) HH_n (\mc F_{\mf d}^t)$. Similarly Theorem \ref{thm:5.9} tells us that
$HH_n (\mc H (G,K)^{\mf s})^{\mf d}$ is isomorphic as $\mc O (X_\nr (M))^{W(M,\mf d)}$-module to 
\[
p \big( \bigoplus\nolimits_{w \in W(L,\mf s)} \Omega^n (X_\nr (L)^w) \big) ,
\]
via $HH_n (\phi_{\mf s}^*) HH_n (\mc F_{\mf d})$. Thus we have translated the statement to:
the natural map
\begin{multline*}
C^\infty (X_\unr (M))^{W(M,\mf d)} \underset{\mc O (X_\nr (M))^{W(M,\mf d)}}{\otimes}
p \big( \bigoplus\nolimits_{w \in W(L,\mf s)} \Omega^n (X_\nr (L)^w) \big) \\
\longrightarrow  p \big( \bigoplus\nolimits_{w \in W(L,\mf s)} \Omega^n_{sm} (X_\unr (L)^w) \big)
\end{multline*}
is an isomorphism of Fr\'echet $C^\infty (X_\unr (M))^{W(M,\mf d)}$-modules. As the action comes
from an embedding $X_\unr (L)^w \to X_\unr (M)$ for each relevant $w$, that claim is an instance 
of Proposition \ref{prop:6.12}.
\end{proof}

Let us record a consequence of Theorem \ref{thm:6.1}:
\begin{equation}\label{eq:6.44}
\bigoplus\nolimits_{\mf d \in \Delta_G^{\mf s}} Z(\mc S (G,K)^{\mf d}) 
\underset{Z (\mc H (G,K)^{\mf s})}{\otimes} HH_n (\mc H (G,K)^{\mf s})^{\mf d} 
\; \cong \; HH_n (\mc S (G,K)^{\mf s}) 
\end{equation}
as Fr\'echet $Z(\mc S (G,K)^{\mf s})$-modules. However, usually
\[
Z(\mc S(G,K)^{\mf s}) \underset{Z(\mc H (G,K)^{\mf s})}{\otimes} HH_n (\mc H (G,K)^{\mf s})
\]
is not isomorphic to $HH_n (\mc S (G,K)^{\mf s})$ as $Z(\mc S (G,K)^{\mf s})$-module. The
reason is that the terms 
\[
Z(\mc S(G,K)^{\mf d'}) \underset{Z(\mc H (G,K)^{\mf s})}{\otimes} 
HH_n (\mc H (G,K)^{\mf s})^{\mf d}
\]
with $\mf d' \neq \mf d$ can be nonzero, but do not occur in \eqref{eq:6.44}.

\section{Cyclic homology}
\label{sec:cyclic}

Recall from \cite[\S 2.1.7]{Lod} that the cyclic homology of a unital algebra $A$ can be 
computed as the total homology of a bicomplex $(\mc B (A), b, B)$. Here
\[
\mc B (A)_{p,q} = A^{\otimes p+1-q} \quad \text{if} \quad  p \geq q \geq 0,
\] 
and $\mc B (A)_{p,q}$ is zero otherwise.
The vertical differential $b$ is the same as in the bar-resolution, so each column of
$\mc B (A)$ computes the Hochschild homology of $A$. The horizontal differential $B$
induces a map $B : HH_n (A) \to HH_{n+1}(A)$. When $A = \mc O (V)$ for a nonsingular complex 
affine variety or $A = C^\infty (V)$ for a smooth manifold $V$, $B$ is the usual exterior
differential $d : \Omega^n (V) \to \Omega^{n+1}(V)$ \cite[\S 2.3.6]{Lod}.

For $A = \mc H (G,K)^{\mf s}$, we know from Theorem \ref{thm:5.9} that there is an isomorphism
\begin{equation}\label{eq:1.1}
HH_n (\mc H (G,K)^{\mf s}) \to \Big( \bigoplus\nolimits_{w \in W(L,\mf s)} \Omega^n (X_\nr (L)^w )
\otimes \natural_{\mf s}^w \Big)^{W(L,\mf s)} ,
\end{equation}
induced by the algebraic families of virtual representations
\[
\big\{ \nu^1_{w,\chi} : \chi \in X_\nr (L)^w_c \big\}  
\qquad w \in W(L,\mf s), c \in \pi_0 (X_\nr (L)^w) .
\]
By \eqref{eq:5.31} each of these families is a linear combination of algebraic families
$\mf F' (M_i,\eta_i)$ obtained from $\mf F (M_i,\eta_i)$ by composition with an algebraic map
from $X_\nr (L)^w_c$ to $X_\nr (M_i)$. In particular \eqref{eq:1.1} is a linear combination of maps
\begin{equation}\label{eq:1.2}
HH_n (\mc F'_{M_i,\eta_i}) : HH_n (\mc H (G,K)^{\mf s}) \to
HH_n \big( \mc O (X_\nr (L)^w_c) \otimes \End_\C \big( I_{P_i}^G (\eta_i)^K \big) \big) .
\end{equation}
By Morita invariance and the Hochschild--Kostant--Rosenberg theorem, the right hand side of
\eqref{eq:1.2} can be identified with
\begin{equation}\label{eq:1.4}
HH_n \big( \mc O (X_\nr (L)^w_c) \big) \cong \Omega^n (X_\nr (L)^w_c) .
\end{equation}
Via these maps, the natural differential $B$ on $HH_* (\mc H (G,K)^{\mf s})$ is transformed into
the exterior differential $d$ on $\Omega^* (X_\nr (L)^w_c)$. All the maps in \eqref{eq:1.2} and
\eqref{eq:1.4} (and between them) can be realized on the level of chain complexes. For 
$HH_n (\mc F'_{M_i,\eta_i})$ that is clear, the Morita equivalence between
\[
\mc O (X_\nr (L)^w_c) \otimes \End_\C \big( I_{P_i}^G (\eta_i)^K \big) 
\quad \text{and} \quad \mc O (X_\nr (L)^w_c)
\]
is implemented by the generalized trace map \cite[\S 1.2]{Lod} and \eqref{eq:1.4} comes from the
map $\pi_n$ in \cite[Lemma 1.3.14]{Lod}. Altogether these furnish a morphism of bicomplexes
\[
\big( \mc B (\mc H (G,K)^{\mf s}), b, B \big) \longrightarrow \bigoplus_{w \in W(L,\mf s), 
c \in \pi_0 (X_\nr (L)^w)} \big( \mc B \Omega^* (X_\nr (L)^w_c), 0,d \big) ,
\]
where $\mc B \Omega^* (V)$ is the bicomplex with $\Omega^{p-q}(V)$ in degree $(p,q)$, provided
that $p \geq q \geq 0$. From Theorem \ref{thm:5.9} we know that its image is actually smaller, 
we can restrict it to a morphism of bicomplexes
\begin{equation}\label{eq:1.5}
\big( \mc B (\mc H (G,K)^{\mf s}), b, B \big) \longrightarrow \Big( 
\big( \bigoplus\nolimits_{w \in W(L,\mf s)} \mc B \Omega^* (X_\nr (L)^w) \otimes \natural_{\mf s}^w 
\big)^{W(L,\mf s)}, 0,d \Big) ,
\end{equation}
Analogous considerations for $\mc S (G,K)^{\mf s}$, now using Theorem \ref{thm:6.4}, lead to a
morphism of bicomplexes
\begin{equation}\label{eq:1.6}
\big( \mc B (\mc S (G,K)^{\mf s}), b, B \big) \longrightarrow \Big( \big(
\bigoplus\nolimits_{w \in W(L,\mf s)} \mc B \Omega^*_{sm} (X_\unr (L)^w) \otimes \natural_{\mf s}^w 
\big)^{W(L,\mf s)}, 0,d \Big) .
\end{equation}
By Theorems \ref{thm:5.9} and \ref{thm:6.4}, the maps \eqref{eq:1.5} and \eqref{eq:1.6} induce
isomorphisms on the Hochschild homology of the involved bicomplexes. It follows from Connes'
periodicity exact sequence that \eqref{eq:1.5} and \eqref{eq:1.6} also induce isomorphisms
on cyclic homology, see \cite[\S 2.5]{Lod}.

\begin{thm}\label{thm:1.1}
There are isomorphisms of vector spaces
\begin{align*}
HC_n (\mc H (G,K)^{\mf s}) \; \cong \; & \Big( \bigoplus\nolimits_{w \in W(L,\mf s)} 
\Omega^n (X_\nr (L)^w) / d \Omega^{n-1} (X_\nr (L)^w) \otimes \natural_{\mf s}^w \Big)^{W(L,\mf s)} 
\oplus \\
& \bigoplus_{m=1}^{\lfloor n/2 \rfloor} \Big( \bigoplus_{w \in W(L,\mf s)} H_{dR}^{n - 2m} 
(X_\nr (L)^w) \otimes \natural_{\mf s}^w \Big)^{W(L,\mf s)}  ,\\
HC_n (\mc S (G,K)^{\mf s}) \; \cong \; & \Big( \bigoplus\nolimits_{w \in W(L,\mf s)} 
\Omega^n_{sm} (X_\unr (L)^w) / d \Omega^{n-1} (X_\unr (L)^w) \otimes \natural_{\mf s}^w 
\Big)^{W(L,\mf s)} \oplus \\
& \bigoplus_{m=1}^{\lfloor n/2 \rfloor} \Big( \bigoplus_{w \in W(L,\mf s)} H_{dR}^{n - 2m} 
(X_\unr (L)^w) \otimes \natural_{\mf s}^w \Big)^{W(L,\mf s)} .
\end{align*}
\end{thm}
\begin{proof}
As explained above, it remains to identify the cyclic homology of the right hand sides of 
\eqref{eq:1.5} and \eqref{eq:1.6}. By design
\begin{equation}\label{eq:1.7} 
HC_n \big( \mc B \Omega^* (V),0,d \big) = \Omega^n (V) / d \Omega^{n-1}(V) \oplus 
\bigoplus\nolimits_{m=1}^{\lfloor n/2 \rfloor} H_{dR}^{n - 2m} (V) ,
\end{equation}
see \cite[\S 2.3]{Lod}. In our case $V = \bigsqcup_{w \in W(L,\mf s)} X_\nr (L)^w$ and the group 
$W(L,\mf s)$ acts on $\Omega^* (V)$, namely by the natural action on the underlying
space tensored with the characters $\natural_{\mf s}^w$. Taking invariants for an action of
a finite group commutes with homology, so we may just take the $W(L,\mf s)$-invariants in
\eqref{eq:1.7}. That yields $HC_n (\mc H (G,K)^{\mf s})$, and the argument for 
$HC_n (\mc S (G,K)^{\mf s})$ is completely analogous.
\end{proof}

From Theorem \ref{thm:1.1} we see that $HC_n (\mc H (G,K)^{\mf s})$ and $HC_n (\mc S (G,K)^{\mf s})$
stabilize: for $n > \dim_\C (X_\nr (L))$ they depend only on the parity of $n$. By 
\cite[Proposition 5.1.9]{Lod}, the periodic cyclic homology is the limit term:
\begin{align}
\label{eq:1.8} & HP_n (\mc H (G,K)^{\mf s}) \cong \bigoplus\nolimits_{m \in \Z} 
\Big( \bigoplus\nolimits_{w \in W(L,\mf s)} H_{dR}^{n + 2m} (X_\nr (L)^w) \otimes
 \natural_{\mf s}^w \Big)^{W(L,\mf s)} , \\
\label{eq:1.9} & HP_n (\mc S (G,K)^{\mf s}) \cong \bigoplus\nolimits_{m \in \Z} 
\Big( \bigoplus\nolimits_{w \in W(L,\mf s)} H_{dR}^{n + 2m} (X_\unr (L)^w) \otimes 
\natural_{\mf s}^w \Big)^{W(L,\mf s)} .
\end{align}
We point out that the right hand sides of, respectively, \eqref{eq:1.8} and \eqref{eq:1.9} 
are naturally isomorphic with the periodic cyclic homology groups of, respectively, 
\begin{equation}\label{eq:1.10}
\mc O (X_\nr (L)) \rtimes \C [W(L,\mf s),\natural_{\mf s}] \quad \text{and} \quad
C^\infty (X_\unr (L)) \rtimes \C [W(L,\mf s),\natural_{\mf s}] .
\end{equation}
That can be derived with similar arguments. Hence \eqref{eq:1.8} and \eqref{eq:1.9} are the
versions of Theorems \ref{thm:4.2} and \ref{thm:6.11} for periodic cyclic homology.

We note also that \eqref{eq:1.9} relates to the conjectural description of the topological 
K-theory of $\mc S (G,K)^{\mf s}$ in \cite[Conjecture 5]{ABPS2}. In \cite[\S 4]{ABPS2} things 
are formulated for the $C^*$-completion of $\mc S (G)^{\mf s}$, which has the same topological 
K-theory as $\mc S (G,K)^{\mf s}$ by \cite[(3.2)]{SolHPadic}. Since the Chern character
\[
K_* (\mc S (G,K)^{\mf s}) \otimes_\Z \C \to HP_* (\mc S (G,K)^{\mf s})
\]
is an isomorphism \cite[Theorem 3.2]{SolHPadic}, \eqref{eq:1.9} provides a description
of $K_* (\mc S (G,K)^{\mf s})$ modulo torsion. With the comments around \eqref{eq:1.10} 
we can formulate that as an isomorphism
\begin{equation}\label{eq:1.11}
K_* (\mc S (G,K)^{\mf s}) \otimes_\Z \C \cong 
HP_* \big( C^\infty (X_\unr (L)) \rtimes \C [W(L,\mf s),\natural_{\mf s}] \big) .
\end{equation}
Via the equivariant Chern character \cite{BaCo} for an action of a central extension of 
$W(L,\mf s)$ on $X_\unr (L)$, the right hand side of \eqref{eq:1.11} is isomorphic with 
\[
K_* \big( C^\infty (X_\unr (L)) \rtimes \C [W(L,\mf s),\natural_{\mf s}] \big) \otimes_\Z \C  
= K^*_{W(L,\mf s), \natural_{\mf s}} (X_\unr (L)) \otimes_\Z \C ,
\]
where the latter is the notation from \cite[\S 4.1]{ABPS2}. We have proved 
\cite[Conjecture 5]{ABPS2} modulo torsion:

\begin{thm}\label{thm:1.5}
There is an isomorphism of vector spaces
\[
K_* (\mc S (G,K)^{\mf s}) \otimes_\Z \C \cong 
K^*_{W(L,\mf s), \natural_{\mf s}} (X_\unr (L)) \otimes_\Z \C .
\]
\end{thm}

Since $X_\unr (L)$ is a $W(L,\mf s)$-equivariant deformation retract of $X_\nr (L)$,
\[
H_{dR}^n (X_\nr (L)^w) = H_{dR}^n (X_\unr (L)^w) . 
\]
Combining that with \eqref{eq:1.8} and \eqref{eq:1.9}, we recover 
\cite[Theorem 3.3]{SolHPadic}:

\begin{cor}\label{cor:1.3}
The inclusion $\mc H (G,K)^{\mf s} \to \mc S (G,K)^{\mf s}$ induces an isomorphism on
periodic cyclic homology.
\end{cor}

With elementary Lie theory one sees that $X_\nr (L)^w$ is a finite union of cosets 
$X_\nr (L)^w_c$ of the complex torus $X_\nr (L)^{w,\circ}$. Since $X_\nr (L)^w$ is 
a commutative Lie group. its tangent spaces at any two points are canonically isomorphic.

\begin{lem}\label{lem:1.4}
$HP_n (\mc H (G,K)^{\mf s})$ can be represented by the elements of 
\[
\bigoplus\nolimits_{m \in \Z} HH_{n + 2m}(\mc H (G,K)^{\mf s})
\]
that are locally constant (as differential forms). The same holds for $\mc S (G,K)^{\mf s}$.
\end{lem}
\begin{proof}
First we consider a simpler setting, namely the graded algebra of differential forms 
$\Omega^* (T)$ on a complex algebraic torus $T$. Write $T$ as a direct product of onedimensional 
algebraic subtori $T_i$, then
\begin{equation}\label{eq:1.12}
\Omega^*  (T) = \bigotimes\nolimits_i \Omega^* (T_i) .
\end{equation}
For $T_i$ everything is explicit:
\[
H_{dR}^0 (T_i) = \C ,\quad H_{dR}^1 (T_i) = \C \textup{d}z ,
\]
and this equals the subspace of constant elements in
\[
\Omega^* (T_i) = \C[z,z^{-1}] \oplus \C [z,z^{-1}] \textup{d}z .
\]
By the K\"unneth formula
\begin{equation}\label{eq:1.13}
H_{dR}^* (T) = \bigotimes\nolimits_i H_{dR}^* (T_i) ,
\end{equation}
and in combination with \eqref{eq:1.12} we find that this is precisely the space of constant
differential forms in $\Omega^* (T)$. 

The above argument uses the structure of $T$ as algebraic variety, not as group, so it applies
to all the varieties $X_\nr (L)^w_c$. Further, the action of $Z_{W(L,\mf s)} (w,c)$ on 
$\Omega^* (X_\nr (L)^w_c)$ preserve the subspace of constant differential forms. Hence 
\[
\Big( \bigoplus\nolimits_{w \in W(L,\mf s)} 
H_{dR}^n (X_\nr (L)^w) \otimes \natural_{\mf s}^w \Big)^{W(L,\mf s)}
\]
can be represented by the elements of 
\[
\Big( \bigoplus\nolimits_{w \in W(L,\mf s)} 
\Omega^n (X_\nr (L)^w) \otimes \natural_{\mf s}^w \Big)^{W(L,\mf s)}
\]
that are locally constant (keeping the canonical identifications of different tangent spaces in mind).
Combining that with Theorem \ref{thm:5.9} and \eqref{eq:1.8}, we get the lemma for $\mc H (G,K)^{\mf s}$.
The above arguments involving $T$ also work for smooth differential forms on compact real tori.
With that, Theorem \ref{thm:6.4} and \eqref{eq:1.9}, we establish the lemma for $\mc S (G,K)^{\mf s}$. 
\end{proof}

Assume now that the 2-cocycle $\natural_{\mf s}$ is trivial, like in most examples.
Then \eqref{eq:1.8}--\eqref{eq:1.9} and Corollary \ref{cor:1.3} simplify to 
\begin{equation}\label{eq:2.7}
HP_n (\mc H (G,K)^{\mf s}) \cong HP_n (\mc S (G,K)^{\mf s}) \cong \bigoplus_{m \in \Z} 
\Big( \bigoplus_{w \in W(L,\mf s) \hspace{-3mm}} \hspace{-2mm} 
H_{dR}^{n + 2m} (X_\nr (L)^w) \Big)^{W(L,\mf s)} .
\end{equation}
For any nonsingular complex affine variety $V$, there is a natural isomorphism 
\[
H_{dR}^n (V) \cong H_{\mr{sing}}^n (V). 
\]
Here $H^*_{\mr{sing}}$ denotes singular cohomology with complex coefficients, and it is applied to $V$ 
with the analytic topology. If a finite group $\Gamma$ acts algebraically on $V$
and $V'$ is a $\Gamma$-equivariant deformation retract of $V$, then naturally
\[
H^n_{\mr{sing}} (V)^\Gamma \cong H^n_{\mr{sing}} (V / \Gamma) \cong H^n_{\mr{sing}} (V' / \Gamma) .
\]
Hence \eqref{eq:2.7} can be expressed as
\begin{equation}\label{eq:2.6}
\begin{aligned}
HP_n (\mc H (G,K)^{\mf s}) & \cong HP_n (\mc S (G,K)^{\mf s}) \\ 
& \cong \bigoplus\nolimits_{m \in \Z} 
H^{n+2m}_{\mr{sing}} \Big( \bigsqcup\nolimits_{w \in W(L,\mf s)} X_\unr (L)^w / W(L,\mf s) \Big) .
\end{aligned}
\end{equation}

\section{Examples}
\label{sec:ex}

The smallest nontrivial example is $G = SL_2 (F)$, where $F$ is any non-archimedean local field. 
The Hochschild homology for this group is known entirely from \cite{SolHomAHA}, here we work 
it out in our notations.

For every supercuspidal $G$-representation $V$ with $V^K \neq 0$, the corresponding direct
summand of $\mc H (G,K)$ is Morita equivalent with $\C$. This contributes a factor $\C$ 
to $HH_0 (\mc H (G,K))$, and nothing to $HH_n (\mc H (G,K))$ with $n > 0$. The same applies 
to $HH_* (\mc S (G,K))$.

Let $T \cong F^\times$ be the diagonal subgroup of $G$ and write $W = W(G,T) = \{w,e\}$.
Consider a $K \cap T$-invariant
character $\chi_0$ of the maximal compact subgroup $\mf o_F^\times$ of $T$. Let 
$\chi_1 : T \to \C^\times$ be an extension of $\chi_0$ and let Rep$(G)^{\mf s}$ be the Bernstein 
block associated to $(T,\chi_1)$. There are three different cases:
\begin{itemize}
\item{ord$(\chi_0 ) > 2$.} Here $W(T,\mf s)$ is trivial and $\mc H (G,K)^{\mf s}$ is 
Morita equivalent with $\mc O (X_\nr (T))$. Hence 
\[
HH_n (\mc H (G,K)^{\mf s}) \cong \Omega^n (X_\nr (T)) \cong \mc O (X_\nr (T)) \cong \C [z,z^{-1}]
\qquad n = 0, 1, 
\]
and $HH_n (\mc H (G,K)^{\mf s}) = 0$ for other $n$.
Similarly $\mc S (G,K)^{\mf s}$ is Morita equivalent with $C^\infty (X_\unr (T))$. Hence
\[
HH_n (\mc S (G,K)^{\mf s}) \cong \Omega^n_{sm} (X_\unr (T)) 
\cong C^\infty (X_\unr (T)) \cong C^\infty (S^1) \qquad n = 0, 1 ,
\]
and 0 in other degrees $n$. 
\item{ord$(\chi_0 ) = 2$.} Now $W(T,\mf s)$ equals $W$. If we pick a $\chi_1$ of
order 2 and regard it as basepoint of 
\[
\Irr (T)^{\mf s} = \chi_1 X_\nr (T) \cong \C^\times, 
\]
then $W$ acts on $\C^\times$ by inversion. The algebra $\mc H (G,K)^{\mf s}$ is Morita equivalent with 
$\mc O (X_\nr (T)) \rtimes W$. The representation $I_B^G (\chi_1 \chi)$ is irreducible unless 
$\chi_1 \chi$ has order 2, then it is a direct sum of two inequivalent irreducible representations.
In this case we need 3 algebraic families of $G$-representations in Rep$(G)^{\mf s}$, namely
\[
\mf F (T,\chi_1) = \{ I_B^G (\chi_1 \chi) : \chi \in X_\nr (T) \} ,
\]
one irreducible summand $\pi_1$ of $I_B^G (\chi_1)$ and one irreducible summand $\pi_2$ of 
$I_B^G (\chi_2)$, where $\chi_2$ is the other order 2 extension of $\chi_0$. The families of
virtual representations differ slightly, namely
\[
\begin{array}{lll}
\nu^1_{e,\chi} & = & I_B^G (\chi_1 \chi) \hspace{2cm} \chi \in X_\nr (T) , \\
\nu^1_{w,\chi_1} & = & \pi_1 - I_B^G (\chi_1) / 2 ,\\
\nu^1_{w,\chi_2} & = & \pi_2 - I_B^G (\chi_2) / 2 .
\end{array}
\]
(This works, but for the best normalization we should make sure that $\pi_1$ and $\pi_2$
are chosen so that $W$ acts trivially on their $K$-invariant vectors.) We find
\[
\begin{array}{lll}
HH_0 (\mc H (G,K)^{\mf s}) & \cong & \mc O (X_\nr (T))^W \oplus \C \oplus \C ,\\
HH_1 (\mc H (G,K)^{\mf s}) & \cong & \Omega^1 (X_\nr (T))^W \cong \mc O (X_\nr (T))^W , \\
HH_n (\mc H (G,K)^{\mf s}) & = & 0 \qquad \text{for } n > 1.
\end{array}
\] 
The centre $\mc O (X_\nr (T))^W$ acts on one factor $\C$ via $\chi_1$ and on the other via $\chi_2$.
Further $\mc S (G,K)^{\mf s}$ is Morita equivalent with $C^\infty (X_\unr (T)) \rtimes W$,
and one obtains
\[
\begin{array}{lll}
HH_0 (\mc S (G,K)^{\mf s}) & \cong & C^\infty (X_\unr (T))^W \oplus \C \oplus \C ,\\
HH_1 (\mc S (G,K)^{\mf s}) & \cong & \Omega^1_{sm} (X_\unr (T))^W \cong C^\infty (X_\unr (T))^W , \\
HH_n (\mc S (G,K)^{\mf s}) & = & 0 \qquad \text{for } n > 1.
\end{array}
\]
\item{ord$(\chi_0 ) = 1$.}
Now $\chi_1 = 1$ and $\mc H (G,K)^{\mf s}$ is Morita equivalent to an affine Hecke algebra of
type $A_1$ with equal parameters $q_F$. The representation $I_B^G (\chi)$ with $\chi \in X_\nr (T)$
is reducible if and only if the value of $\chi$ at a uniformizer $\varpi_F$ of $G$ lies in $\{-1, 
q_F, q_F^{-1} \}$. Like in the previous case we need three algebraic families of representations:
\[
\mf F (T,\chi_1) = \{ I_B^G (\chi) : \chi \in X_\nr (T) \} ,
\]
the Steinberg representation St and one irreducible summand $\pi_-$ of $I_B^G (\chi_-)$, where
$\chi_-$ is the unique unramified character of order 2. The algebraic fa\-mi\-lies of virtual
representations are:
\[
\begin{array}{lll}
\nu^1_{e,\chi} & = & I_B^G (\chi) \hspace{2cm} \chi \in X_\nr (T) , \\
\nu^1_{w,1} & = & I_B^G (\mr{triv}_T) / 2 - \mr{St} ,\\
\nu^1_{w,\chi_-} & = & \pi_- - I_B^G (\chi_-) / 2  .
\end{array}
\]
(For the correct normalization, we should pick $\pi_-$ such that via Theorem \ref{thm:4.1}
$W$ acts trivially on $\zeta^\vee (\pi_-)$.) The maps \eqref{eq:3.8} provide
isomorphisms of $Z(\mc H (G,K)^{\mf s})$-modules
\[
\begin{array}{lll}
HH_0 (\mc H (G,K)^{\mf s}) & \cong & \mc O (X_\nr (T))^W \oplus \C \oplus \C ,\\
HH_1 (\mc H (G,K)^{\mf s}) & \cong & \Omega^1 (X_\nr (T))^W \cong \mc O (X_\nr (T))^W , \\
HH_n (\mc H (G,K)^{\mf s}) & = & 0 \qquad \text{for } n > 1.
\end{array}
\] 
Here $\mc O (X_\nr (T))^W$ acts on one factor $\C$ via $\chi_-$, and on the other via unramified 
characters with values $\{q_F,q_F^{-1}\}$ at $\varpi_F$. 

As before, these findings extend naturally to $\mc S (G,K)^{\mf s}$:
\[
\begin{array}{lll}
HH_0 (\mc S (G,K)^{\mf s}) & \cong & C^\infty (X_\unr (T))^W \oplus \C \oplus \C ,\\
HH_1 (\mc S (G,K)^{\mf s}) & \cong & \Omega^1_{sm} (X_\unr (T))^W \cong C^\infty (X_\unr (T))^W , \\
HH_n (\mc S (G,K)^{\mf s}) & = & 0 \qquad \text{for } n > 1.
\end{array}
\]
\end{itemize}

Another well-studied example is the general linear group $G = GL_n (F)$. For this $G$ most
aspects are simpler than for other reductive $p$-adic groups. 
Consider an arbitrary inertial equivalence class $\mf s = [L,\sigma]$ for $G$. By picking suitable 
representatives, we can achieve that
\[
L = \prod\nolimits_{i=1}^\ell GL_{n_i}(F)^{e_i}, \quad \sigma = \boxtimes_{i=1}^\ell 
\sigma_i^{\boxtimes e_i} ,
\]
where $\sigma_i$ and $\sigma_{i'}$ with $i \neq i'$ are not equivalent up to unramified twists.
There are natural isomorphisms
\[
X_\nr (L) \cong \prod\nolimits_{i=1}^\ell (\C^\times)^{e_i} ,\quad
W(L,\mf s) \cong \prod\nolimits_{i=1}^\ell S_{e_i} .
\]
Let $M$ be a Levi subgroup of $G$ containing $L$ and let $\delta \in \Irr (M)$ be square-integrable 
modulo centre, such that $\delta \in \Rep (M)^{\mf s}$. Then $\mf d = [M,\delta]$ can be represented
by data
\[
M = \prod\nolimits_{i=1}^\ell \prod\nolimits_{j=1}^{\ell_i} GL_{n_i d_j} (F)^{e_{i,j}}, \quad 
\delta = \boxtimes_{i=1}^\ell \boxtimes_{j=1}^{\ell_i} \mr{St}(d_j,\sigma_i)^{\boxtimes e_{i,j}},
\] 
where $\sum_{j=1}^{\ell_i} d_j e_{i,j} = e_i$ and St$(d_j,\sigma_i)$ is the generalized Steinberg
representation associated to $\sigma_i^{\boxtimes d_j}$. Moreover we may assume that
St$(d_j,\sigma_i)$ and St$(d_{j'},\sigma_i)$ do not differ by an unramified twist if $j \neq j'$.
In this case there are natural isomorphisms
\[
X_\nr (M) \cong \prod\nolimits_{i=1}^\ell \prod\nolimits_{j=1}^{\ell_i} (\C^\times )^{e_{i,j}},
\quad W(M,\mf d) \cong \prod\nolimits_{i=1}^\ell \prod\nolimits_{j=1}^{\ell_i} S_{e_{i,j}} .
\]
It is known from \cite[Th\'eor\`eme B.2.d]{DKV} that $I_P^G$ preserves irreducibility for 
tempered representations. Hence the intertwining operators by which $W(M,\mf d)$ acts on 
$C^\infty (X_\unr (M)) \otimes \End_\C (I_P^G (\delta))$ must be scalar at every point of 
$X_\unr (M)$. In particular $W(L,\mf s)$ acts on $C^\infty (X_\unr (L)) \otimes 
\End_\C \big( I_{P_0 L}^G (\sigma)^K \big)$ as a group, not via a projective representation, 
so the 2-cocycle $\natural_{\mf s}$ is trivial.

Choose $K$ so that $\mc S (G)^{\mf s}$ is Morita equivalent to $\mc S (G,K)^{\mf s}$. 
The Plancherel isomorphism (Theorem \ref{thm:3.1}) provides an isomorphism of Fr\'echet algebras
\begin{equation}\label{eq:2.10}
\mc S (G,K)^{\mf d} \cong C^\infty (X_\unr (M))^{W(M,\mf d)} \otimes 
\End_\C \big( I_P^G (\delta)^K \big) .
\end{equation}
In particular $\mc S (G,K)^{\mf d}$ is Morita equivalent to $C^\infty (X_\unr (M))^{W(M,\mf d)}$,
an algebra whose irreducible representations are naturally parametrized by $X_\unr (M) / W(M,\mf d)$.
(An isomorphism like \eqref{eq:2.10} is rather specific for $GL_n (F)$, most other reductive $p$-adic 
groups have Bernstein components in which that fails.) From the above Morita equivalences and the 
decomposition \eqref{eq:2} of $\mc S (G)^{\mf s}$ we deduce that
\begin{equation}\label{eq:2.1}
\Irr^t (G)^{\mf s} \cong \Irr (\mc S (G,K)^{\mf s}) \cong 
\bigsqcup\nolimits_{\mf d \in \Delta_G^{\mf s}} X_\unr (M) / W(M,\mf d) .
\end{equation}
This space is naturally homeomorphic with
\begin{equation}\label{eq:2.8}
\big( \bigsqcup\nolimits_{w \in W(L,\mf s)} X_\unr (L)^w \big) / W(L,\mf s) ,
\end{equation}
via \cite[Theorem 1]{BrPl2} or Theorem \ref{thm:4.1}. That makes it clear how to choose our 
smooth families of representations $\mf F (M_i,\eta_i)$, we just take the $\mf F (M,\delta)$ 
with $[M,\delta] \in \Delta_G^{\mf s}$. Theorem \ref{thm:6.4} provides a natural 
isomorphism of Fr\'echet spaces
\begin{equation}\label{eq:2.3}
HH_n (\mc S (G,K)^{\mf s}) \longrightarrow 
\Big( \bigoplus\nolimits_{w \in W(L,\mf s)} \Omega^n_{sm} (X_\unr (L)^w) \Big)^{W(L,\mf s)} .
\end{equation}
By the homeomorphism between \eqref{eq:2.1} and \eqref{eq:2.8}, this implies that the map
\begin{equation}\label{eq:2.9}
\bigoplus\nolimits_{\mf d \in \Delta_G^{\mf s}} HH_n (\mc F^t_{M,\delta}) :  
HH_n (\mc S (G,K)^{\mf s}) \to
\bigoplus\nolimits_{\mf d \in \Delta_G^{\mf s}} \Omega^n_{sm} (X_\unr (M))^{W(M,\mf d)} 
\end{equation}
is bijective. In fact \eqref{eq:2.9} is an isomorphism of Fr\'echet 
$Z(\mc S (G,K)^{\mf s})$-modules, with respect to the module structure from Lemma \ref{lem:6.8}.
Then Theorem \ref{thm:6.1} shows that \eqref{eq:2.9} restricts to an isomorphism of
$Z(\mc S (G,K)^{\mf d})$-modules
\begin{equation}\label{eq:2.2}
HH_n (\mc F^t_{M,\delta}) : HH_n (\mc S (G,K)^{\mf d}) \to 
\Omega^n_{sm} (X_\unr (M))^{W(M,\mf d)} ,
\end{equation}
which was established before in \cite[p. 676]{BrPl}. 

The space $\Irr (G)^{\mf s}$ cannot decompose like in \eqref{eq:2.1} because it is connected, 
but still it comes close. Namely, by \cite[Theorem 1]{BrPl2} the homeomorphism \eqref{eq:2.1} 
extends naturally to a continuous bijection
\begin{equation}\label{eq:2.4}
\bigsqcup\nolimits_{\mf d \in \Delta_G^{\mf s}} X_\nr (M) / W(M,\mf d) \to \Irr (G)^{\mf s}. 
\end{equation}
Comparing Theorem \ref{thm:4.1} and \cite[Theorem 1]{BrPl2}, we see that \eqref{eq:2.4} can 
be obtained as $(\zeta^\vee)^{-1}$ followed by taking Langlands quotients of standard
$G$-representations. From Theorem \ref{thm:5.9} and \eqref{eq:2.3} we deduce that
the algebraic families $\mf F (M,\delta)$ induce $\C$-linear bijections 
\begin{multline}\label{eq:2.5}
HH_n (\mc H (G,K)^{\mf s}) \longrightarrow \bigoplus\nolimits_{\mf d \in \Delta_G^{\mf s}} 
\Omega^n (X_\nr (M))^{W(M,\mf d)} \\
\cong \Big( \bigoplus\nolimits_{w \in W(L,\mf s)} \Omega^n (X_\nr (L)^w) \Big)^{W(L,\mf s)} .
\end{multline}
By \eqref{eq:2.2}, \eqref{eq:2.5} restricts to an isomorphism of $Z(\mc H (G,K)^{\mf s})$-modules
\begin{equation*} 
HH_n (\mc F_{M,\delta}) : HH_n (\mc H (G,K)^{\mf s})^{\mf d} \to 
\Omega^n (X_\nr (M))^{W(M,\mf d)} .
\end{equation*}
Like in \eqref{eq:2.7} and \eqref{eq:2.6}, the corresponding direct summand of 
$HP_n (\mc H (G,K)^{\mf s})$ is canonically isomorphic with 
\begin{align*}
HP_n \big( \mc S (G,K)^{\mf d} \big) & \cong \bigoplus\nolimits_{m \in \Z} H^{n+2m}_{dR} 
(X_\unr (M))^{W(M,\mf d)} \\
& \cong \bigoplus\nolimits_{m \in \Z} H^{n+2m}_{\mr{sing}} (X_\unr (M) / W(M,\mf d)) .
\end{align*}

\end{document}